\documentclass[12pt]{article}
\usepackage{amsmath,amscd,amsbsy,amssymb,latexsym,url,bm,amsthm}
\usepackage[vlined,boxed,commentsnumbered,linesnumbered,ruled]{algorithm2e}
\usepackage{epsfig,graphicx,subfigure}
\usepackage{enumitem,balance,mathtools}
\usepackage{wrapfig}
\usepackage{mathrsfs, euscript}
\usepackage{xcolor}
\usepackage{hyperref}
\usepackage{epstopdf}
\usepackage{float}
\usepackage{comment}
\usepackage{abstract}
\usepackage[title]{appendix}

\newcommand{\R}{\mathbb R}

\newcommand{\N}{\mathbb N}

\author{}
\title{Well-Posedness of The Compressible Boundary Layer Equations with Data in the Gevrey Class}
\date{}
\begin{document}
	\maketitle
        \vspace{-2cm}
\newtheorem{theorem}{Theorem}[section]
\newtheorem{lemma}[theorem]{Lemma}
\newtheorem{proposition}[theorem]{Proposition}
\newtheorem{assumption}[theorem]{Assumption}
\newtheorem{corollary}[theorem]{Corollary}
\newtheorem{definition}[theorem]{Definition}
\newtheorem{remark}{Remark}[section]
\newtheorem{example}[theorem]{Example}
\newtheorem{exercise}{Exercise}
\newenvironment{solution}{\begin{proof}[Solution]}{\end{proof}}
	\setlength{\lineskiplimit}{2.625bp}
	\setlength{\lineskip}{2.625bp}
	\numberwithin{equation}{section}
	\newenvironment{partlist}[1][]
	{\begin{enumerate}[itemsep=0pt, label=(\arabic*), wide, labelindent=\parindent, listparindent=\parindent, #1]}
		{\end{enumerate}}
	\setcounter{page}{1}

\begin{center}

Ya-Guang Wang\footnote{email address: ygwang@sjtu.edu.cn}\\
School of Mathematical Sciences, MOE-LSC and SHL-MAC, Shanghai Jiao Tong University, 200240 Shanghai, China\\[3mm]
Yi-Lei  Zhao\footnote{corresponding author, email address:  zhaoyilei@sjtu.edu.cn}\\
School of Mathematical Sciences, Shanghai Jiao Tong University, 200240 Shanghai, China

\end{center}

\vspace{.1in}

\begin{abstract}
    This paper is devoted to the study of the compressible boundary layer equations in the Gevrey-2 solution space. Compared to the classical Prandtl equation, the additional complexity arises from the strong interaction between viscous layer and thermal layer. By introducing new auxiliary functions and observing the cancellation mechanism to overcome the loss of derivatives, we show the local existence and uniqueness of the solution in the Gevrey-2 space in the tangential variable and Sobolev regularity in the normal variable by using a direct energy method.
\end{abstract}

\tableofcontents

\section{Introduction}
In this paper, we study the following initial-boundary value problem for a non-linear coupled system of two degenerate parabolic equations and an elliptic equation in $\{(t,x,y)|t>0,x\in\R,y>0\}$:
\begin{equation}\label{eq:main}
    \left\{
    \begin{aligned}
        &\partial_t u+u\partial_x u+v\partial_y u=(\theta+\theta^E)\partial_y^2 u,\\
        &\partial_t \theta+u\partial_x \theta+v\partial_y \theta=(\theta+\theta^E)\partial_y^2 \theta+(\theta+\theta^E)(\partial_y u)^2,\\
        &\partial_x u+\partial_y v=\partial_y^2\theta+(\partial_y u)^2,\\
        &u|_{y=0}=v|_{y=0}=\partial_y\theta|_{y=0}=0,\quad \lim\limits_{y\to +\infty}u=\lim\limits_{y\to +\infty}\theta=0,\\
        &(u,\theta)|_{t=0}=(u_0,\theta_0)(x,y),
    \end{aligned}
    \right.
\end{equation}
where $u,v,\theta$ represent the tangential velocity, normal velocity and temperature of fluid respectively, with $\theta^E$ being a positive constant. This system describes the behavior of the viscous layer and thermal layer in the small viscosity and heat conductivity limit for a two-dimensional compressible non-isentropic viscous flow with a nonslip boundary condition for the velocity (\cite{liu2021study}).

To study the behavior and stability of viscous flow near a physical boundary is a classical and challenging problem. Since Prandtl \cite{prandtl1904uber} introduced the boundary layer theory for incompressible viscous flows with nonslip boundary conditions in 1904, there has been much inportant and interesting progress on its mathematical theory.  The well/ill-posedness of the classical two-dimensional Prandtl equation has been well studied in H\"{o}lder spaces or Sobolev spaces with or without monotonicity assumption of the initial tangential velocity in the normal variable, see \cite{oleinik1963}, \cite{oleinik1999mathematical}, \cite{alexandre2015well}, \cite{masmoudi2015local}, \cite{xin2004global}, \cite{xin2024global}, \cite{weinan1997blowup}, \cite{Fei2017}, \cite{gerard2010ill} for details. Meanwhile, without the monotonicity condition, many works have addressed the problem of the Prandtl equations in analytic or Gevrey class spaces due to the inherent degeneracy and derivative loss in the Prandtl equation. The well-posedness of the Prandtl equation in the frame of analytic solutions was studied in \cite{sammartino1998zero}, \cite{lombardo2003well}, \cite{zhang2016long}, \cite{ignatova2016almost}, \cite{paicu2021global}. With the non-degenerate critical point assumption, Gérard-Varet and Masmoudi \cite{gerard2015well} obtained the well-posedness result for a class of data in the Gevrey class of order $\frac{7}{4}$, which was subsequently improved to hold in the order $2$ space without any structural assumption by Dietert and Gérard-Varet \cite{dietert2019well}, Li and Yang \cite{li-yang}, and Li, Masmoudi and Yang \cite{li-mas-yang} in the three-dimensional problem respectively.

In numerous physical scenarios, such as aerodynamics, both of density and temperature play a significant role in the behavior of fluids. Accordingly, it is important to understand how variations in density and temperature influence the development of the boundary layer. However, the study on problems of the boundary layers in compressible flow remains very scarce. We refer to \cite{wang2012inviscid}, \cite{wang2015local}, \cite{gong2016boundary}, \cite{chen2024global} for studies on isentropic compressible flow. When the heat conduction is taken into account, the temperature near the boundary varies sharply due to the conductive heat transfer and viscous dissipation, thereby generating a thermal boundary layer in addition to the classical viscous layer. The stability of such boundary layers in two-dimensional circularly symmetric compressible non-isentropic flows was first studied by Liu and Wang \cite{liu-wang2014}. Recently, Liu, Wang and Yang \cite{liu2021study} derived the boundary layer problems in the small viscosity and  heat conductivity limit simultaneously for two-dimensional compressible non-isentropic flows, subject to the no-slip condition for velocity and a heat flux boundary condition for temperature, under various scales of viscosity and heat conductivity. They further obtained the local existence of classical solutions to the problem \eqref{eq:main} by using an energy method under the monotonicity assumption of the tangential velocity in the normal variable, when the viscosity and heat conductivity are of the same order, in which both of the viscous and thermal layers are strong. More recently, without the monotonicity assumption, the authors in \cite{wang2025compressible} proved the local well-posedness of the problem \eqref{eq:main} in the analytic setting.

The aim of this work is to extend this study to the local well-posedness of the problem \eqref{eq:main} in the Gevrey class of order 2 in the $x$-variable. 
The main result can be stated as the following one.

\begin{theorem}\label{thm:main}
    For fixed positive constants $\delta>0$ and $s>\frac{5}{2}$ , if the initial data $(u_0,\theta_0)$ satisfy
    \begin{equation}
        \begin{aligned}
            &e^{\delta\langle D_x\rangle^{\frac{1}{2}}}u_0\in H^{s+\frac{3}{2},0}_{\Psi_0},\quad e^{\delta\langle D_x\rangle^{\frac{1}{2}}}\theta_0\in H^{s+1,0}_{\Psi_0},\\
            &(e^{\delta\langle D_x\rangle^{\frac{1}{2}}}\partial_y^{j}u_0,~e^{\delta\langle D_x\rangle^{\frac{1}{2}}}\partial_y^{j}\theta_0)\in H^{s+1-j,0}_{\Psi_0}\quad \text{~for~} j=1,2,
        \end{aligned}
    \end{equation}
    and
    \begin{equation}\label{assumption2}
        \|e^{\delta\langle D_x\rangle^{\frac{1}{2}}}\partial_y\theta_0\|_{H^{s,0}_{\Psi_0}}\leq\epsilon
    \end{equation}
    for a fixed small constant $0<\epsilon<\theta^E$, with $\Psi_0(y)=\frac{y^2}{16\theta^E}$, $H^{s,k}_\Psi$ being the weighted Sobolev space defined in Definition \ref{def:norm1}, and $\langle D_x\rangle^\frac{1}{2}$ denoting the Fourier multiplier in the $x$-variable with symbol $(1+\xi^2)^{\frac{1}{4}}$, then there exists a positive time $T>0$, such that the problem \eqref{eq:main} has a unique solution $(u,\theta)$ on $[0,T]$ satisfying
    \begin{equation}
        (e^{\Phi(t,D)}\partial_y^ju,~e^{\Phi(t,D)}\partial_y^j\theta)\in L^\infty_T(H^{s+1-j,0}_{\Psi})
    \end{equation}
    and
    \begin{equation}
        (e^{\Phi(t,D)}\partial_y^{j+1}u,~e^{\Phi(t,D)}\partial_y^{j+1}\theta)\in L^2_T(H^{s+1-j,0}_{\Psi}),
    \end{equation}
    for $j=0,1,2$, where $L_T^\infty(H^{s,k}_\Psi)$ and $L_T^2(H^{s,k}_\Psi)$ are the weighted Sobolev spaces defined in Definition \ref{def:norm2} with $\Psi(t,y)=\frac{y^2}{16\theta^E(1+t)}$, and $e^{\Phi(t,D)}$ is the Fourier integral operator in the $x-$variable with the symbol $\Phi(t,\xi)=(\delta-\gamma\mu(t))\langle\xi\rangle^\frac{1}{2}$ being given in \eqref{eq:Phi}.
\end{theorem}

\begin{remark}
It is hard to extend the approach used in \cite{wang2025compressible} on the local well-posedness of the problem \eqref{eq:main} in the analytic setting to treat this problem in the Gevrey class $G^2$ directly. Indeed, by employing the energy method of \cite{wang2025compressible} for the problem \eqref{eq:main}, it gives that
$$
    \frac{1}{2}\frac{d}{dt}\|u_\Phi(t)\|_{H^{s,0}_\Psi}^2+\gamma\dot{\mu}(t)\|u_\Phi\|_{H^{s+\frac{1}{4},0}_\Psi}^2+\frac{1}{2}\theta^E\|\partial_yu_\Phi\|_{H^{s,0}_\Psi}^2\leq C\dot{\mu}(t)\|u_\Phi\|_{H^{s+\frac{1}{2},0}_\Psi}^2+\cdots,
$$
with $u_\Phi=\mathcal{F}^{-1}(e^{(\delta-\gamma\mu(t))\langle\xi\rangle^\frac{1}{2}} \hat{u})$, from which one could not close the estimate of $u$ in the Gevrey class $G^2$. On the other hand, due to the additional nonlinearity and non-locality in the system \eqref{eq:main}, we note that mainly there are two challenging issues  in proving the  well-posedness of the problem \eqref{eq:main} in the Gevrey class in contrast with the problem of the classical Prandtl equation when applying the existing techniques:
\begin{enumerate}[label=(\Roman*)]
    \item Because in \eqref{eq:main} the velocity is not divergence-free, it precludes the existence of nontrivial shear flow. This presents a critical obstacle to construct the required auxiliary functions, which are essential in performing the analysis of the linearized system around such a shear flow, a key point in the energy method utilized for the classical Prandtl equation, as given in \cite{gerard2015well}.
    \item The system \eqref{eq:main} exhibits strong, nonlinear coupling between the velocity and temperature equations, driven by both the divergence condition and nonlinear diffusion terms. Hence for the system to retain its parabolicity in $y$, one must ensure a small upper bound for $\theta$. Unfortunately, the cancellation mechanism successfully employed in \cite{li-mas-yang} fails to provide an explicit energy bound for $\theta$ without introducing supplementary auxiliary functions, which thereby significantly complicates closing the a priori estimates in the Gevrey class.
\end{enumerate}
\end{remark}
The difficulty (I) is somewhat arisen from the structure of the system \eqref{eq:main}. Nevertheless, one may solve this technical problem by introducing a cancellation mechanism to overcome the loss of derivatives rather than analysing the linearized problem. However, the difficulty (II) is the major one for the boundary layer equations \eqref{eq:main}. Inspired by \cite{li-mas-yang}, we define the auxiliary function $\mathcal{U}$ by
\begin{equation}\label{eq:intmathcalU}
\begin{cases}
    \mathcal{L}\left(\int_0^y\mathcal{U}(t,x,\tilde{y})d\tilde{y}\right)=-\partial_x v(t,x,y),\\
    \mathcal{U}|_{t=0}=0,\quad \partial_y\mathcal{U}|_{y=0}=\lim\limits_{y\to\infty}\mathcal{U}=0
\end{cases}
\end{equation}
with
\begin{equation}   \mathcal{L}\triangleq\partial_t+T_u\partial_x+T_v\partial_y-T_{(\theta+\theta^E)}\partial_y^2
\end{equation}
being the Prandtl-type operator, and $\lambda$ by 
\begin{equation}\label{def:lambda}
    \lambda\triangleq\partial_x u-T_{\partial_y u}\int_0^y\mathcal{U}d\tilde{y}.
\end{equation}
Moreover, to treat $\theta$ at the same order of regularity as $u$, we introduce the following auxiliary function $\varphi$:
\begin{equation}\label{def:varphi}
    \varphi\triangleq\partial_x\theta-T_{\partial_y \theta}\int_0^y\mathcal{U}d\tilde{y},
\end{equation}
which allows us to exploit a cancellation mechanism for both the terms $v\partial_yu$ and $v\partial_y\theta$ in \eqref{eq:main}. By a direct computation, we can deduce the following equations of $\mathcal{U},\lambda$ and $\varphi$ from \eqref{eq:main}:
\begin{equation}\label{eq:newsystem}
    \begin{cases}
        \mathcal{L}(\mathcal{U})=\partial_x\lambda+l.o.t,\\
        \mathcal{L}(\lambda)=l.o.t,\\
        \mathcal{L}(\varphi)=l.o.t,
    \end{cases}
\end{equation}
where $l.o.t$ denotes lower-order terms. If one takes the $H^{s,0}_\Psi$ inner product of the first equation of \eqref{eq:newsystem} with $\mathcal{U}$, one could view that $\partial_x\lambda$ behaves like $\partial_x^\frac{1}{2}\mathcal{U}$, thus, in Propositions \ref{prop:lambda}, we shall take the $H^{s+\frac{1}{2},0}_\Psi$ inner product of the second equation of \eqref{eq:newsystem} with $\lambda$, to close their estimates. Instead of using the abstract Cauchy-Kowalewski theorem as in \cite{li-mas-yang}, we shall derive their energy estimates directly and recover the estimates of $u$ and $\theta$ by \eqref{def:lambda} and \eqref{def:varphi}.

The remainder of this paper is organized as follows. In Section 2, we recall some basic facts of the Littlewood-Paley decomposition and define the function spaces which will be used later. Then, we prove a series of a priori estimates in Section 3. Finally, we give the proof of Theorem \ref{thm:main} in Section 4. 

Before the end of this introduction, we give two notations to be used later. For $a\lesssim b$, we mean there is a positive constant $C$, which may be different from line to line but independent of $\epsilon$, such that $a\leq Cb$ holds, and  $\langle f,g\rangle_{H^{s,0}_\Psi}\triangleq\text{Re}\int_{\R_+}\!\!\int_{\R}e^{2\Psi}\langle D_x\rangle^s f\overline{\langle D_x\rangle^s g}dxdy$ stands for the inner product of $f$ and $g$ in $H^{s,0}_\Psi(\R\times\R_+)$.

\section{Preliminaries}

\subsection{Littlewood-Paley decomposition}

We shall frequently use the Littlewood-Paley decomposition in the tangential variable $x$ for the product of functions. First, recall some notations from \cite{bahouri2011fourier}. Let 
 $\chi(\xi)$, $\varphi(\xi)$ be two fixed smooth functions such that
$$\begin{aligned}
    &\mathrm{supp}\ \varphi\subset\left\{\xi\in\mathbb{R}\big|~\frac{3}{4}\leqslant|\xi|\leqslant\frac{8}{3}\right\}&\mathrm{and}&\quad\sum_{k\in\mathbb{Z}}\varphi(2^{-k}\xi)=1\quad  (\forall\xi>0),\\
    &\mathrm{supp}\ \chi\subset\left\{\xi\in\mathbb{R}\big|~|\xi|\leqslant\frac{4}{3}\right\}&\mathrm{and}&\quad\chi(\xi)+\sum_{k\geq0}\varphi(2^{-k}\xi)=1\quad (\forall\xi\ge 0).
\end{aligned}$$
For any given Schwartz distribution $a\in {\mathcal S}'(\R_+^2)$, denote by  
$$\begin{aligned}
&\Delta_ka=0\quad\mathrm{if}\quad k\leq -2,\quad\Delta_{-1}a=\mathcal{F}^{-1}(\chi(|\xi|)\hat{a}),\quad\Delta_ka=\mathcal{F}^{-1}(\varphi(2^{-k}|\xi|)\hat{a})\quad\mathrm{if}\quad k\geqslant 0,\\&\mathrm{and}\quad S_ka=\mathcal{F}^{-1}(\chi(2^{-k}|\xi|)\hat{a}), \end{aligned}$$
where $\hat{a}$ and $\mathcal{F}^{-1}(a)$ denote the partial Fourier transform and Fourier inverse transform of $a$ with respect to the $x$ variable, that is, $\hat{a}(\xi,y)=\mathcal{F}_{x\rightarrow\xi}(a)(\xi,y)$.
The following Bony's decomposition shall be used frequently,
\begin{equation}\label{eq:bony}
    fg=T_fg+T_gf+R(f,g),
\end{equation}
where 
$$T_fg=\sum_kS_{k-1}f\Delta_kg,\quad R(f,g)=\sum_k\sum\limits_{k'=k-1}^{k+1}\Delta_{k'}f\Delta_kg.$$

\subsection{Functional spaces}

Next, we introduce several Sobolev spaces, which are similar to those introduced in \cite{wang2024global}.

\begin{definition}\label{def:norm1}
    For $s\in\R,k\in\N$, we define the norms in anisotropic Sobolev spaces $H^{s,k}(\R\times\R_+)$ as
    \begin{equation}
        \|u\|_{H^{s,k}}\triangleq\sum\limits_{l=0}^{k}\left(\int_{\R_+}\!\!\int_{\R}\langle\xi\rangle^{2s}|\partial_y^l\hat{u}(\xi,y)|^2d\xi dy\right)^{\frac{1}{2}}
    \end{equation}
    and the norms in weighted anisotropic Sobolev spaces $H^{s,k}_\Psi(\R\times\R_+)$ as
    \begin{equation}
        \|u\|_{H^{s,k}_\Psi}\triangleq\sum\limits_{l=0}^{k}\left(\int_{\R_+}\!\!\int_{\R}e^{2\Psi(t,y)}\langle\xi\rangle^{2s}|\partial_y^l\hat{u}(\xi,y)|^2d\xi dy\right)^{\frac{1}{2}}
    \end{equation}
    for a given weighted function $\Psi(t,y)$.
\end{definition}

\begin{definition}\label{def:norm2}
    For fixed $p\in[1,\infty]$ and $t>0$, define $L^p_t(H^{s,k}_\Psi)$ the space of functions with the norm being finite for $1\leq p<\infty$,
    \begin{equation}
        \|u\|_{L^p_t(H^{s,k}_\Psi)}\triangleq\left(\int_0^t\|u(t',\cdot)\|_{H^{s,k}_\Psi}^p dt'\right)^{\frac{1}{p}}
    \end{equation}
    with the usual modification when $p=\infty$. Moreover, for a given non-negative function $f(t)\in L^1_{\rm loc}(\R_+)$, define
    \begin{equation}
        \|u\|_{L^p_{t,f}(H^{s,k}_\Psi)}\triangleq\left(\int_0^tf(t')\|u(t',\cdot)\|_{H^{s,k}_\Psi}^p dt'\right)^{\frac{1}{p}}.
    \end{equation}
    with the usual modification when $p=\infty$.
\end{definition}

In the following, we shall fix the weighted function $\Psi(t,y)$ as
\begin{equation}\label{eq:weight}
    \Psi(t,y)=\frac{y^2}{16\theta^E(1+t)},
\end{equation}
and $\Psi_0(y)\triangleq\Psi(0,y)=\frac{y^2}{16\theta^E}$. Obviously, $\Psi(t,y)$ satisfies
\begin{equation}\label{eq:weightpro}
    \partial_t\Psi+4\theta^E(\partial_y\Psi)^2=0.
\end{equation}

To control the terms consisting of $\partial_y\Psi$ resulting from integration by parts, we shall use the following lemma, which is a special case of \cite[Lemma 2.5]{wang2024global}:
\begin{lemma}\label{lem:Psi}
    Let $u(x,y)$ be a smooth function on $\R_+^2$ which decays to zero sufficiently fast as $y$ goes to $+\infty$. Then, one has 
    \begin{equation}\label{ineq:Psi}
        \int_{\mathbb{R}_+^2}|u(x,y)\partial_y\Psi|^2e^{2\Psi}dxdy\leq C\int_{\mathbb{R}_+^2}|\partial_y u(x,y)|^2e^{2\Psi}dxdy,
    \end{equation}
    where $C$ is a constant depending only on $\theta^E$.
\end{lemma}

Furthermore, we have that 
\begin{lemma}\label{lem:Poincare}
    Under the same assumption as given in Lemma \ref{lem:Psi}, for any $s\in\R$, one has
    \begin{equation}\label{ineq:Poincare}
        \|u(t)\|_{L_y^{\infty}(H^s_x)}\leq (2\pi\theta^E)^\frac{1}{4}(1+t)^{\frac{1}{4}}\|\partial_yu\|_{H^{s,0}_\Psi}.
    \end{equation}
\end{lemma}

\begin{proof}
    It follows from $u(t,x,y)=-\int_y^\infty\partial_yu(t,x,\tilde{y}) d\tilde{y}$ and Holder's inequality that
    \begin{equation}
        \|u(t,\cdot, y)\|_{H^s_x}= \|\int_y^\infty\partial_yu d\tilde{y}\|_{H^s_x}\leq \left(\int_y^\infty e^{-2\Psi}d\tilde{y}\right)^{\frac{1}{2}}\left(\int_y^\infty e^{2\Psi}\|\partial_yu\|_{H^s_x}^2 d\tilde{y}\right)^{\frac{1}{2}},
    \end{equation}
   which implies
    \begin{equation}
        \|u(t)\|_{L_y^{\infty}(H^s_x)}=\sup\limits_{y>0}\|u(t)\|_{H^s_x}\leq (2\pi\theta^E)^\frac{1}{4}(1+t)^{\frac{1}{4}}\|\partial_yu\|_{H^{s,0}_\Psi}.
    \end{equation}
\end{proof}

Denote by
\begin{equation}
u_\Phi(t,x,y)\triangleq\mathcal{F}^{-1}_{\xi\rightarrow x}(e^{\Phi(t,\xi)}\hat{u}(t,\xi,y)),
\end{equation}
where the phase function $\Phi$ is defined by
\begin{equation}\label{eq:Phi}
    \Phi(t,\xi)\triangleq(\delta-\gamma\mu(t))\langle\xi\rangle^\frac{1}{2},
\end{equation}
with $\langle\xi\rangle\triangleq(1+\xi^2)^{\frac{1}{2}}$ and $\gamma$ being a sufficiently large constant, to overcome the difficulty of loss of derivatives in \eqref{eq:main}, inspired by the idea given in \cite{zhang2016long}.

The key quantity $\mu(t)$ for describing the evolution of the Gevrey radius of $u,\theta$ is determined by

\begin{equation}\label{eq:mu}
    \left\{
    \begin{aligned}
        &\dot{\mu}=1+(1+t)^\frac{1}{4}(\|(\partial_yu_\Phi,\partial_y\theta_\Phi)\|_{H^{\frac{5}{2}+,0}_\Psi}+\|(\partial_y^2u_\Phi,\partial_y^2\theta_\Phi)\|_{H^{\frac{3}{2}+,1}_\Psi})\\
        &\quad\ \ +(1+t)^\frac{1}{2}(\|u_\Phi\|_{H^{\frac{5}{2}+,0}_\Psi}^2+\|\theta_\Phi\|_{H^{\frac{3}{2}+,0}_\Psi}^2+\|(\partial_yu_\Phi,\partial_y\theta_\Phi)\|_{H^{\frac{5}{2}+,0}_\Psi}^2
        +\|(\partial_y^2u_\Phi,\partial_y^2\theta_\Phi)\|_{H^{\frac{3}{2}+,0}_\Psi}^2)\\
        &\quad\ \ +\|\theta_\Phi\|_{H^{\frac{1}{2}+,0}_\Psi}^4+(1+t)\|(\partial_yu_\Phi,\partial_y\theta_\Phi)\|_{H^{\frac{3}{2}+,1}_\Psi}^4\\
        &\quad\ \ +(1+t)^\frac{1}{2}\|\partial_y\theta_\Phi\|_{H^{\frac{3}{2}+,0}_\Psi}(\|\partial_y^3u_\Phi\|_{H^{\frac{3}{2}+,0}_\Psi}+\|\partial_y^3\theta_\Phi\|_{H^{\frac{3}{2}+,0}_\Psi})\\
        &\quad\ \ +\|\partial_y^3u_\Phi\|_{H^{\frac{1}{2}+,0}_\Psi}(\|\partial_yu_\Phi\|_{H^{\frac{1}{2}+,0}_\Psi}+\|\partial_y^2u_\Phi\|_{H^{\frac{1}{2}+,0}_\Psi}+\|\theta_\Phi\|_{H^{\frac{1}{2}+,0}_\Psi})+\|\partial_y^3\theta_\Phi\|_{H^{\frac{1}{2}+,0}_\Psi}\|\partial_yu_\Phi\|_{H^{\frac{1}{2}+,0}_\Psi},\\
        &\mu|_{t=0}=0,
    \end{aligned}
\right.\end{equation}
where we denote $\sigma+$ to be a constant slightly larger than $\sigma$.

In what follows, we shall always assume that $T^*$ is determined by
\begin{equation}\label{def:T*}
    T^*\triangleq\sup\{t>0|~\mu( t)<\delta/\gamma\}.
\end{equation}
Obviously, when  $0<t<T^*$, there holds the following convex inequality
\begin{equation}
    \Phi(t,\xi)\leqslant\Phi(t,\xi-\eta)+\Phi(t,\eta),\quad \ \forall\xi,\eta\in\R.
\end{equation}

\subsection{Para-product related estimates}

Along the same line as the proof of the classical para-product laws given in \cite{bahouri2011fourier}, corresponding to the cases of $\Phi(t,\xi)=0$, we can derive the following estimates:

\begin{lemma}\label{lem:paraestimate}
    For any fixed $s\in\R$, and $\sigma>\frac{1}{2}$, we have
    \begin{equation}
        \begin{aligned}
        &\|(T_f g)_\Phi\|_{H^{s}_x}\leq C\|f_\Phi\|_{H^{\sigma}_x}\|g_\Phi\|_{H^{s}_x},\\
        &\|((T_f)^*g)_\Phi\|_{H^{s}_x}\leq C\|f_\Phi\|_{H^{\sigma}_x}\|g_\Phi\|_{H^{s}_x}.
    \end{aligned}
    \end{equation}
    If $s_1+s_2>s+\frac{1}{2}>0$, then there holds
    \begin{equation}
        \|(R(f,g))_\Phi\|_{H^{s}_x}\leq C\|f_\Phi\|_{H^{s_1}_x}\|g_\Phi\|_{H^{s_2}_x}.
    \end{equation}
\end{lemma}

\begin{lemma}\label{lem:commutator}
    For any fixed $s>0$ and $\sigma>\frac{3}{2}$, one has
    \begin{equation}
        \begin{aligned}
        &\|((T_a T_b-T_{ab})f)_\Phi\|_{H^{s}_x}\leq C\|a_\Phi\|_{H^{\sigma}_x}\|b_\Phi\|_{H^{\sigma}_x}\|f_\Phi\|_{H^{s-1}_x},\\
        &
        \|([\langle D_x\rangle^s;T_a]f)_\Phi\|_{L^2_x}\leq C\|a_\Phi\|_{H^{\sigma}_x}\|f_\Phi\|_{H^{s-1}_x},
        \\
        &\|((T_a-(T_a)^*)f)_\Phi\|_{H^{s}_x}\leq C\|a_\Phi\|_{H^{\sigma}_x}\|f_\Phi\|_{H^{s-1}_x},\\
        &\|([T_a;T_b]f)_\Phi\|_{H^{s}_x}\leq C\|a_\Phi\|_{H^{\sigma}_x}\|b_\Phi\|_{H^{\sigma}_x}\|f_\Phi\|_{H^{s-1}_x},
    \end{aligned}
    \end{equation}
    where $\langle D_x\rangle^s$ is the Fourier multiplier operator with symbol $(1+|\xi|^2)^{\frac{s}{2}}$.
\end{lemma}

Recall the following commutator estimate from \cite[Lemma 2.3]{wang2024global}:
\begin{lemma}\label{lem:commutator2}
    For any fixed $s\in\R$, and $\sigma>\frac{3}{2}$, we have
    \begin{equation}\label{ineq:commutator}
        \|(T_a\partial_x f)_\Phi-T_a\partial_xf_\Phi\|_{H^{s}_x}\leq C(\delta-\gamma\mu(t))\|a_\Phi\|_{H^{\sigma}_x}\|f_\Phi\|_{H^{s+\frac{1}{2}}_x}.
    \end{equation}
\end{lemma}

\section{A priori estimates}\label{sec:priori}
The main goal of this section is to prove the a priori estimates for the solution of the problem \eqref{eq:main}. Before proceeding to the proof, we make the following assumption which will be verified in Section 4 by the bootstrap argument:
\begin{assumption}\label{assumption}
    Suppose that the solution $(u,\theta)$ of \eqref{eq:main} satisfies the following boundedness: 
    \begin{equation}\label{ineq:assumption}
    \|\partial_y u_\Phi\|_{L_{T^\ast}^\infty(H^{s,0}_\Psi)} \leq M, \qquad
    \|\partial_y \theta_\Phi\|_{L_{T^\ast}^\infty(H^{s,0}_\Psi)} \leq \zeta \triangleq 2\epsilon
    \end{equation}
    where $s$ and $\epsilon$ are given in Theorem \ref{thm:main}, and $M$ is the constant defined as
    \begin{equation}\label{def:assumption}
        M \triangleq 2(\|e^{\delta\langle D_x\rangle^{\frac{1}{2}}}\partial_yu_0\|_{H^{s,0}_{\Psi_0}}+\|e^{\delta\langle D_x\rangle^{\frac{1}{2}}}\partial_y\theta_0\|_{H^{s,0}_{\Psi_0}})<\infty. 
    \end{equation}
\end{assumption}
It is worth to note that all the estimates of this section are derived under Assumption \ref{assumption}.
\subsection{Estimate on \texorpdfstring{$\mathcal{U}$}{partial x\textasciicircum m U}}
The proposal of this subsection is to establish the following  priori estimate for the function $\mathcal{U}$ defined in \eqref{eq:intmathcalU}. 

\begin{proposition}\label{prop:mathcalU}
    Suppose that $\mathcal{U}$ is a solution to \eqref{eq:intmathcalU} with norms appeared in \eqref{ineq:mathcalU}being finite. Then for any $0<t<T^*$ with $T^*$ being given in \eqref{def:T*} and $0<\eta<1$, there exists a constant $C_\eta$ such that
    \begin{equation}\label{ineq:mathcalU}
        \begin{aligned}
            &\|\mathcal{U}_\Phi(t)\|_{H^{s,0}_\Psi}^2 
            + 2(\gamma - C_\eta) \|\mathcal{U}_\Phi\|_{L^2_{t,\dot{\mu}}(H^{s+\frac{1}{4},0}_\Psi)}^2 
            + \big(\theta^E - C\zeta (1+t)^{1/4} - 2\eta\big) \|\partial_y \mathcal{U}_\Phi\|_{L^2_t(H^{s,0}_\Psi)}^2 \\
            &\qquad \leq\; 
            2\eta \|\partial_y u_\Phi\|_{L^2_t(H^{s+1,0}_\Psi)}^2 
            + C \|\lambda_\Phi\|_{L^2_{t,\dot{\mu}}(H^{s+\frac{3}{4},0}_\Psi)}^2 
            + \frac{5}{\theta^E} \|\partial_y \theta_\Phi\|_{L^2_t(H^{s+1,0}_\Psi)}^2,
        \end{aligned}
    \end{equation}
    with $\lambda$ being defined in \eqref{def:lambda}.
\end{proposition}
\noindent 
By applying $\partial_y$ on \eqref{eq:intmathcalU} it follows
\begin{equation}\label{eq:mathcalU}
    \begin{aligned}
    \mathcal{L}(\mathcal{U})=&-\partial_x\partial_yv-T_{\partial_y u}\partial_x\int_0^y\mathcal{U}d\tilde{y}-T_{\partial_y v}\mathcal{U}-T_{\partial_y\theta}\partial_y\mathcal{U}\\
    =&\partial_x^2u-T_{\partial_y u}\partial_x\int_0^y\mathcal{U}d\tilde{y}-2\partial_yu\partial_x\partial_yu-\partial_x\partial_y^2\theta-T_{\partial_y v}\mathcal{U}-T_{\partial_y\theta}\partial_y\mathcal{U}\\
    =&\partial_x\lambda+T_{\partial_x\partial_yu}\int_0^y\mathcal{U}d\tilde{y}-2\partial_yu\partial_x\partial_yu-\partial_x\partial_y^2\theta-T_{\partial_y v}\mathcal{U}-T_{\partial_y\theta}\partial_y\mathcal{U},
    \end{aligned}
\end{equation}
where we have used the third equation given in \eqref{eq:main}, and $\lambda$ is given in \eqref{def:lambda}. Moreover, we apply the operator $e^{\Phi(t,D_x)}$ to \eqref{eq:mathcalU} and take $H^{s,0}_\Psi$ inner product with $\mathcal{U}_\Phi$ to get
\begin{equation}\label{eq:innerofmathcalU}
    \begin{aligned}
        &\langle\partial_t\mathcal{U}_\Phi,\mathcal{U}_\Phi\rangle_{H^{s,0}_\Psi}+\gamma\langle\dot{\mu}(t)\langle D_x\rangle^\frac{1}{2}\mathcal{U}_\Phi,\mathcal{U}_\Phi\rangle_{H^{s,0}_\Psi}-\theta^E\langle\partial_y^2\mathcal{U}_\Phi,\mathcal{U}_\Phi\rangle_{H^{s,0}_\Psi}\\
        =&-\langle [T_u\partial_x\mathcal{U}]_\Phi,\mathcal{U}_\Phi\rangle_{H^{s,0}_\Psi}-\langle [T_v\partial_y\mathcal{U}]_\Phi,\mathcal{U}_\Phi\rangle_{H^{s,0}_\Psi}+\langle [T_\theta\partial_y^2\mathcal{U}]_\Phi,\mathcal{U}_\Phi\rangle_{H^{s,0}_\Psi}+\langle \partial_x\lambda_\Phi,\mathcal{U}_\Phi\rangle_{H^{s,0}_\Psi}\\
        &+\langle \left[T_{\partial_x\partial_yu}\int_0^y\mathcal{U}d\tilde{y}\right]_\Phi,\mathcal{U}_\Phi\rangle_{H^{s,0}_\Psi}-2\langle [\partial_yu\partial_x\partial_yu]_\Phi,\mathcal{U}_\Phi\rangle_{H^{s,0}_\Psi}-\langle \partial_x\partial_y^2\theta_\Phi,\mathcal{U}_\Phi\rangle_{H^{s,0}_\Psi}\\
        &-\langle [T_{\partial_yv}\mathcal{U}]_\Phi,\mathcal{U}_\Phi\rangle_{H^{s,0}_\Psi}-\langle [T_{\partial_y\theta}\partial_y\mathcal{U}]_\Phi,\mathcal{U}_\Phi\rangle_{H^{s,0}_\Psi}\\
        \triangleq&A_1+\cdots+A_9.
    \end{aligned}
\end{equation}

\begin{lemma}\label{lem:leftofmathcalU}
    For any $0<t<T^*$, the terms on the left-hand side of \eqref{eq:innerofmathcalU} satisfy the following estimate:
    \begin{equation}\label{ineq:leftofmathcalU}
        \begin{aligned}
            &\langle\partial_t\mathcal{U}_\Phi,\mathcal{U}_\Phi\rangle_{H^{s,0}_\Psi}+\gamma\langle\dot{\mu}(t)\langle D_x\rangle^\frac{1}{2}\mathcal{U}_\Phi,\mathcal{U}_\Phi\rangle_{H^{s,0}_\Psi}-\theta^E\langle\partial_y^2\mathcal{U}_\Phi,\mathcal{U}_\Phi\rangle_{H^{s,0}_\Psi}\\
            \geq& \frac{1}{2}\frac{d}{dt}\|\mathcal{U}_\Phi(t)\|_{H^{s,0}_\Psi}^2+\gamma\dot{\mu}(t)\|\mathcal{U}_\Phi\|_{H^{s+\frac{1}{4},0}_\Psi}^2+\frac{2}{3}\theta^E\|\partial_y\mathcal{U}_\Phi\|_{H^{s,0}_\Psi}^2+\theta^E\|\partial_y\Psi\mathcal{U}_\Phi\|_{H^{s,0}_\Psi}^2.
        \end{aligned}
    \end{equation}
\end{lemma}

\begin{proof}
    In view of \eqref{eq:weightpro}, we get, by using integration by parts and Young's inequality, that
    \begin{equation}
        \begin{aligned}
        &\langle\partial_t\mathcal{U}_\Phi-\theta^E\partial_y^2\mathcal{U}_\Phi,\mathcal{U}_\Phi\rangle_{H^{s,0}_\Psi}\\
        =&\frac{1}{2}\frac{d}{dt}\|\mathcal{U}_\Phi(t)\|_{H^{s,0}_\Psi}^2-\langle\partial_t\Psi \mathcal{U}_\Phi,\mathcal{U}_\Phi\rangle_{H^{s,0}_\Psi}+\theta^E\|\partial_y\mathcal{U}_\Phi\|_{H^{s,0}_\Psi}^2+2\theta^E\langle\partial_y\Psi \partial_y\mathcal{U}_\Phi,\mathcal{U}_\Phi\rangle_{H^{s,0}_\Psi}\\
        \geq&\frac{1}{2}\frac{d}{dt}\|\mathcal{U}_\Phi(t)\|_{H^{s,0}_\Psi}^2+(1-\frac{1}{\kappa})\theta^E\|\partial_y\mathcal{U}_\Phi\|_{H^{s,0}_\Psi}^2-\langle(\partial_t\Psi+\kappa\theta^E(\partial_y\Psi)^2)\mathcal{U}_\Phi,\mathcal{U}_\Phi\rangle_{H^{s,0}_\Psi}\\
        =&\frac{1}{2}\frac{d}{dt}\|\mathcal{U}_\Phi(t)\|_{H^{s,0}_\Psi}^2+(1-\frac{1}{\kappa})\theta^E\|\partial_y\mathcal{U}_\Phi\|_{H^{s,0}_\Psi}^2+(4-\kappa)\theta^E\|\partial_y\Psi\mathcal{U}_\Phi\|_{H^{s,0}_\Psi}^2,
        \end{aligned}
    \end{equation}
    for a positive constant $\kappa>0$.
    A direct computation shows that
    \begin{equation}
        \gamma\langle\dot{\mu}(t)\langle D_x\rangle^\frac{1}{2}\mathcal{U}_\Phi,\mathcal{U}_\Phi\rangle_{H^{s,0}_\Psi}=\gamma\dot{\mu}(t)\|\mathcal{U}_\Phi\|_{H^{s+\frac{1}{4},0}_\Psi}^2.
    \end{equation}
    By taking $\kappa=3$ and summing up the above estimates, we complete the proof of \eqref{ineq:leftofmathcalU}.
\end{proof}

\begin{lemma}\label{lem:rightofmathcalU}
    For any $0<t<T^*$, and $\eta>0$, there exists a constant $C_\eta>0$, such that the terms $A_1,\cdots,A_9$ given on the right-hand side of \eqref{eq:innerofmathcalU} satisfy the following estimate:
    \begin{equation}\label{3.9}
        \begin{aligned}
            \sum\limits_{i=1}^9|A_i|\leq& C_\eta\dot{\mu}(t)\|\mathcal{U}_\Phi\|_{H^{s+\frac{1}{4},0}_\Psi}^2+\eta(\|\partial_y\mathcal{U}_\Phi\|_{H^{s,0}_\Psi}^2+\|\partial_yu_\Phi\|_{H^{s+1,0}_\Psi}^2)+C\dot{\mu}(t)\|\lambda_\Phi\|_{H^{s+\frac{3}{4},0}_\Psi}^2\\
            &+\left(\frac{1}{6}\theta^E+C\zeta(1+t)^\frac{1}{4}\right)\|\partial_y\mathcal{U}_\Phi\|_{H^{s,0}_\Psi}^2+\frac{5}{2\theta^E}\|\partial_y\theta_\Phi\|_{H^{s+1,0}_\Psi}^2+\theta^E\|\partial_y\Psi\mathcal{U}_\Phi\|_{H^{s,0}_\Psi}^2.
        \end{aligned}
    \end{equation}
\end{lemma}

\begin{proof}
(1) \underline{Estimate of $A_1$}. We write
    \begin{equation}
        \begin{aligned}
            A_1=&-\langle [T_u\partial_x\mathcal{U}]_\Phi,\mathcal{U}_\Phi\rangle_{H^{s,0}_\Psi}\\
            =&-\langle [T_u\partial_x\mathcal{U}]_\Phi-T_u\partial_x\mathcal{U}_\Phi,\mathcal{U}_\Phi\rangle_{H^{s,0}_\Psi}-\langle T_u\partial_x\mathcal{U}_\Phi,\mathcal{U}_\Phi\rangle_{H^{s,0}_\Psi}\\
            \triangleq&A_{1,1}+A_{1,2}.
        \end{aligned}
    \end{equation}
    By using  Lemmas \ref{lem:commutator2} and \ref{lem:Poincare}, we get 
    \begin{equation}
        \begin{aligned}
        |A_{1,1}|&\leq \|[T_u\partial_x\mathcal{U}]_\Phi-T_u\partial_x\mathcal{U}_\Phi\|_{H^{s-\frac{1}{4},0}_\Psi}\|\mathcal{U}_\Phi\|_{H^{s+\frac{1}{4},0}_\Psi}\leq C\|u_\Phi\|_{L_y^\infty(H_x^{\frac{3}{2}+})}\|\mathcal{U}_\Phi\|_{H^{s+\frac{1}{4},0}_\Psi}^2\\
        &\leq C(1+t)^\frac{1}{4}\|\partial_yu_\Phi\|_{H^{\frac{3}{2}+,0}_\Psi}\|\mathcal{U}_\Phi\|_{H^{s+\frac{1}{4},0}_\Psi}^2.
        \end{aligned}    
    \end{equation}
    By  a standard commutator's argument, we write
    \begin{equation}
        \begin{aligned}
            A_{1,2}=&-\langle\langle D_x\rangle^s\mathcal{U}_\Phi,[\langle D_x\rangle^s;T_u]\partial_x\mathcal{U}_\Phi\rangle_{L_\Psi^2}+\frac{1}{2}\langle\langle D_x\rangle^s\mathcal{U}_\Phi,((T_u)^*-T_u)\langle D_x\rangle^s\partial_x\mathcal{U}_\Phi\rangle_{L_\Psi^2}\\
            &-\frac{1}{2}\langle\langle D_x\rangle^s\mathcal{U}_\Phi,[T_u;\partial_x]\langle D_x\rangle^s\mathcal{U}_\Phi\rangle_{L_\Psi^2}.
        \end{aligned}
    \end{equation}
    By using the commutator estimates given in Lemma \ref{lem:commutator}, and Lemma \ref{lem:Poincare}, we obtain
    \begin{equation}
        |A_{1,2}|\leq C(1+t)^\frac{1}{4}\|\partial_yu_\Phi\|_{H^{\frac{3}{2}+,0}_\Psi}\|\mathcal{U}_\Phi\|_{H^{s,0}_\Psi}^2.
    \end{equation}
    Summarizing the above estimates and noting \eqref{eq:mu}, we conclude that
    \begin{equation}
        |A_1|\leq C(1+t)^\frac{1}{4}\|\partial_yu_\Phi\|_{H^{\frac{3}{2}+,0}_\Psi}\|\mathcal{U}_\Phi\|_{H^{s+\frac{1}{4},0}_\Psi}^2\leq C\dot{\mu}(t)\|\mathcal{U}_\Phi\|_{H^{s+\frac{1}{4},0}_\Psi}^2.
    \end{equation}
    
    (2) \underline{Estimate of $A_2$}. By using the third equation given in  \eqref{eq:main}, we get that
    \begin{equation}\label{3.15}
        \begin{aligned}
            A_2=&-\langle [T_v\partial_y\mathcal{U}]_\Phi,\mathcal{U}_\Phi\rangle_{H^{s,0}_\Psi}\\
            =&\langle [T_{\int_0^y\partial_xud\tilde{y}}\partial_y\mathcal{U}]_\Phi,\mathcal{U}_\Phi\rangle_{H^{s,0}_\Psi}-\langle [T_{\partial_y\theta}\partial_y\mathcal{U}]_\Phi,\mathcal{U}_\Phi\rangle_{H^{s,0}_\Psi}-\langle [T_{\int_{0}^{y}(\partial_yu)^2d\tilde{y}}\partial_y\mathcal{U}]_\Phi,\mathcal{U}_\Phi\rangle_{H^{s,0}_\Psi}\\
            \triangleq&A_{2,1}+A_{2,2}+A_{2,3}.
        \end{aligned}
    \end{equation}
    It follows by using Lemmas \ref{lem:paraestimate} and \ref{lem:Poincare} that
    \begin{equation}\label{3.16}
        \begin{aligned}
        |A_{2,1}|\leq& C\|\int_0^y\partial_xu_\Phi d\tilde{y}\|_{L_y^\infty(H_x^{\frac{1}{2}+})}\|\partial_y\mathcal{U}_\Phi\|_{H^{s,0}_\Psi}\|\mathcal{U}_\Phi\|_{H^{s,0}_\Psi}\\
        \leq& C(1+t)^\frac{1}{4}\|u_\Phi\|_{H^{\frac{3}{2}+,0}_\Psi}\|\partial_y\mathcal{U}_\Phi\|_{H^{s,0}_\Psi}\|\mathcal{U}_\Phi\|_{H^{s,0}_\Psi}\\
        \leq& C_\eta(1+t)^\frac{1}{2}\|u_\Phi\|_{H^{\frac{3}{2}+,0}_\Psi}^2\|\mathcal{U}_\Phi\|_{H^{s,0}_\Psi}^2+\eta\|\partial_y\mathcal{U}_\Phi\|_{H^{s,0}_\Psi}^2.
        \end{aligned}
    \end{equation}
    Similarly, we have
    \begin{equation}\label{ineq:A22}
        |A_{2,2}|\leq C_\eta(1+t)^\frac{1}{2}\|\partial_y^2\theta_\Phi\|_{H^{\frac{1}{2}+,0}_\Psi}^2\|\mathcal{U}_\Phi\|_{H^{s,0}_\Psi}^2+\eta\|\partial_y\mathcal{U}_\Phi\|_{H^{s,0}_\Psi}^2.
    \end{equation}
    For $A_{2,3}$ given in \eqref{3.15}, thanks to Minkowski's  inequality, we have
    \begin{equation}\label{3.18}
        \begin{aligned}
        |A_{2,3}|\leq& C\|\int_0^y[(\partial_yu)^2]_\Phi d\tilde{y}\|_{L_y^\infty(H_x^{\frac{1}{2}+})}\|\mathcal{U}_\Phi\|_{H^{s,0}_\Psi}\|\partial_y\mathcal{U}_\Phi\|_{H^{s,0}_\Psi}\\
        \leq& C\int_0^\infty\|[(\partial_yu)^2]_\Phi\|_{H_x^{\frac{1}{2}+}}d\tilde{y}\|\mathcal{U}_\Phi\|_{H^{s,0}_\Psi}\|\partial_y\mathcal{U}_\Phi\|_{H^{s,0}_\Psi}.
        \end{aligned}
    \end{equation}
    By applying Bony's decomposition to $(\partial_yu)^2$, we deduce from Lemma \ref{lem:paraestimate} that
    \begin{equation}
        \|[(\partial_yu)^2]_\Phi\|_{H_x^{\frac{1}{2}+}}\leq 2\|[T_{\partial_yu}\partial_yu]_\Phi\|_{H_x^{\frac{1}{2}+}}+\|[R(\partial_yu,\partial_yu)]_\Phi\|_{H_x^{\frac{1}{2}+}}^2\leq C\|\partial_yu_\Phi\|_{H_x^{\frac{1}{2}+}}^2,
    \end{equation}
    thus from \eqref{3.18} we get that
    \begin{equation}\label{3.20}
        |A_{2,3}|\leq C_\eta\|\partial_yu_\Phi\|_{H^{\frac{1}{2}+,0}_\Psi}^4\|\mathcal{U}_\Phi\|_{H^{s,0}_\Psi}^2+\eta\|\partial_y\mathcal{U}_\Phi\|_{H^{s,0}_\Psi}^2.
    \end{equation}
    Summarizing the above estimates \eqref{3.16}-\eqref{3.20}, it follows that
    \begin{equation}
        |A_2|\leq C_\eta\dot{\mu}(t)\|\mathcal{U}_\Phi\|_{H^{s,0}_\Psi}^2+\eta\|\partial_y\mathcal{U}_\Phi\|_{H^{s,0}_\Psi}^2,
    \end{equation}
with $\mu(t)$ being defined in \eqref{eq:mu}.

     (3) \underline{Estimate of $A_3$}.
    By using integration by parts and the boundary condition $\partial_y\mathcal{U}|_{y=0}=0$, we have
    \begin{equation}
        \begin{aligned}
            A_3=&\langle [T_\theta\partial_y^2\mathcal{U}]_\Phi,\mathcal{U}_\Phi\rangle_{H^{s,0}_\Psi}\\
            =&-\langle [T_{\partial_y\theta}\partial_y\mathcal{U}]_\Phi,\mathcal{U}_\Phi\rangle_{H^{s,0}_\Psi}-\langle [T_\theta\partial_y\mathcal{U}]_\Phi,\partial_y\mathcal{U}_\Phi\rangle_{H^{s,0}_\Psi}-2\langle \partial_y\Psi[T_\theta\partial_y\mathcal{U}]_\Phi,\mathcal{U}_\Phi\rangle_{H^{s,0}_\Psi}\\
            \triangleq&A_{3,1}+A_{3,2}+A_{3,3}.
        \end{aligned}
    \end{equation}
    The term $A_{3,1}$ is the same as $A_{2.2}$ given in \eqref{3.15}, so it can be estimated by the right hand side of \eqref{ineq:A22}. 
    
    
By using Lemmas \ref{lem:paraestimate}, \ref{lem:Poincare}, and Assumption \ref{assumption}, it follows that     
\begin{equation}
        |A_{3,2}|\leq C\|\theta_\Phi\|_{L_y^\infty(H_x^{\frac{1}{2}+})}\|\partial_y\mathcal{U}_\Phi\|_{H^{s,0}_\Psi}^2\leq C(1+t)^\frac{1}{4}\|\partial_y\theta_\Phi\|_{H^{s,0}_\Psi}\|\partial_y\mathcal{U}_\Phi\|_{H^{s,0}_\Psi}^2\leq C\zeta(1+t)^\frac{1}{4}\|\partial_y\mathcal{U}_\Phi\|_{H^{s,0}_\Psi}^2.
    \end{equation}
On the other hand, by using Lemmas \ref{lem:paraestimate}, \ref{lem:Poincare} and \ref{lem:Psi}, one gets that
    \begin{equation}
        \begin{aligned}
        |A_{3,3}|\leq& C\|\partial_y\Psi\theta_\Phi\|_{L_y^\infty(H_x^{\frac{1}{2}+})}\|\partial_y\mathcal{U}_\Phi\|_{H^{s,0}_\Psi}\|\mathcal{U}_\Phi\|_{H^{s,0}_\Psi}\\
        \leq& C(1+t)^\frac{1}{4}(\|\partial_y\Psi\partial_y\theta_\Phi\|_{H^{\frac{1}{2}+,0}_\Psi}+(1+t)^{-1}\|\theta_\Phi\|_{H^{\frac{1}{2}+,0}_\Psi})\|\partial_y\mathcal{U}_\Phi\|_{H^{s,0}_\Psi}\|\mathcal{U}_\Phi\|_{H^{s,0}_\Psi}\\
        \leq& C_\eta\left((1+t)^\frac{1}{2}\|\partial_y^2\theta_\Phi\|_{H^{\frac{1}{2}+,0}_\Psi}^2+(1+t)^{-\frac{3}{2}}\|\theta_\Phi\|_{H^{\frac{1}{2}+,0}_\Psi}^2\right)\|\mathcal{U}_\Phi\|_{H^{s,0}_\Psi}^2+\eta\|\partial_y\mathcal{U}_\Phi\|_{H^{s,0}_\Psi}^2.
        \end{aligned}
    \end{equation}
   Summarizing the above estimates of $A_{3,1}, A_{3,2}$ and $A_{3,3}$, we conclude that
    \begin{equation}
        |A_3|\leq C_\eta\dot{\mu}(t)\|\mathcal{U}_\Phi\|_{H^{s,0}_\Psi}^2+(\eta+C\zeta(1+t)^\frac{1}{4})\|\partial_y\mathcal{U}_\Phi\|_{H^{s,0}_\Psi}^2.
    \end{equation}
   
     (4) \underline{Estimate of $A_4$}. For the fourth term given in \eqref{eq:innerofmathcalU}, it is easy to have by applying Holder's inequality that 
    \begin{equation}
        |A_4|\leq \|\lambda_\Phi\|_{H^{s+\frac{3}{4},0}_\Psi}\|\mathcal{U}_\Phi\|_{H^{s+\frac{1}{4},0}_\Psi}\leq C\dot{\mu}(t)(\|\lambda_\Phi\|_{H^{s+\frac{3}{4},0}_\Psi}^2+\|\mathcal{U}_\Phi\|_{H^{s+\frac{1}{4},0}_\Psi}^2).
    \end{equation}

     (5) \underline{Estimate of $A_5$}. For the fifth term given in \eqref{eq:innerofmathcalU},
     by using Lemmas \ref{lem:paraestimate} and \ref{lem:Poincare}, one has
    \begin{equation}
        \begin{aligned}
            |A_5|\leq& C\|\partial_yu_\Phi\|_{H^{\frac{3}{2}+,0}_\Psi}\|\int_0^y\mathcal{U}_\Phi d\tilde{y}\|_{L_y^\infty(H_x^s)}\|\mathcal{U}_\Phi\|_{H^{s,0}_\Psi}\\
            \leq& C(1+t)^\frac{1}{4}\|\partial_yu_\Phi\|_{H^{\frac{3}{2}+,0}_\Psi}\|\mathcal{U}_\Phi\|_{H^{s,0}_\Psi}^2\\
            \leq& C\dot{\mu}(t)\|\mathcal{U}_\Phi\|_{H^{s,0}_\Psi}^2.
        \end{aligned}
    \end{equation}
   
     (6) \underline{Estimate of $A_6$}. To estimate the sixth term given in \eqref{eq:innerofmathcalU}, by applying Bony's decomposition to $\partial_yu\partial_x\partial_yu$, and using Lemmas \ref{lem:paraestimate} and \ref{lem:Poincare}, we have
    \begin{equation}
        \begin{aligned}
            |A_6|\leq& |\langle[T_{\partial_yu}\partial_x\partial_yu]_\Phi,\mathcal{U}_\Phi\rangle_{H^{s,0}_\Psi}|+|\langle[T_{\partial_x\partial_yu}\partial_yu]_\Phi,\mathcal{U}_\Phi\rangle_{H^{s,0}_\Psi}|+|\langle[R(\partial_yu,\partial_x\partial_yu)]_\Phi,\mathcal{U}_\Phi\rangle_{H^{s,0}_\Psi}|\\
            \leq& C(1+t)^\frac{1}{4}\|\partial_y^2u_\Phi\|_{H^{\frac{1}{2}+,0}_\Psi}\|\partial_yu_\Phi\|_{H^{s+1,0}_\Psi}\|\mathcal{U}_\Phi\|_{H^{s,0}_\Psi}+C(1+t)^\frac{1}{4}\|\partial_y^2u_\Phi\|_{H^{\frac{3}{2}+,0}_\Psi}\|\partial_yu_\Phi\|_{H^{s,0}_\Psi}\|\mathcal{U}_\Phi\|_{H^{s,0}_\Psi}\\
            \leq& C_\eta(1+t)^\frac{1}{2}\|\partial_y^2u_\Phi\|_{H^{\frac{3}{2}+,0}_\Psi}^2\|\mathcal{U}_\Phi\|_{H^{s,0}_\Psi}^2+\eta\|\partial_yu_\Phi\|_{H^{s+1,0}_\Psi}^2\\
            \leq& C_\eta\dot{\mu}(t)\|\mathcal{U}_\Phi\|_{H^{s,0}_\Psi}^2+\eta\|\partial_yu_\Phi\|_{H^{s+1,0}_\Psi}^2.
        \end{aligned}
    \end{equation}
   
     (7) \underline{Estimate of $A_7$}.
    By using integration by parts and the boundary condition $\partial_y\theta|_{y=0}=0$, we have
    \begin{equation}
        \begin{aligned}
            A_7=&-\langle\partial_x\partial_y^2\theta_\Phi,\mathcal{U}_\Phi\rangle_{H^{s,0}_\Psi}\\
            =&\langle\partial_x\partial_y\theta_\Phi,\partial_y\mathcal{U}_\Phi\rangle_{H^{s,0}_\Psi}+2\langle\partial_y\Psi\partial_x\partial_y\theta_\Phi,\mathcal{U}_\Phi\rangle_{H^{s,0}_\Psi}\\
            \triangleq&A_{7,1}+A_{7,2}.
        \end{aligned}
    \end{equation}
    Applying Holder's inequality gives
    \begin{equation}
        |A_{7,1}|\leq \|\partial_y\theta_\Phi\|_{H^{s+1,0}_\Psi}\|\partial_y\mathcal{U}_\Phi\|_{H^{s,0}_\Psi}\leq \frac{1}{6}\theta^E\|\partial_y\mathcal{U}_\Phi\|_{H^{s,0}_\Psi}^2+\frac{3}{2\theta^E}\|\partial_y\theta_\Phi\|_{H^{s+1,0}_\Psi}^2.
    \end{equation}
    In the same way as above, we have
    \begin{equation}
        |A_{7,2}|\leq 2\|\partial_y\theta_\Phi\|_{H^{s+1,0}_\Psi}\|\partial_y\Psi\mathcal{U}_\Phi\|_{H^{s,0}_\Psi}\leq \theta^E\|\partial_y\Psi\mathcal{U}_\Phi\|_{H^{s,0}_\Psi}^2+\frac{1}{\theta^E}\|\partial_y\theta_\Phi\|_{H^{s+1,0}_\Psi}^2.
    \end{equation}
    Therefore, we conclude that
    \begin{equation}
        |A_7|\leq \frac{1}{6}\theta^E\|\partial_y\mathcal{U}_\Phi\|_{H^{s,0}_\Psi}^2+\frac{5}{2\theta^E}\|\partial_y\theta_\Phi\|_{H^{s+1,0}_\Psi}^2+\theta^E\|\partial_y\Psi\mathcal{U}_\Phi\|_{H^{s,0}_\Psi}^2.
    \end{equation}

     (8) \underline{Estimates of $A_8$ and $A_9$}.
     Similarly to the estimate of $A_2$, we have 
    \begin{equation}
        \begin{aligned}
            |A_8|=&|\langle T_{\partial_yv}\mathcal{U}_\Phi,\mathcal{U}_\Phi\rangle_{H^{s,0}_\Psi}|\\
            \leq& |\langle T_{\partial_xu}\mathcal{U}_\Phi,\mathcal{U}_\Phi\rangle_{H^{s,0}_\Psi}|+|\langle T_{\partial_y^2\theta}\mathcal{U}_\Phi,\mathcal{U}_\Phi\rangle_{H^{s,0}_\Psi}|+|\langle T_{(\partial_yu)^2}\mathcal{U}_\Phi,\mathcal{U}_\Phi\rangle_{H^{s,0}_\Psi}|\\
            \leq& C\left((1+t)^\frac{1}{4}\|\partial_yu_\Phi\|_{H^{\frac{3}{2}+,0}_\Psi}+(1+t)^\frac{1}{2}\|\partial_y^2\theta_\Phi\|_{H^{\frac{1}{2}+,0}_\Psi}+\|\partial_yu_\Phi\|_{H^{\frac{1}{2}+,0}_\Psi}\|\partial_y^2u_\Phi\|_{H^{\frac{1}{2}+,0}_\Psi}\right)\|\mathcal{U}_\Phi\|_{H^{s,0}_\Psi}^2\\
            \leq& C\dot{\mu}(t)\|\mathcal{U}_\Phi\|_{H^{s,0}_\Psi}^2
        \end{aligned}
    \end{equation}
    and
    \begin{equation}
        |A_9|\leq C_\eta(1+t)^\frac{1}{2}\|\partial_y^2\theta_\Phi\|_{H^{\frac{1}{2}+,0}_\Psi}^2\|\mathcal{U}_\Phi\|_{H^{s,0}_\Psi}^2+\eta\|\partial_y\mathcal{U}_\Phi\|_{H^{s,0}_\Psi}^2\leq C_\eta\dot{\mu}(t)\|\mathcal{U}_\Phi\|_{H^{s,0}_\Psi}^2+\eta\|\partial_y\mathcal{U}_\Phi\|_{H^{s,0}_\Psi}^2.
    \end{equation}
   
    By summing up the above estimates of $A_1,\cdots,A_9$, we obtain the inequality \eqref{3.9} given in this lemma.
\end{proof}
\begin{proof}[Proof of Proposition \ref{prop:mathcalU}]
    Noting that $\mathcal{U}|_{t=0}=0$, by inserting the estimates obtained in Lemmas \ref{lem:leftofmathcalU} and \ref{lem:rightofmathcalU} into the identity \eqref{eq:innerofmathcalU}, and integrating the resulting inequality over $[0,t]$, we obtain the estimate \eqref{ineq:mathcalU} immediately.
\end{proof}

\subsection{Estimate on \texorpdfstring{$\lambda$}{lambda}}
In this subsection, we study the a priori estimate of the function $\lambda$ defined in \eqref{def:lambda}.
\begin{proposition}\label{prop:lambda}
    Let $\lambda$ be the function defined in \eqref{def:lambda}. Then, for any $0<t<T^*$ and $0<\eta<1$, there is a constant $C_\eta$ depending on $\eta$ such that the following estimate holds when the right hand side is finite
    \begin{equation}\label{ineq:lambda}
        \begin{aligned}
            &\|\lambda_\Phi(t)\|_{H^{s+\frac{1}{2},0}_\Psi}^2 
            + 2(\gamma - C_\eta) \|\lambda_\Phi\|_{L_{t,\dot{\mu}}^2(H^{s+\frac{3}{4},0}_\Psi)}^2 
            + (\theta^E - C\zeta (1+t)^{\frac{1}{4}} - 2\eta) \|\partial_y\lambda_\Phi\|_{L_t^2(H^{s+\frac{1}{2},0}_\Psi)}^2 \\
            &\qquad \leq\; 
            \|\partial_xu_\Phi(0)\|_{H^{s+\frac{1}{2},0}_{\Psi_0}}^2 
            + C \|u_\Phi\|_{L_{t,\dot{\mu}}^2(H^{s+\frac{5}{4},0}_\Psi)}^2 
            + (2\eta + C\zeta (1+t)^{\frac{1}{4}}) \|\partial_yu_\Phi\|_{L_t^2(H^{s+\frac{1}{2},0}_\Psi)}^2 \\
            &\qquad\quad 
            + C \|\theta_\Phi\|_{L_{t,\dot{\mu}}^2(H^{s+\frac{5}{4},0}_\Psi)}^2 
            + 2\eta \|\partial_y\theta_\Phi\|_{L_t^2(H^{s+\frac{1}{2},0}_\Psi)}^2 
            + C \|\mathcal{U}_\Phi\|_{L_{t,\dot{\mu}}^2(H^{s+\frac{1}{4},0}_\Psi)}^2 
            + 2\eta \|\partial_y\mathcal{U}_\Phi\|_{L_t^2(H^{s-\frac{1}{2},0}_\Psi)}^2.
        \end{aligned}
    \end{equation}
\end{proposition}
\noindent To prove the above estimate we first derive the equation satisfied by $\lambda$. Acting the para-product operator $T_{\partial_y u}$ on \eqref{eq:intmathcalU}, we obtain
\begin{equation}\label{3.36}
    \begin{aligned}
        \mathcal{L}(T_{\partial_yu}\int_0^y\mathcal{U}d\tilde{y})=&T_{\partial_t\partial_yu}\int_0^y\mathcal{U}d\tilde{y}+T_uT_{\partial_x\partial_yu}\int_0^y\mathcal{U}d\tilde{y}+T_vT_{\partial_y^2u}\int_0^y\mathcal{U}d\tilde{y}-T_{(\theta+\theta^E)}T_{\partial^3_yu}\int_0^y\mathcal{U}d\tilde{y}\\
        &+T_{\partial_yu}\partial_t\int_0^y\mathcal{U}d\tilde{y}+T_uT_{\partial_yu}\partial_x\int_0^y\mathcal{U}d\tilde{y}+T_vT_{\partial_yu}\partial_y\int_0^y\mathcal{U}d\tilde{y}\\
        &-T_{(\theta+\theta^E)}T_{\partial_yu}\partial_y^2\int_0^y\mathcal{U}d\tilde{y}-2T_{(\theta+\theta^E)}T_{\partial_y^2u}\mathcal{U}\\
        =&T_{[\partial_t+u\partial_x+v\partial_y-(\theta+\theta^E)\partial_y^2]\partial_yu}\int_0^y\mathcal{U}d\tilde{y}\\
        &
        +[(T_uT_{\partial_x\partial_yu}-T_{u\partial_x\partial_yu})+(T_vT_{\partial_y^2u}-T_{v\partial_y^2u})-(T_\theta T_{\partial^3_yu}-T_{\theta\partial^3_yu})]\int_0^y\mathcal{U}d\tilde{y}
        \\
        &+T_{\partial_yu}(\partial_t+T_u\partial_x+T_v\partial_y-T_{\theta+\theta^E}\partial_y^2)\int_0^y\mathcal{U}d\tilde{y}\\
        &+[T_u;T_{\partial_yu}]\partial_x\int_0^y\mathcal{U}d\tilde{y}+[T_v;T_{\partial_yu}]\mathcal{U}-[T_\theta;T_{\partial_yu}]\partial_y\mathcal{U}-2T_{(\theta+\theta^E)}T_{\partial_y^2u}\mathcal{U}\\
        =&-T_{\partial_yu}\partial_xv+T_{(-\partial_yu\partial_xu-\partial_yv\partial_yu+\partial_y\theta\partial_y^2u)}\int_0^y\mathcal{U}d\tilde{y}\\
        &+[(T_uT_{\partial_x\partial_yu}-T_{u\partial_x\partial_yu})+(T_vT_{\partial_y^2u}-T_{v\partial_y^2u})-(T_\theta T_{\partial^3_yu}-T_{\theta\partial^3_yu})]\int_0^y\mathcal{U}d\tilde{y}\\
        &+[T_u;T_{\partial_yu}]\partial_x\int_0^y\mathcal{U}d\tilde{y}+[T_v;T_{\partial_yu}]\mathcal{U}-[T_\theta;T_{\partial_yu}]\partial_y\mathcal{U}-2T_{(\theta+\theta^E)}T_{\partial_y^2u}\mathcal{U}.
    \end{aligned}
\end{equation}
By applying $\partial_x$ to the first equation of \eqref{eq:main}, we obtain
\begin{equation}\label{3.37}
    \begin{aligned}
        \mathcal{L}(\partial_xu)=&-(\partial_xu)^2-\partial_xv\partial_yu+\partial_x\theta\partial_y^2u-T_{\partial_x^2u}u-R(\partial_x^2u,u)\\
        &-T_{\partial_x\partial_yu}v-R(\partial_x\partial_yu,v)-T_{\partial_x\partial_y^2u}\theta-R(\partial_x\partial_y^2u,\theta).
    \end{aligned}
\end{equation}
Noting that $\lambda=\partial_xu-T_{\partial_yu}\int_0^y\mathcal{U}d\tilde{y}$, by combining the above two equations we obtain that
\begin{equation}\label{eq:lambda}
    \begin{aligned}
        \mathcal{L}(\lambda)=&-(\partial_xu)^2-T_{\partial_xv}\partial_yu-R(\partial_xv,\partial_yu)+\partial_x\theta\partial_y^2u-T_{\partial_x^2u}u-R(\partial_x^2u,u)\\
        &-T_{\partial_x\partial_yu}v-R(\partial_x\partial_yu,v)-T_{\partial_x\partial_y^2u}\theta-R(\partial_x\partial_y^2u,\theta)-T_{(-\partial_yu\partial_xu-\partial_yv\partial_yu+\partial_y\theta\partial_y^2u)}\int_0^y\mathcal{U}d\tilde{y}\\
        &-[(T_uT_{\partial_x\partial_yu}-T_{u\partial_x\partial_yu})+(T_vT_{\partial_y^2u}-T_{v\partial_y^2u})-(T_\theta T_{\partial^3_yu}-T_{\theta\partial^3_yu})]\int_0^y\mathcal{U}d\tilde{y}\\
        &-[T_u;T_{\partial_yu}]\partial_x\int_0^y\mathcal{U}d\tilde{y}-[T_v;T_{\partial_yu}]\mathcal{U}+[T_\theta;T_{\partial_yu}]\partial_y\mathcal{U}+2T_{(\theta+\theta^E)}T_{\partial_y^2u}\mathcal{U}.
    \end{aligned}
\end{equation}
By applying the operator $e^{\Phi(t,D_x)}$ to \eqref{eq:lambda} and taking $H^{s+\frac{1}{2},0}_\Psi$ inner product with $\lambda_\Phi$, it follows

\begin{equation}\label{eq:inneroflambda}
    \begin{aligned}
        &\langle \partial_t\lambda_\Phi,\lambda_\Phi\rangle_{H^{s+\frac{1}{2},0}_\Psi}+\gamma\langle\dot{\mu}\langle D_x\rangle^\frac{1}{2}\lambda_\Phi,\lambda_\Phi\rangle_{H^{s+\frac{1}{2},0}_\Psi}-\theta^E\langle\partial_y^2\lambda_\Phi,\lambda_\Phi\rangle_{H^{s+\frac{1}{2},0}_\Psi}\\
        &=-\langle [T_u\partial_x\lambda]_\Phi,\lambda_\Phi\rangle_{H^{s+\frac{1}{2},0}_\Psi}-\langle [T_v\partial_y\lambda]_\Phi,\lambda_\Phi\rangle_{H^{s+\frac{1}{2},0}_\Psi}+\langle [T_\theta\partial_y^2\lambda]_\Phi,\lambda_\Phi\rangle_{H^{s+\frac{1}{2},0}_\Psi}+\langle[\mathcal{L}(\lambda)]_\Phi,\lambda_\Phi\rangle_{H^{s+\frac{1}{2},0}_\Psi}\\
        &\triangleq R.
    \end{aligned}
\end{equation}

For the left hand side of \eqref{eq:inneroflambda}, similar  to Lemma \ref{lem:leftofmathcalU}, one has

\begin{lemma}\label{lem:leftoflambda}
    For any $0<t<T^*$, the following estimate holds:
    \begin{equation}
        \begin{aligned}
            &\langle\partial_t\lambda_\Phi,\lambda_\Phi\rangle_{H^{s+\frac{1}{2},0}_\Psi}+\gamma\langle\dot{\mu}(t)\langle D_x\rangle^\frac{1}{2}\lambda_\Phi,\lambda_\Phi\rangle_{H^{s+\frac{1}{2},0}_\Psi}-\theta^E\langle\partial_y^2\lambda_\Phi,\lambda_\Phi\rangle_{H^{s+\frac{1}{2},0}_\Psi}\\
             & \geq \frac{1}{2}\frac{d}{dt}\|\lambda_\Phi(t)\|_{H^{s+\frac{1}{2},0}_\Psi}^2+\gamma\dot{\mu}(t)\|\lambda_\Phi\|_{H^{s+\frac{3}{4},0}_\Psi}^2+\frac{1}{2}\theta^E\|\partial_y\lambda_\Phi\|_{H^{s+\frac{1}{2},0}_\Psi}^2.\\
        \end{aligned}
    \end{equation}
\end{lemma}

For the terms given on the right hand side of \eqref{eq:inneroflambda}, we have the following lemma, whose proof shall be given  in Appendix \ref{app:A}.

\begin{lemma}\label{lem:rightoflambda}
    For any $0<t<T^*$ and $\eta>0$, the term $R$ given on the right-hand side of \eqref{eq:inneroflambda} satisfies the following estimate:
    \begin{equation}
        \begin{aligned}
           |R|\leq& C_\eta\dot{\mu}(t)\|\lambda_\Phi\|_{H^{s+\frac{3}{4},0}_\Psi}^2+(\eta+C\zeta(1+t)^\frac{1}{4})\|\partial_y\lambda_\Phi\|_{H^{s+\frac{1}{2},0}_\Psi}^2+C\dot{\mu}(t)\|u_\Phi\|_{H^{s+\frac{5}{4},0}_\Psi}^2\\
            &+(\eta+C\zeta(1+t)^\frac{1}{4})\|\partial_yu_\Phi\|_{H^{s+\frac{1}{2},0}_\Psi}^2+C\dot{\mu}(t)\|\theta_\Phi\|_{H^{s+\frac{5}{4},0}_\Psi}^2+\eta\|\partial_y\theta_\Phi\|_{H^{s+\frac{1}{2},0}_\Psi}^2\\
            &+C\dot{\mu}(t)\|\mathcal{U}_\Phi\|_{H^{s+\frac{1}{4},0}_\Psi}^2+\eta\|\partial_y\mathcal{U}_\Phi\|_{H^{s-\frac{1}{2},0}_\Psi}^2.
        \end{aligned}
    \end{equation}
\end{lemma}

\begin{proof}[Proof of Proposition \ref{prop:lambda}]
    Observing from \eqref{def:lambda} that $\lambda|_{t=0}=\partial_xu|_{t=0}$, by substituting the estimates given in Lemmas \ref{lem:leftoflambda} and \ref{lem:rightoflambda} into \eqref{eq:inneroflambda}, and integrating the resulting inequality over $[0,t]$, we deduce \eqref{ineq:lambda}.
\end{proof}

\subsection{Estimate on \texorpdfstring{$\varphi$}{phi}}
In this subsection, we shall derive the a priori estimate of $\varphi$ defined in \eqref{def:varphi}.
\begin{proposition}\label{prop:phi}
    Let $\varphi$ be the function defined in \eqref{def:varphi}. Then, for any $0<t<T^*$ and $0<\eta<1$, there is a constant $C_\eta$ depending on $\eta$ such that the following estimate holds when the right hand side is finite
    \begin{equation}\label{ineq:phi}
    \begin{aligned}
        &\|\varphi_\Phi(t)\|_{H^{s,0}_\Psi}^2
        + 2(\gamma - C_\eta)\|\varphi_\Phi\|_{L_{t,\dot{\mu}}^2(H^{s+\frac{1}{4},0}_\Psi)}^2
        + \big(\theta^E - C\zeta (1+t)^{\frac{1}{4}} - 2\eta\big) \|\partial_y \varphi_\Phi\|_{L_t^2(H^{s,0}_\Psi)}^2 \\
        &\qquad \leq\;
        \|\partial_x \theta_\Phi(0)\|_{H^{s,0}_{\Psi_0}}^2
        + C \|u_\Phi\|_{L_{t,\dot{\mu}}^2(H^{s+1,0}_\Psi)}^2
        + 2\eta \|\partial_y u_\Phi\|_{L_t^2(H^{s+1,0}_\Psi)}^2
        + C \|\theta_\Phi\|_{L_{t,\dot{\mu}}^2(H^{s+1,0}_\Psi)}^2 \\
        &\qquad\quad
        + \big(2\eta + C\zeta (1+t)^{\frac{1}{4}}\big) \|\partial_y \theta_\Phi\|_{L_t^2(H^{s+1,0}_\Psi)}^2
        + C \|\mathcal{U}_\Phi\|_{L_{t,\dot{\mu}}^2(H^{s,0}_\Psi)}^2
        + 2\eta \|\partial_y \mathcal{U}_\Phi\|_{L_t^2(H^{s-1,0}_\Psi)}^2.
    \end{aligned}
    \end{equation}
\end{proposition}

\noindent The proof of this proposition shall be along a way similar to that of Proposition \ref{prop:lambda}. First, as from 
\eqref{3.36} to \eqref{eq:lambda}, by acting the operators $T_{\partial_y\theta}$ and $\partial_x$ on \eqref{eq:intmathcalU} and the second equation of \eqref{eq:main} respectively, and using the equations of \eqref{eq:main}, one gets
\begin{equation}
    \begin{aligned}
        \mathcal{L}(T_{\partial_y\theta}\int_0^y\mathcal{U}d\tilde{y})
        =&-T_{\partial_y\theta}\partial_xv+T_{(-\partial_yu\partial_x\theta+\partial_xu\partial_y\theta+2(\theta+\theta^E)\partial_yu\partial_y^2u)}\int_0^y\mathcal{U}d\tilde{y}\\
        &+[(T_uT_{\partial_x\partial_y\theta}-T_{u\partial_x\partial_y\theta})+(T_vT_{\partial_y^2\theta}-T_{v\partial_y^2\theta})-(T_\theta T_{\partial^3_y\theta}-T_{\theta\partial^3_y\theta})]\int_0^y\mathcal{U}d\tilde{y}\\
        &+[T_u;T_{\partial_y\theta}]\partial_x\int_0^y\mathcal{U}d\tilde{y}+[T_v;T_{\partial_y\theta}]\mathcal{U}-[T_\theta;T_{\partial_y\theta}]\partial_y\mathcal{U}-2T_{(\theta+\theta^E)}T_{\partial_y^2\theta}\mathcal{U},
    \end{aligned}
\end{equation}
and
\begin{equation}
    \begin{aligned}
        \mathcal{L}(\partial_x\theta)=&-\partial_xu\partial_x\theta-\partial_xv\partial_y\theta+\partial_x\theta\partial_y^2\theta+\partial_x\theta(\partial_yu)^2+2(\theta+\theta^E)\partial_yu\partial_x\partial_yu\\
        &-T_{\partial_x^2\theta}u-R(\partial_x^2\theta,u)-T_{\partial_x\partial_y\theta}v-R(\partial_x\partial_y\theta,v)-T_{\partial_x\partial_y^2\theta}\theta-R(\partial_x\partial_y^2\theta,\theta).
    \end{aligned}
\end{equation}
Therefore,  $\varphi=\partial_x\theta-T_{\partial_y\theta}\int_0^y\mathcal{U}d\tilde{y}$ satisfies the following equation, 
\begin{equation}\label{eq:phi}
    \begin{aligned}
        \mathcal{L}(\varphi)=&-\partial_xu\partial_x\theta-T_{\partial_xv}\partial_y\theta-R(\partial_xv,\partial_y\theta)+\partial_x\theta\partial_y^2\theta+\partial_x\theta(\partial_yu)^2+2(\theta+\theta^E)\partial_yu\partial_x\partial_yu\\
        &-T_{\partial_x^2\theta}u-R(\partial_x^2\theta,u)-T_{\partial_x\partial_y\theta}v-R(\partial_x\partial_y\theta,v)-T_{\partial_x\partial_y^2\theta}\theta-R(\partial_x\partial_y^2\theta,\theta)\\
        &-T_{(-\partial_yu\partial_x\theta+\partial_xu\partial_y\theta+2(\theta+\theta^E)\partial_yu\partial_y^2u)}\int_0^y\mathcal{U}d\tilde{y}\\
        &-[(T_uT_{\partial_x\partial_y\theta}-T_{u\partial_x\partial_y\theta})+(T_vT_{\partial_y^2\theta}-T_{v\partial_y^2\theta})-(T_\theta T_{\partial^3_y\theta}-T_{\theta\partial^3_y\theta})]\int_0^y\mathcal{U}d\tilde{y}\\
        &-[T_u;T_{\partial_y \theta}]\partial_x \int_0^y \mathcal{U}d \tilde{y}-[T_v;T_{\partial_y \theta}]\mathcal{U}+[T_\theta;T_{\partial_y \theta}]\partial_y \mathcal{U}+2 T_{(\theta+\theta^E)} T_{\partial_y^2 \theta}\mathcal{U}.
    \end{aligned}
\end{equation}
By applying the operator $e^{\Phi(t,D_x)}$ to \eqref{eq:phi} and taking $H^{s,0}_\Psi$ inner product with $\varphi_\Phi$, we find
\begin{equation}\label{eq:innerofphi}
    \begin{aligned}
        &\langle\partial_t\varphi_\Phi,\varphi_\Phi\rangle_{H^{s,0}_\Psi}+\gamma\langle\dot{\mu}(t)\langle D_x\rangle^\frac{1}{2}\varphi_\Phi,\varphi_\Phi\rangle_{H^{s,0}_\Psi}-\theta^E\langle\partial_y^2\varphi_\Phi,\varphi_\Phi\rangle_{H^{s,0}_\Psi}\\
        =&-\langle[T_u\partial_x\varphi]_\Phi,\varphi_\Phi\rangle_{H^{s,0}_\Psi}-\langle[T_v\partial_y\varphi]_\Phi,\varphi_\Phi\rangle_{H^{s,0}_\Psi}+\langle[T_\theta\partial_y^2\varphi]_\Phi,\varphi_\Phi\rangle_{H^{s,0}_\Psi}-\langle[\partial_xu\partial_x\theta]_\Phi,\varphi_\Phi\rangle_{H^{s,0}_\Psi}\\
        &-\langle[T_{\partial_xv}\partial_y\theta]_\Phi,\varphi_\Phi\rangle_{H^{s,0}_\Psi}-\langle[R(\partial_xv,\partial_y\theta)]_\Phi,\varphi_\Phi\rangle_{H^{s,0}_\Psi}\\
        &+\langle[\partial_x\theta\partial_y^2\theta]_\Phi,\varphi_\Phi\rangle_{H^{s,0}_\Psi}+\langle[\partial_x\theta(\partial_yu)^2]_\Phi,\varphi_\Phi\rangle_{H^{s,0}_\Psi}+2\langle[(\theta+\theta^E)\partial_yu\partial_x\partial_yu]_\Phi,\varphi_\Phi\rangle_{H^{s,0}_\Psi}\\
        &-\langle[T_{\partial_x^2\theta}u]_\Phi,\varphi_\Phi\rangle_{H^{s,0}_\Psi}-\langle[R(\partial_x^2\theta,u)]_\Phi,\varphi_\Phi\rangle_{H^{s,0}_\Psi}-\langle[T_{\partial_x\partial_y\theta}v]_\Phi,\varphi_\Phi\rangle_{H^{s,0}_\Psi}-\langle[R(\partial_x\partial_y\theta,v)]_\Phi,\varphi_\Phi\rangle_{H^{s,0}_\Psi}\\
        &-\langle[T_{\partial_x\partial_y^2\theta}\theta]_\Phi,\varphi_\Phi\rangle_{H^{s,0}_\Psi}-\langle[R(\partial_x\partial_y^2\theta,\theta)]_\Phi,\varphi_\Phi\rangle_{H^{s,0}_\Psi}\\
        &-\langle\left[T_{(-\partial_yu\partial_x\theta+\partial_xu\partial_y\theta+2(\theta+\theta^E)\partial_yu\partial_y^2u)}\int_0^y\mathcal{U}d\tilde{y}\right]_\Phi,\varphi_\Phi\rangle_{H^{s,0}_\Psi}\\
        &-\langle\left[[(T_uT_{\partial_x\partial_y\theta}-T_{u\partial_x\partial_y\theta})+(T_vT_{\partial_y^2\theta}-T_{v\partial_y^2\theta})-(T_\theta T_{\partial^3_y\theta}-T_{\theta\partial^3_y\theta})]\int_0^y\mathcal{U}d\tilde{y}\right]_\Phi,\varphi_\Phi\rangle_{H^{s,0}_\Psi}\\
        &-\langle\left[[T_u;T_{\partial_y\theta}]\partial_x\int_0^y\mathcal{U}d\tilde{y}\right]_\Phi,\varphi_\Phi\rangle_{H^{s,0}_\Psi}-\langle[[T_v;T_{\partial_y\theta}]\mathcal{U}]_\Phi,\varphi_\Phi\rangle_{H^{s,0}_\Psi}+\langle[[T_\theta;T_{\partial_y\theta}]\partial_y\mathcal{U}]_\Phi,\varphi_\Phi\rangle_{H^{s,0}_\Psi}\\
        &+2\langle [T_{(\theta+\theta^E)}T_{\partial_y^2\theta}\mathcal{U}]_\Phi,\varphi_\Phi\rangle_{H^{s,0}_\Psi}\\
        \triangleq&C_1+\cdots+C_{21}.
    \end{aligned}
\end{equation}

For the left hand side of \eqref{eq:innerofphi}, similar  to Lemma \ref{lem:leftofmathcalU}, one has

\begin{lemma}\label{lem:leftofphi}
    For any $0<t<T^*$, the terms on the left-hand side of \eqref{eq:innerofphi} satisfy the following estimate:
    \begin{equation}
        \begin{aligned}
            &\langle\partial_t\varphi_\Phi,\varphi_\Phi\rangle_{H^{s,0}_\Psi}+\gamma\langle\dot{\mu}(t)\langle D_x\rangle^\frac{1}{2}\varphi_\Phi,\varphi_\Phi\rangle_{H^{s,0}_\Psi}-\theta^E\langle\partial_y^2\varphi_\Phi,\varphi_\Phi\rangle_{H^{s,0}_\Psi}\\
            &\geq \frac{1}{2}\frac{d}{dt}\|\varphi_\Phi\|_{H^{s,0}_\Psi}^2+\gamma\dot{\mu}(t)\|\varphi_\Phi\|_{H^{s+\frac{1}{4},0}_\Psi}^2+\frac{1}{2}\theta^E\|\partial_y\varphi_\Phi\|_{H^{s,0}_\Psi}^2.
        \end{aligned}
    \end{equation}
\end{lemma}

For the terms given on the right hand side of \eqref{eq:innerofphi}, we have the following result.

\begin{lemma}\label{lem:rightofphi}
    For any $0<t<T^*$ and $\eta>0$, the following estimate holds for the terms $C_1,\cdots,C_{21}$ given on the right-hand side of \eqref{eq:innerofphi}:
    \begin{equation}\label{ineq:rightofphi}
        \begin{aligned}
            \sum\limits_{i=1}^{21}|C_i|\leq& C_\eta\dot{\mu}(t)\left(\|\varphi_\Phi\|_{H^{s+\frac{1}{4},0}_\Psi}^2+\|u_\Phi\|_{H^{s+1,0}_\Psi}^2+\|\theta_\Phi\|_{H^{s+1,0}_\Psi}^2+\|\mathcal{U}_\Phi\|_{H^{s,0}_\Psi}^2\right)
            \\
            &
            +(\eta+C\zeta(1+t)^\frac{1}{4})\left(\|\partial_y\theta_\Phi\|_{H^{s+1,0}_\Psi}^2+\|\partial_y\varphi_\Phi\|_{H^{s,0}_\Psi}^2\right)
            \\
            &+
            \eta\left(\|\partial_yu_\Phi\|_{H^{s+1,0}_\Psi}^2
+\|\partial_y\mathcal{U}_\Phi\|_{H^{s-1,0}_\Psi}^2\right).
        \end{aligned}
    \end{equation}
\end{lemma}

\begin{proof}
(1) All terms given on the right hand side of \eqref{eq:innerofphi}, except that the terms $C_8$, $C_9$ and $C_{16}$, are of similar form of $\{B_j\}_{j=1}^{19}$ given in \eqref{A-1}, whose estimates shall be given in Appendix A, so as stated in Lemma \ref{lem:rightoflambda}, one has the estimate 
   \begin{equation}\label{3.49}
        \begin{aligned}
            \sum_{\substack{j=1\\ j\neq 8,9,16}}^{21}|C_j|\leq &  C_\eta\dot{\mu}(t)\left(\|\varphi_\Phi\|_{H^{s+\frac{1}{4},0}_\Psi}^2+\|u_\Phi\|_{H^{s+1,0}_\Psi}^2
+\|\theta_\Phi\|_{H^{s+1,0}_\Psi}^2+\|\mathcal{U}_\Phi\|_{H^{s,0}_\Psi}^2\right)\\
&+            (\eta+C\zeta(1+t)^\frac{1}{4})\left(\|\partial_y\varphi_\Phi\|_{H^{s,0}_\Psi}^2+\|\partial_y\theta_\Phi\|_{H^{s,0}_\Psi}^2\right)+\eta\left(\|\partial_yu_\Phi\|_{H^{s,0}_\Psi}^2+\|\partial_y\mathcal{U}_\Phi\|_{H^{s-1,0}_\Psi}^2\right)       
\end{aligned}
    \end{equation}
    for any fixed $\eta>0$. It remains to study the terms $C_8$, $C_9$ and $C_{16}$ respectively.

 (2) \underline{Estimate of $C_8$}. 
    Obviously, one has
    \begin{equation}
        \begin{aligned}
        C_8=&\langle [T_{\partial_x\theta}(\partial_yu)^2]_\Phi,\varphi_\Phi\rangle_{H^{s,0}_\Psi}+\langle [T_{(\partial_yu)^2}\partial_x\theta]_\Phi,\varphi_\Phi\rangle_{H^{s,0}_\Psi}+\langle [R(\partial_x\theta,(\partial_yu)^2)]_\Phi,\varphi_\Phi\rangle_{H^{s,0}_\Psi}\\
        \triangleq&C_{8,1}+C_{8,2}+C_{8,3}.
        \end{aligned}
    \end{equation}
    By applying Lemmas \ref{lem:paraestimate} and \ref{lem:Poincare}, we have
    \begin{equation}\label{3.51}
        \begin{aligned}
            |C_{8,1}|\leq& 2|\langle [T_{\partial_x\theta}T_{\partial_yu}\partial_yu]_\Phi,\varphi_\Phi\rangle_{H^{s,0}_\Psi}|+|\langle [T_{\partial_x\theta}R(\partial_yu,\partial_yu)]_\Phi,\varphi_\Phi\rangle_{H^{s,0}_\Psi}|\\
            \leq& C(1+t)^\frac{1}{2}\|\partial_y\theta_\Phi\|_{H^{\frac{3}{2}+,0}_\Psi}\|\partial_y^2u_\Phi\|_{H^{\frac{1}{2}+,0}_\Psi}\|\partial_yu_\Phi\|_{H^{s,0}_\Psi}\|\varphi_\Phi\|_{H^{s,0}_\Psi}\\
            \leq& C_\eta(1+t)\|\partial_y\theta_\Phi\|_{H^{\frac{3}{2}+,0}_\Psi}^2\|\partial_y^2u_\Phi\|_{H^{\frac{1}{2}+,0}_\Psi}^2\|\varphi_\Phi\|_{H^{s,0}_\Psi}^2+\eta\|\partial_yu_\Phi\|_{H^{s,0}_\Psi}^2\\
            \leq& C_\eta\dot{\mu}(t)\|\varphi_\Phi\|_{H^{s,0}_\Psi}^2+\eta\|\partial_yu_\Phi\|_{H^{s,0}_\Psi}^2.
        \end{aligned}
    \end{equation}
On the other hand, to estimate $C_{8,2}$ and $C_{8,3}$, first we have the following inequalities for a quadratic term,
    \begin{equation}\label{ineq:nonlinear}
        \begin{aligned}
            \|[f\cdot g]_\Phi\|_{H^{\sigma,0}_\Psi}\leq&C\|f_\Phi\|_{H^{\sigma,0}_\Psi}\|g_\Phi\|_{L_y^\infty(H_x^{\sigma})}\\
            \leq&C(1+t)^\frac{1}{4}\|f_\Phi\|_{H^{\sigma,0}_\Psi}\|\partial_y g_\Phi\|_{H^{\sigma,0}_\Psi},\quad \forall \sigma>\frac{1}{2},
        \end{aligned}
    \end{equation}
    by applying Bony's decomposition, and using Lemmas \ref{lem:paraestimate} and \ref{lem:Poincare}. Thus, by using Lemmas \ref{lem:paraestimate}, \ref{lem:Poincare} and the inequality \eqref{ineq:nonlinear}, we deduce that
    \begin{equation}\label{3.53}
        \begin{aligned}
            |C_{8,2}|+|C_{8,3}|\leq& C(1+t)^\frac{1}{4}\|[(\partial_yu)^2]_\Phi\|_{H^{\frac{1}{2}+,0}_\Psi}\|\partial_y\theta_\Phi\|_{H^{s+1,0}_\Psi}\|\varphi_\Phi\|_{H^{s,0}_\Psi}\\
            \leq& C(1+t)^\frac{1}{2}\|\partial_yu_\Phi\|_{H^{\frac{1}{2}+,0}_\Psi}\|\partial_y^2u_\Phi\|_{H^{\frac{1}{2}+,0}_\Psi}\|\partial_y\theta_\Phi\|_{H^{s+1,0}_\Psi}\|\varphi_\Phi\|_{H^{s,0}_\Psi}\\
            \leq& C_\eta\dot{\mu}(t)\|\varphi_\Phi\|_{H^{s,0}_\Psi}^2+\eta\|\partial_y\theta_\Phi\|_{H^{s+1,0}_\Psi}^2.
        \end{aligned}
    \end{equation}
    Therefore, by summing the  estimates \eqref{3.51} and \eqref{3.53} we conclude
    \begin{equation}\label{3.54}
        |C_8|\leq C_\eta\dot{\mu}(t)\|\varphi_\Phi\|_{H^{s,0}_\Psi}^2+\eta(\|\partial_yu_\Phi\|_{H^{s,0}_\Psi}^2+\|\partial_y\theta_\Phi\|_{H^{s+1,0}_\Psi}^2).
    \end{equation}

    (3) \underline{Estimate of $C_9$}.
    Obviously, one has
    \begin{equation}
        \begin{aligned}
            C_9=&2\langle [T_{\partial_x\partial_yu}((\theta+\theta^E)\partial_yu)]_\Phi,\varphi_\Phi\rangle_{H^{s,0}_\Psi}+2\langle [T_{(\theta+\theta^E)\partial_yu}\partial_x\partial_yu]_\Phi,\varphi_\Phi\rangle_{H^{s,0}_\Psi}\\
            &+2\langle [R((\theta+\theta^E)\partial_yu,\partial_x\partial_yu)]_\Phi,\varphi_\Phi\rangle_{H^{s,0}_\Psi}\\
            \triangleq&C_{9,1}+C_{9,2}+C_{9,3}.
        \end{aligned}
    \end{equation}
Similar to \eqref{3.51}, we have
    \begin{equation}
        \begin{aligned}
        |C_{9,1}|\leq& 2\theta^E|\langle [T_{\partial_x\partial_yu}\partial_yu]_\Phi,\varphi_\Phi\rangle_{H^{s,0}_\Psi}|+2\langle [T_{\partial_x\partial_yu}T_{\theta}\partial_yu]_\Phi,\varphi_\Phi\rangle_{H^{s,0}_\Psi}\\
        &+2\langle [T_{\partial_x\partial_yu}T_{\partial_yu}\theta]_\Phi,\varphi_\Phi\rangle_{H^{s,0}_\Psi}+2\langle [T_{\partial_x\partial_yu}R(\partial_yu,\theta)]_\Phi,\varphi_\Phi\rangle_{H^{s,0}_\Psi}\\
        \leq& C(1+t)^\frac{1}{4}\|\partial_y^2u_\Phi\|_{H^{\frac{3}{2}+,0}_\Psi}\|\partial_yu_\Phi\|_{H^{s,0}_\Psi}\|\varphi_\Phi\|_{H^{s,0}_\Psi}\\
        &+C(1+t)^\frac{1}{2}\|\partial_y^2u_\Phi\|_{H^{\frac{3}{2}+,0}_\Psi}\|\partial_y\theta_\Phi\|_{H^{\frac{1}{2}+,0}_\Psi}\|\partial_yu_\Phi\|_{H^{s,0}_\Psi}\|\varphi_\Phi\|_{H^{s,0}_\Psi}\\
        &+C(1+t)^\frac{1}{2}\|\partial_y^2u_\Phi\|_{H^{\frac{3}{2}+,0}_\Psi}\|\partial_yu_\Phi\|_{H^{\frac{1}{2}+,0}_\Psi}\|\partial_y\theta_\Phi\|_{H^{s,0}_\Psi}\|\varphi_\Phi\|_{H^{s,0}_\Psi}\\
        \leq& C_\eta(1+t)\left(\|\partial_y^2u_\Phi\|_{H^{\frac{3}{2}+,0}_\Psi}^2\|\partial_y\theta_\Phi\|_{H^{\frac{1}{2}+,0}_\Psi}^2+\|\partial_y^2u_\Phi\|_{H^{\frac{3}{2}+,0}_\Psi}^2\|\partial_yu_\Phi\|_{H^{\frac{1}{2}+,0}_\Psi}^2\right)\|\varphi_\Phi\|_{H^{s,0}_\Psi}^2\\
        &+C_\eta(1+t)^\frac{1}{2}\|\partial_y^2u_\Phi\|_{H^{\frac{3}{2}+,0}_\Psi}^2\|\varphi_\Phi\|_{H^{s,0}_\Psi}^2+\eta(\|\partial_yu_\Phi\|_{H^{s,0}_\Psi}^2+\|\partial_y\theta_\Phi\|_{H^{s,0}_\Psi}^2)\\
        \leq& C_\eta\dot{\mu}(t)\|\varphi_\Phi\|_{H^{s,0}_\Psi}^2+\eta(\|\partial_yu_\Phi\|_{H^{s,0}_\Psi}^2+\|\partial_y\theta_\Phi\|_{H^{s,0}_\Psi}^2).
        \end{aligned}
    \end{equation}
    By using Lemmas \ref{lem:paraestimate} and \ref{lem:Poincare}, one gets
    \begin{equation}
        \begin{aligned}
            &|C_{9,2}|+|C_{9,3}|\\
            \leq&C\left((1+t)^\frac{1}{4}\|\partial_y^2u_\Phi\|_{H^{\frac{1}{2}+,0}_\Psi}+\|[\theta\partial_yu]_\Phi\|_{L_y^\infty(H_x^{\frac{1}{2}+})}\right)\|\partial_yu_\Phi\|_{H^{s+1,0}_\Psi}\|\varphi_\Phi\|_{H^{s,0}_\Psi}\\
            \leq&C\left((1+t)^\frac{1}{4}\|\partial_y^2u_\Phi\|_{H^{\frac{1}{2}+,0}_\Psi}+\|\theta_\Phi\|_{L_y^\infty(H_x^{\frac{1}{2}+})}\|\partial_yu_\Phi\|_{L_y^\infty(H_x^{\frac{1}{2}+})}\right)\|\partial_yu_\Phi\|_{H^{s+1,0}_\Psi}\|\varphi_\Phi\|_{H^{s,0}_\Psi}\\
            \leq& C\left((1+t)^\frac{1}{4}\|\partial_y^2u_\Phi\|_{H^{\frac{1}{2}+,0}_\Psi}+(1+t)^\frac{1}{2}\|\partial_y\theta_\Phi\|_{H^{\frac{1}{2}+,0}_\Psi}\|\partial_y^2u_\Phi\|_{H^{\frac{1}{2}+,0}_\Psi}\right)\|\partial_yu_\Phi\|_{H^{s+1,0}_\Psi}\|\varphi_\Phi\|_{H^{s,0}_\Psi}\\
            \leq& C_\eta\left((1+t)^\frac{1}{2}\|\partial_y^2u_\Phi\|_{H^{\frac{1}{2}+,0}_\Psi}^2+(1+t)\|\partial_y\theta_\Phi\|_{H^{\frac{1}{2}+,0}_\Psi}^2\|\partial_y^2u_\Phi\|_{H^{\frac{1}{2}+,0}_\Psi}^2\right)\|\varphi_\Phi\|_{H^{s,0}_\Psi}^2+\eta\|\partial_yu_\Phi\|_{H^{s+1,0}_\Psi}^2\\
            \leq& C_\eta\dot{\mu}(t)\|\varphi_\Phi\|_{H^{s,0}_\Psi}^2+\eta\|\partial_yu_\Phi\|_{H^{s+1,0}_\Psi}^2.
        \end{aligned}
    \end{equation}
    By summing the above estimates, we obtain
    \begin{equation}\label{3.58}
        |C_9|\leq C_\eta\dot{\mu}(t)\|\varphi_\Phi\|_{H^{s,0}_\Psi}^2+\eta(\|\partial_yu_\Phi\|_{H^{s+1,0}_\Psi}^2+\|\partial_y\theta_\Phi\|_{H^{s,0}_\Psi}^2).
    \end{equation}

    (3) \underline{Estimate of $C_{16}$}. Recall that 
    \begin{equation}
        C_{16}=-\langle\left[T_{(-\partial_yu\partial_x\theta+\partial_xu\partial_y\theta+2(\theta+\theta^E)\partial_yu\partial_y^2u)}\int_0^y\mathcal{U}d\tilde{y}\right]_\Phi,\varphi_\Phi\rangle_{H^{s,0}_\Psi}.
    \end{equation}
    From Lemmas \ref{lem:paraestimate}, \ref{lem:Poincare} and the inequality \eqref{ineq:nonlinear}, we obtain that
    \begin{equation}\label{3.60}
        \begin{aligned}
            |C_{16}| \leq\; & C (1+t)^{\frac{1}{4}} \Big(
                \| [\partial_y u\, \partial_x \theta]_\Phi \|_{H^{\frac{1}{2}+,0}_\Psi}
                + \| [\partial_x u\, \partial_y \theta]_\Phi \|_{H^{\frac{1}{2}+,0}_\Psi} \\
            & \qquad\qquad
                + 2 \| [(\theta+\theta^E)\, \partial_y u\, \partial_y^2 u]_\Phi \|_{H^{\frac{1}{2}+,0}_\Psi}
            \Big)
            \|\mathcal{U}_\Phi\|_{H^{s,0}_\Psi} \|\varphi_\Phi\|_{H^{s,0}_\Psi} \\[1ex]
            \leq\; & C \Big(
                (1+t)^{\frac{1}{2}} \|\partial_y u_\Phi\|_{H^{\frac{3}{2}+,0}_\Psi} \|\partial_y \theta_\Phi\|_{H^{\frac{3}{2}+,0}_\Psi}
                + (1+t)^{\frac{1}{2}} \|\partial_y^2 u_\Phi\|_{H^{\frac{1}{2}+,0}_\Psi}^2 \\
            & \qquad\qquad
                + (1+t)^{\frac{3}{4}} \|\partial_y^2 u_\Phi\|_{H^{\frac{1}{2}+,0}_\Psi}^2 \|\partial_y \theta_\Phi\|_{H^{\frac{1}{2}+,0}_\Psi}
            \Big)
            \|\mathcal{U}_\Phi\|_{H^{s,0}_\Psi} \|\varphi_\Phi\|_{H^{s,0}_\Psi} \\[1ex]
            \leq\; & C \dot{\mu}(t) \|\mathcal{U}_\Phi\|_{H^{s,0}_\Psi} \|\varphi_\Phi\|_{H^{s,0}_\Psi} \\[1ex]
            \leq\; & C \dot{\mu}(t) \left( \|\mathcal{U}_\Phi\|_{H^{s,0}_\Psi}^2 + \|\varphi_\Phi\|_{H^{s,0}_\Psi}^2 \right).
        \end{aligned}
    \end{equation}
    
    Finally, by collecting the estimates \eqref{3.54}, \eqref{3.58} and \eqref{3.60} of $C_8, C_9$ and $C_{16}$ respectively, together with \eqref{3.49} of the other terms, we conclude the estimate \eqref{ineq:rightofphi}. This ends the proof of Lemma \ref{lem:rightofphi}.
\end{proof}
\begin{proof}[Proof of Proposition \ref{prop:phi}]
    In a way similar to the proof of Proposition \ref{prop:lambda}, by using $\varphi|_{t=0}=\partial_x\theta|_{t=0}$, and combining the estimates given in Lemmas \ref{lem:leftofphi} and \ref{lem:rightofphi}, from \eqref{eq:innerofphi} we deduce the estimate \eqref{ineq:phi} immediately.
\end{proof}

\subsection{Estimates on \texorpdfstring{$u,\theta$}{u,theta}}
In this subsection, we establish the a priori estimates of $u$ and $\theta$. First, we have
\begin{proposition}\label{prop:u}
    Let $(u,\theta)$ be a solution to \eqref{eq:main}. Then, for any $0<t<T^*$ and $0<\eta<1$, there is a constant $C_\eta$ depending on $\eta$ such that the following estimate holds when the right hand side is finite
    \begin{equation}\label{ineq:u}
        \begin{aligned}
            &\|u_\Phi(t)\|_{H^{s+1,0}_\Psi}^2 + 2(\gamma-C_\eta)\|u_\Phi\|_{L_{t,\dot{\mu}}^2(H^{s+\frac{5}{4},0}_\Psi)}^2 + (\theta^E-C\zeta(1+t)^\frac{1}{4}-2\eta) \|\partial_y u_\Phi\|_{L_t^2(H^{s+1,0}_\Psi)}^2 \\
            \leq& \|u_\Phi(0)\|_{H^{s+\frac{3}{2},0}_{\Psi_0}}^2+C\|u_\Phi\|_{L_{t,\dot{\mu}}^2(H^{s+\frac{5}{4},0}_\Psi)}^2 + (2\eta+C\zeta(1+t)^\frac{1}{4}) \|\partial_y u_\Phi\|_{L_t^2(H^{s+1,0}_\Psi)}^2 \\
            & + C\|\theta_\Phi\|_{L_{t,\dot{\mu}}^2(H^{s+1,0}_\Psi)}^2 + 4\eta\|\partial_y \theta_\Phi\|_{L_t^2(H^{s+1,0}_\Psi)}^2 + C\|\mathcal{U}_\Phi\|_{L_{t,\dot{\mu}}^2(H^{s+\frac{1}{4},0}_\Psi)}^2 + 2\eta\|\partial_y \mathcal{U}_\Phi\|_{L_t^2(H^{s-\frac{1}{4},0}_\Psi)}^2\\
            & + C(1+t)^\frac{1}{2}M^2 \Big(2\eta\|\partial_yu_\Phi\|_{L_t^2(H^{s+1,0}_\Psi)}^2 + C\|\lambda_\Phi\|_{L_{t,\dot{\mu}}^2(H^{s+\frac{3}{4},0}_\Psi)}^2 + \frac{5}{\theta^E}\|\partial_y\theta_\Phi\|_{L_t^2(H^{s+1,0}_\Psi)}^2\Big), 
        \end{aligned}
    \end{equation}
    where $M$ is defined in \eqref{def:assumption}.
\end{proposition}

From the definition \eqref{def:lambda} of $\lambda$, we have
\begin{equation}\label{def:lambda2}
    \partial_xu=\lambda+T_{\partial_yu}\int_0^y\mathcal{U}d\tilde{y},
\end{equation}
thus one can obtain the $H^{s,0}_\Psi$ estimate of $\partial_xu$ once the corresponding estimates of $\lambda$ and $\mathcal{U}$ are derived. Therefore, before proceeding to the proof of Proposition \ref{prop:u}, we first give an estimate of $u$ itself as follows.
\begin{lemma}\label{lem:zeroorderofu}
    Under the same assumptions as in Proposition \ref{prop:u}, there holds
    \begin{equation}\label{ineq:zeroofu}
        \begin{aligned}
            &\|u_\Phi(t)\|_{H^{1,0}_\Psi}^2+2(\gamma-C_\eta)\|u_\Phi\|_{L_{t,\dot{\mu}}^2(H^{\frac{5}{4},0}_\Psi)}^2+(\theta^E-C\zeta(1+t)^\frac{1}{4}-2\eta)\|\partial_yu_\Phi\|_{L_t^2(H^{1,0}_\Psi)}^2\\
 & \leq \|u_\Phi(0)\|_{H^{1,0}_{\Psi_0}}^2+C\|\partial_xu_\Phi\|_{L_{t,\dot{\mu}}^2(H^{s+\frac{1}{4},0}_\Psi)}^2+2\eta\|\partial_y\theta_\Phi\|_{L_t^2(H^{1,0}_\Psi)}^2.
        \end{aligned}
    \end{equation}
\end{lemma}

\begin{proof}
By using Bony's decomposition to the equation of $u$ in \eqref{eq:main}, we can write
\begin{equation}\label{eq:u}
    \partial_t u-\theta^E\partial_y^2 u=-T_u\partial_x u+T_{\partial_yu}\int_0^y\partial_xu d\tilde{y}+\mathfrak{f},
\end{equation}
where 
\begin{equation}
    \begin{aligned}
        \mathfrak{f}\triangleq&-T_{\partial_xu}u-R(u,\partial_xu)+T_{\int_0^y\partial_xud\tilde{y}}\partial_yu+R\left(\int_0^y\partial_xu d\tilde{y},\partial_yu\right)\\
        &-\partial_y\theta\partial_yu-\int_0^y(\partial_yu)^2 d\tilde{y}\partial_yu+\theta\partial_y^2u.
    \end{aligned}
\end{equation}
Applying the operator $e^{\Phi(t,D_x)}$ to \eqref{eq:u} and taking $H^{1,0}_\Psi$ inner product with $u_\Phi$, we obtain
\begin{equation}\label{eq:uPhi}
    \begin{aligned}
        &\langle\partial_tu_\Phi,u_\Phi\rangle_{H^{1,0}_\Psi}+\gamma\langle\dot{\mu}\langle D_x\rangle^\frac{1}{2}u_\Phi,u_\Phi\rangle_{H^{1,0}_\Psi}-\theta^E\langle\partial_y^2u_\Phi,u_\Phi\rangle_{H^{1,0}_\Psi}\\
        & =-\langle[T_u\partial_xu]_\Phi,u_\Phi\rangle_{H^{1,0}_\Psi}+\langle\left[T_{\partial_yu}\int_0^y\partial_xud\tilde{y}\right]_\Phi,u_\Phi\rangle_{H^{1,0}_\Psi}+\langle\mathfrak{f}_\Phi,u_\Phi\rangle_{H^{1,0}_\Psi}.
    \end{aligned}
\end{equation}
First, as in Lemma \ref{lem:leftofmathcalU}, the left-hand side of \eqref{eq:uPhi} is estimated as:
\begin{equation}
    \begin{aligned}
        &\langle\partial_tu_\Phi,u_\Phi\rangle_{H^{1,0}_\Psi}+\gamma\langle\dot{\mu}(t)\langle D_x\rangle^\frac{1}{2}u_\Phi,u_\Phi\rangle_{H^{1,0}_\Psi}-\theta^E\langle\partial_y^2u_\Phi,u_\Phi\rangle_{H^{1,0}_\Psi}\\
        &  \geq\frac{1}{2}\frac{d}{dt}\|u_\Phi(t)\|_{H^{1,0}_\Psi}^2 + \gamma\dot{\mu}(t)\|u_\Phi\|_{H^{\frac{5}{4},0}_\Psi}^2+\frac{1}{2}\theta^E\|\partial_yu_\Phi\|_{H^{1,0}_\Psi}^2.
    \end{aligned}
\end{equation}
To estimate the right-hand side of \eqref{eq:uPhi}, first, in a way similar to the study of $A_1$ given in Lemma \ref{lem:rightofmathcalU}, we have
\begin{equation}
    |\langle [T_u \partial_x u]_\Phi, u_\Phi \rangle_{H^{1,0}_\Psi}|\leq C(1+t)^{\frac{1}{4}}\|\partial_yu_\Phi\|_{H^{\frac{3}{2}+,0}_\Psi}\|u_\Phi\|_{H^{\frac{5}{4},0}_\Psi}^2\leq C\dot{\mu}(t)\|u_\Phi\|_{H^{\frac{5}{4},0}_\Psi}^2.       
\end{equation}
Moreover, by using Lemmas \ref{lem:Poincare} and  \ref{lem:paraestimate} it follows that
\begin{equation}
    \begin{aligned}
        &\left|\langle \left[T_{\partial_y u} \int_0^y \partial_x u \, d\tilde{y}\right]_\Phi, u_\Phi \rangle_{H^{1,0}_\Psi}\right| \leq C (1+t)^{\frac{1}{4}} \|\partial_y u_\Phi\|_{H^{\frac{1}{2}+,0}_\Psi} \|\partial_x u_\Phi\|_{H^{1,0}_\Psi} \|u_\Phi\|_{H^{1,0}_\Psi} \\
        &\quad \leq C \dot{\mu}(t) \|\partial_x u_\Phi\|_{H^{s+\frac{1}{4},0}_\Psi} \|u_\Phi\|_{H^{1,0}_\Psi}  \leq C \dot{\mu}(t) \left( \|\partial_x u_\Phi\|_{H^{s+\frac{1}{4},0}_\Psi}^2 + \|u_\Phi\|_{H^{1,0}_\Psi}^2 \right).
    \end{aligned}
\end{equation}
The following estimate of $\langle\mathfrak{f}_\Phi,u_\Phi\rangle_{H^{1,0}_\Psi}$ can be obtained in a way similar to the proof of Proposition 3.2 given in \cite{wang2025compressible}, 
\begin{equation}
    \begin{aligned}
        &|\langle\mathfrak{f}_\Phi,u_\Phi\rangle_{H^{1,0}_\Psi}|\leq C_\eta\dot{\mu}(t)\|u_\Phi\|_{H^{1,0}_\Psi}^2+(\eta+C\zeta(1+t)^\frac{1}{4})\|\partial_yu_\Phi\|_{H^{1,0}_\Psi}^2+\eta\|\partial_y\theta_\Phi\|_{H^{1,0}_\Psi}^2.
    \end{aligned}
\end{equation}

Combining the above estimates and integrating the resulting inequality over $[0,t]$, it follows the estimate \eqref{ineq:zeroofu} immediately.

.
\end{proof}

Next let us turn to the proof of the $H^{s+1,0}_\Psi$ estimate of $u$.
\begin{proof}[Proof of Proposition \ref{prop:u}]
By using \eqref{def:lambda2}, a direct calculation gives that
\begin{equation}\label{ineq:partialxu}
    \begin{aligned}
        &\|\partial_x u_\Phi(t)\|_{H^{s,0}_\Psi}^2 
        + 2(\gamma-C_\eta)\|\partial_x u_\Phi\|_{L_{t,\dot{\mu}}^2(H^{s+\frac{1}{4},0}_\Psi)}^2 
        + (\theta^E-C\zeta(1+t)^\frac{1}{4}-2\eta) \|\partial_y \partial_x u_\Phi\|_{L_t^2(H^{s,0}_\Psi)}^2 \\
        \leq& \|\lambda_\Phi(t)\|_{H^{s,0}_\Psi}^2 
        + 2(\gamma-C_\eta) \|\lambda_\Phi\|_{L_{t,\dot{\mu}}^2(H^{s+\frac{1}{4},0}_\Psi)}^2 
        + (\theta^E-C\zeta(1+t)^\frac{1}{4}-2\eta) \|\partial_y \lambda_\Phi\|_{L_t^2(H^{s,0}_\Psi)}^2 \\
        & + \|\left[T_{\partial_y u} \int_0^y \mathcal{U} d\tilde{y}\right]_\Phi(t)\|_{H^{s,0}_\Psi}^2 
        + 2(\gamma-C_\eta) \|\left[T_{\partial_y u} \int_0^y \mathcal{U} d\tilde{y}\right]_\Phi\|_{L_{t,\dot{\mu}}^2(H^{s+\frac{1}{4},0}_\Psi)}^2 \\
        & + (\theta^E-C\zeta(1+t)^\frac{1}{4}-2\eta) \|[T_{\partial_y u} \mathcal{U}]_\Phi\|_{L_t^2(H^{s,0}_\Psi)}^2
        + \theta^E \|\left[T_{\partial_y^2 u} \int_0^y \mathcal{U} d\tilde{y}\right]_\Phi\|_{L_t^2(H^{s,0}_\Psi)}^2. 
    \end{aligned}
\end{equation}
Thanks to Lemmas \ref{lem:paraestimate}, \ref{lem:Poincare} and Assumption \ref{assumption}, we can estimate the last four terms as follows:
\begin{equation}\label{ineq:TpartialyuintU1}
\begin{aligned}
    &\big\| \left[T_{\partial_y u} \int_0^y \mathcal{U} \, d\tilde{y}\right]_\Phi(t) \big\|_{H^{s,0}_\Psi}^2 
    + 2(\gamma - C_\eta) \big\| \left[T_{\partial_y u} \int_0^y \mathcal{U} \, d\tilde{y}\right]_\Phi \big\|_{L_{t,\dot{\mu}}^2(H^{s+\frac{1}{4},0}_\Psi)}^2 \\
    &\quad + \big(\theta^E - C\zeta (1+t)^{\frac{1}{4}} - 2\eta\big) \big\| [T_{\partial_y u} \mathcal{U}]_\Phi \big\|_{L_t^2(H^{s,0}_\Psi)}^2 \\
    &\leq C (1+t)^{\frac{1}{2}} \|\partial_y u_\Phi\|_{L_t^\infty(H^{\frac{1}{2}+,0}_\Psi)}^2 \Big(
        \|\mathcal{U}_\Phi(t)\|_{H^{s,0}_\Psi}^2 
        + 2(\gamma - C\eta) \|\mathcal{U}_\Phi\|_{L_{t,\dot{\mu}}^2(H^{s+\frac{1}{4},0}_\Psi)}^2 \\
    &\hspace{7cm}
        + (\theta^E - C\zeta (1+t)^{\frac{1}{4}} - 2\eta) \|\partial_y \mathcal{U}_\Phi\|_{L_t^2(H^{s,0}_\Psi)}^2
    \Big) \\
    &\leq C (1+t)^{\frac{1}{2}} M^2 \Big(
        \|\mathcal{U}_\Phi(t)\|_{H^{s,0}_\Psi}^2 
        + 2(\gamma - C\eta) \|\mathcal{U}_\Phi\|_{L_{t,\dot{\mu}}^2(H^{s+\frac{1}{4},0}_\Psi)}^2 \\
    &\hspace{7cm}
        + (\theta^E - C\zeta (1+t)^{\frac{1}{4}} - 2\eta) \|\partial_y \mathcal{U}_\Phi\|_{L_t^2(H^{s,0}_\Psi)}^2
    \Big)
\end{aligned}
\end{equation}
and
\begin{equation}\label{ineq:TpartialyuintU2}
    \begin{aligned}
        \theta^E \|\left[T_{\partial_y^2 u} \int_0^y \mathcal{U} d\tilde{y}\right]_\Phi\|_{L_t^2(H^{s,0}_\Psi)}^2\leq C\int_0^t(1+t')^\frac{1}{2}\|\partial_y^2u_\Phi\|_{H^{\frac{1}{2}+,0}_\Psi}^2\|\mathcal{U}_\Phi\|_{H^{s,0}_\Psi}^2 dt'\leq C\|\mathcal{U}_\Phi\|_{L_{t,\dot{\mu}}^2(H^{s,0}_\Psi)}^2.
    \end{aligned}
\end{equation}

By combining Proposition \ref{prop:mathcalU}, Proposition \ref{prop:lambda}, \eqref{ineq:partialxu}-\eqref{ineq:TpartialyuintU2} and Lemma \ref{lem:zeroorderofu}, we obtain the desired estimate.
\end{proof}

\begin{proposition}\label{prop:theta}
    Under the same assumptions as given in Proposition \ref{prop:u}, there holds
    \begin{equation}\label{ineq:theta}
        \begin{aligned}
            &\|\theta_\Phi(t)\|_{H^{s+1,0}_\Psi}^2 + 2(\gamma-C_\eta)\|\theta_\Phi\|_{L_{t,\dot{\mu}}^2(H^{s+\frac{5}{4},0}_\Psi)}^2 + (\theta^E-C\zeta(1+t)^\frac{1}{4}-2\eta) \|\partial_y \theta_\Phi\|_{L_t^2(H^{s+1,0}_\Psi)}^2 \\
            \leq& \|\theta_\Phi(0)\|_{H^{s+1,0}_{\Psi_0}}^2 + C\|u_\Phi\|_{L_{t,\dot{\mu}}^2(H^{s+\frac{5}{4},0}_\Psi)}^2 + 4\eta\|\partial_y u_\Phi\|_{L_t^2(H^{s+1,0}_\Psi)}^2 + C\|\theta_\Phi\|_{L_{t,\dot{\mu}}^2(H^{s+1,0}_\Psi)}^2 \\
            & + (2\eta+C\zeta(1+t)^\frac{1}{4})\|\partial_y\theta_\Phi\|_{L_t^2(H^{s+1,0}_\Psi)}^2 + C\|\mathcal{U}_\Phi\|_{L_{t,\dot{\mu}}^2(H^{s,0}_\Psi)}^2 + 2\eta\|\partial_y \mathcal{U}_\Phi\|_{L_t^2(H^{s-1,0}_\Psi)}^2 \\
            &+ C(1+t)^\frac{1}{2}\zeta^2\Big(2\eta\|\partial_yu_\Phi\|_{L_t^2(H^{s+1,0}_\Psi)}^2 + C\|\lambda_\Phi\|_{L_{t,\dot{\mu}}^2(H^{s+\frac{3}{4},0}_\Psi)}^2 + \frac{5}{\theta^E}\|\partial_y\theta_\Phi\|_{L_t^2(H^{s+1,0}_\Psi)}^2\Big).
        \end{aligned}
    \end{equation}
\end{proposition}

Similar to the proof of Proposition \ref{prop:u}, we can derive the above estimate from 
\begin{equation}\label{3.75}
    \partial_x\theta=\varphi+T_{\partial_y\theta}\int_0^y\mathcal{U}d\tilde{y},
\end{equation}
and the following lemma of the lower order estimate of $\theta$ itself. 

\begin{lemma}\label{lem:zeroorderoftheta}
    Under the same conditions as in Proposition \ref{prop:u}, there holds
    \begin{equation}\label{ineq:zerooftheta}
        \begin{aligned}
            &\|\theta_\Phi(t)\|_{H^{1,0}_\Psi}^2+2(\gamma-C_\eta)\|\theta_\Phi\|_{L_{t,\dot{\mu}}^2(H^{\frac{5}{4},0}_\Psi)}^2+(\theta^E-C\zeta(1+t)^\frac{1}{4}-2\eta)\|\partial_y\theta_\Phi\|_{L_t^2(H^{1,0}_\Psi)}^2\\
            \leq& \|\theta_\Phi(0)\|_{H^{1,0}_{\Psi_0}}^2+C\|u_\Phi\|_{L_{t,\dot{\mu}}^2(H^{1,0}_\Psi)}^2+C\|\partial_xu_\Phi\|_{L_{t,\dot{\mu}}^2(H^{s+\frac{1}{4},0}_\Psi)}^2+2\eta\|\partial_yu_\Phi\|_{L_t^2(H^{1,0}_\Psi)}^2.
        \end{aligned}
    \end{equation}
\end{lemma}

\begin{proof}
    Applying Bony's decomposition to the equation of $\theta$ given in \eqref{eq:main}, one obtains
    \begin{equation}\label{eq:theta}
        \partial_t\theta-\theta^E\partial_y^2\theta=-T_u\partial_x\theta+T_{\partial_y\theta}\int_0^y\partial_xu d\tilde{y}+\mathfrak{g},
    \end{equation}
    where
    \begin{equation}
        \begin{aligned}
            \mathfrak{g}\triangleq&-T_{\partial_x\theta}u-R(u,\partial_x\theta)+T_{\int_0^y\partial_xud\tilde{y}}\partial_y\theta+R\left(\int_0^y\partial_xu d\tilde{y},\partial_y\theta\right)\\
            &-(\partial_y\theta)^2-\int_0^y(\partial_yu)^2 d\tilde{y}\partial_y\theta+\theta\partial_y^2\theta+(\theta+\theta^E)(\partial_yu)^2.
        \end{aligned}
    \end{equation}
    Applying the operator $e^{\Phi(t,D_x)}$ to \eqref{eq:theta} and taking $H^{1,0}_\Psi$ inner product with $\theta_\Phi$, we obtain
    \begin{equation}\label{3.79}
        \begin{aligned}
            &\langle\partial_t\theta_\Phi,\theta_\Phi\rangle_{H^{1,0}_\Psi}+\gamma\langle\dot{\mu}\langle D_x\rangle^\frac{1}{2}\theta_\Phi,\theta_\Phi\rangle_{H^{1,0}_\Psi}-\theta^E\langle\partial_y^2\theta_\Phi,\theta_\Phi\rangle_{H^{1,0}_\Psi}\\
            =&-\langle[T_u\partial_x\theta]_\Phi,\theta_\Phi\rangle_{H^{1,0}_\Psi}+\langle\left[T_{\partial_y\theta}\int_0^y\partial_xud\tilde{y}\right]_\Phi,\theta_\Phi\rangle_{H^{1,0}_\Psi}+\langle\mathfrak{g}_\Phi,\theta_\Phi\rangle_{H^{1,0}_\Psi}.
        \end{aligned}
    \end{equation}
    It is noted that there is only a slight difference in the estimate for $\langle[(\theta+\theta^E)(\partial_yu)^2]_\Phi,\theta_\Phi\rangle_{H^{1,0}_\Psi}$ in $\langle\mathfrak{g}_\Phi,\theta_\Phi\rangle_{H^{1,0}_\Psi}$ compared with the proof of Proposition 3.3 in \cite{wang2025compressible}. In particular, for the term $\langle[T_\theta (\partial_yu)^2]_\Phi,\theta_\Phi\rangle_{H^{1,0}_\Psi}$ obtained from Bony's decomposition, we apply the estimates from Lemmas \ref{lem:paraestimate} and \ref{lem:Poincare} twice to derive the following inequality:
    \begin{equation}
        \begin{aligned}
            |\langle[T_\theta (\partial_yu)^2]_\Phi,\theta_\Phi\rangle_{H^{1,0}_\Psi}|\leq& 2|\langle[T_\theta T_{\partial_yu}\partial_yu]_\Phi,\theta_\Phi\rangle_{H^{1,0}_\Psi}|+|\langle[T_\theta R(\partial_yu,\partial_yu)]_\Phi,\theta_\Phi\rangle_{H^{1,0}_\Psi}|\\
            \leq& C(1+t)^\frac{1}{2}\|\partial_y\theta_\Phi\|_{H^{\frac{1}{2}+,0}_\Psi}\|\partial_y^2u_\Phi\|_{H^{\frac{1}{2}+,0}_\Psi}\|\partial_yu_\Phi\|_{H^{1,0}_\Psi}\|\theta_\Phi\|_{H^{1,0}_\Psi}\\
            \leq& C_\eta(1+t)\|\partial_y\theta_\Phi\|_{H^{\frac{1}{2}+,0}_\Psi}^2\|\partial_y^2u_\Phi\|_{H^{\frac{1}{2}+,0}_\Psi}^2\|\theta_\Phi\|_{H^{1,0}_\Psi}^2+\eta\|\partial_yu_\Phi\|_{H^{1,0}_\Psi}^2\\
            \leq& C_\eta\dot{\mu}(t)\|\theta_\Phi\|_{H^{1,0}_\Psi}^2+\eta\|\partial_yu_\Phi\|_{H^{1,0}_\Psi}^2.
        \end{aligned}
    \end{equation} 
    Thus in the same way as the proof of Lemma \ref{lem:zeroorderofu}, we can obtain the estimate \eqref{ineq:zerooftheta} from \eqref{3.79}.
\end{proof}

\begin{proof}[Proof of Proposition \ref{prop:theta}]
    The proof follows a similar approach to that of Proposition \ref{prop:u}. By using Lemmas \ref{lem:paraestimate}, \ref{lem:Poincare} and Assumption \ref{assumption}, from \eqref{3.75} we have
    \begin{equation}
        \begin{aligned}
            &\|\partial_x\theta_\Phi(t)\|_{H^{s,0}_\Psi}^2+2(\gamma-C_\eta)\|\partial_x\theta_\Phi\|_{L_{t,\dot{\mu}}^2(H^{s+\frac{1}{4},0}_\Psi)}^2+(\theta^E-C\zeta(1+t)^\frac{1}{4}-2\eta)\|\partial_y\partial_x\theta_\Phi\|_{L_t^2(H^{s,0}_\Psi)}^2\\
            \leq& \|\varphi_\Phi(t)\|_{H^{s,0}_\Psi}^2+2(\gamma-C_\eta)\|\varphi_\Phi\|_{L_{t,\dot{\mu}}^2(H^{s+\frac{1}{4},0}_\Psi)}^2+(\theta^E-C\zeta(1+t)^\frac{1}{4}-2\eta)\|\partial_y\varphi_\Phi\|_{L_t^2(H^{s,0}_\Psi)}^2\\
            &+\|\left[T_{\partial_y\theta}\int_0^y\mathcal{U} d\tilde{y}\right]_\Phi(t)\|_{H^{s,0}_\Psi}^2+2(\gamma-C_\eta)\|\left[T_{\partial_y\theta}\int_0^y\mathcal{U} d\tilde{y}\right]_\Phi\|_{L_{t,\dot{\mu}}^2(H^{s+\frac{1}{4},0}_\Psi)}^2\\
            &+(\theta^E-C\zeta(1+t)^\frac{1}{4}-2\eta)\|[T_{\partial_y\theta}\mathcal{U} ]_\Phi\|_{L_t^2(H^{s,0}_\Psi)}^2+\theta^E\|\left[T_{\partial_y^2\theta}\int_0^y\mathcal{U} d\tilde{y}\right]_\Phi\|_{L_t^2(H^{s,0}_\Psi)}^2\\
            \leq& \|\varphi_\Phi(t)\|_{H^{s,0}_\Psi}^2+2(\gamma-C_\eta)\|\varphi_\Phi\|_{L_{t,\dot{\mu}}^2(H^{s+\frac{1}{4},0}_\Psi)}^2+(\theta^E-C\zeta(1+t)^\frac{1}{4}-2\eta)\|\partial_y\varphi_\Phi\|_{L_t^2(H^{s,0}_\Psi)}^2\\
            &+C(1+t)^\frac{1}{2}\zeta^2\Big(\|\mathcal{U}_\Phi(t)\|_{H^{s,0}_\Psi}^2+2(\gamma-C_\eta)\|\mathcal{U}_\Phi\|_{L_{t,\dot{\mu}}^2(H^{s+\frac{1}{4},0}_\Psi)}^2+(\theta^E-C\zeta(1+t)^\frac{1}{4}-2\eta)\|\partial_y\mathcal{U}_\Phi\|_{L_t^2(H^{s,0}_\Psi)}^2\Big)\\
            &+C\|\mathcal{U}_\Phi\|_{L_{t,\dot{\mu}}^2(H^{s,0}_\Psi)}^2.
        \end{aligned}
    \end{equation}
    Thus, by combining the estimates given in Proposition \ref{prop:mathcalU}, Proposition \ref{prop:phi} and Lemma \ref{lem:zeroorderoftheta} respectively, with the above inequality, it follows the estimate \eqref{ineq:theta} immediately.
\end{proof}

\subsection{Estimates on \texorpdfstring{$\partial_yu,\partial_y\theta,\partial_y^2u,\partial_y^2\theta$}{partial y u,partial y theta,partial y2 u,partial y2 theta}}
The aim of this subsection is to establish the Gevrey estimates of the first and second-order normal derivatives of $u$ and $\theta$.
First, we have

\begin{proposition}\label{prop:partialyu}
    Under the same assumptions as in Proposition \ref{prop:u}, there holds
    \begin{equation}\label{ineq:partialyu}
        \begin{aligned}
            &\|\partial_yu_\Phi(t)\|_{H^{s,0}_\Psi}^2+2(\gamma-C_\eta)\|\partial_yu_\Phi\|_{L_{t,\dot{\mu}}^2(H^{s+\frac{1}{4},0}_\Psi)}^2+(\theta^E-C\zeta(1+t)^\frac{1}{4}-2\eta)\|\partial_y^2u_\Phi\|_{L_t^2(H^{s,0}_\Psi)}^2\\
            & \leq\|\partial_yu_\Phi(0)\|_{H^{s,0}_{\Psi_0}}^2+C(\|u_\Phi\|_{L_{t,\dot{\mu}}^2({H^{s+1,0}_\Psi})}^2+\|\partial_y\theta_\Phi\|_{L_{t,\dot{\mu}}^2({H^{s,0}_\Psi})}^2)+2\eta\|\partial_y^2\theta_\Phi\|_{L_t^2({H^{s,0}_\Psi})}^2,
        \end{aligned}
    \end{equation}
    and
    \begin{equation}\label{ineq:partialytheta}
        \begin{aligned}
            &\|\partial_y\theta_\Phi(t)\|_{H^{s,0}_\Psi}^2+2(\gamma-C_\eta)\|\partial_y\theta_\Phi\|_{L_{t,\dot{\mu}}^2(H^{s+\frac{1}{4},0}_\Psi)}^2+(\theta^E-C\zeta(1+t)^\frac{1}{4}-2\eta)\|\partial_y^2\theta_\Phi\|_{L_t^2(H^{s,0}_\Psi)}^2\\
            & \leq\|\partial_y\theta_\Phi(0)\|_{H^{s,0}_{\Psi_0}}^2+C(\|u_\Phi\|_{L_{t,\dot{\mu}}^2({H^{s+1,0}_\Psi})}^2+\|\partial_yu_\Phi\|_{L_{t,\dot{\mu}}^2({H^{s,0}_\Psi})}^2+\|\theta_\Phi\|_{L_{t,\dot{\mu}}^2({H^{s+1,0}_\Psi})}^2)+2\eta\|\partial_y^2u_\Phi\|_{L_t^2({H^{s,0}_\Psi})}^2.
        \end{aligned}
    \end{equation}
\end{proposition}
\begin{proof}
(1) We get first, by taking $\partial_y$ to the equation of $u$ given in \eqref{eq:main}, that
\begin{equation}\label{3.84}
    \partial_t\partial_yu-\theta^E\partial_y^3u=-T_u\partial_x\partial_yu+T_{\partial_y^2u}\int_0^y\partial_xu d\tilde{y}+\partial_y\mathfrak{f}
\end{equation}
with 
\begin{equation}
    \begin{aligned}
        \partial_y\mathfrak{f}=&-T_{\partial_x\partial_yu}u-R(u,\partial_x\partial_yu)+T_{\int_0^y\partial_xud\tilde{y}}\partial_y^2u+R\left(\int_0^y\partial_xu d\tilde{y},\partial_y^2u\right)\\
        &-\partial_y^2\theta\partial_yu-(\partial_yu)^3-\int_0^y(\partial_yu)^2 d\tilde{y}\partial_y^2u+\theta\partial_y^3u.
    \end{aligned}
\end{equation}
Applying the operator $e^{\Phi(t,D_x)}$ on the equation \eqref{3.84} and taking $H^{s,0}_\Psi$ inner product with $\partial_yu_\Phi$, it follows 
\begin{equation}\label{eq:partialyuPhi}
    \begin{aligned}
        &\langle\partial_t\partial_yu_\Phi,\partial_yu_\Phi\rangle_{H^{s,0}_\Psi}+\gamma\langle\dot{\mu}(t)\langle D_x\rangle^\frac{1}{2}\partial_yu_\Phi,\partial_yu_\Phi\rangle_{H^{s,0}_\Psi}-\theta^E\langle\partial_y^3u_\Phi,\partial_yu_\Phi\rangle_{H^{s,0}_\Psi}\\
        =&-\langle[T_u\partial_x\partial_yu]_\Phi,\partial_yu_\Phi\rangle_{H^{s,0}_\Psi}+\langle\left[T_{\partial_y^2u}\int_0^y\partial_xud\tilde{y}\right]_\Phi,\partial_yu_\Phi\rangle_{H^{s,0}_\Psi}+\langle\partial_y\mathfrak{f}_\Phi,\partial_yu_\Phi\rangle_{H^{s,0}_\Psi}.
    \end{aligned}
\end{equation}
In a way similar to the study of $A_1$ given in Lemma \ref{lem:rightofmathcalU}, we have
\begin{equation}
    |\langle[T_u\partial_x\partial_yu]_\Phi,\partial_yu_\Phi\rangle_{H^{s,0}_\Psi}|\leq C(1+t)^\frac{1}{4}\|\partial_yu_\Phi\|_{H^{\frac{3}{2}+,0}_\Psi}\|\partial_yu_\Phi\|_{H^{s+\frac{1}{4},0}_\Psi}^2\leq C\dot{\mu}(t)\|\partial_yu_\Phi\|_{H^{s+\frac{1}{4},0}_\Psi}^2.
\end{equation}
By using Lemmas \ref{lem:paraestimate} and \ref{lem:Poincare}, one gets
\begin{equation}
    \begin{aligned}
        |\langle\left[T_{\partial_y^2u}\int_0^y\partial_xud\tilde{y}\right]_\Phi,\partial_yu_\Phi\rangle_{H^{s,0}_\Psi}|\leq& C(1+t)^\frac{1}{4}\|\partial_y^2u_\Phi\|_{H^{\frac{1}{2}+,0}_\Psi}\|u_\Phi\|_{H^{s+1,0}_\Psi}\|\partial_yu_\Phi\|_{H^{s,0}_\Psi}\\
        \leq& C\dot{\mu}(t)(\|u_\Phi\|_{H^{s+1,0}_\Psi}^2+\|\partial_yu_\Phi\|_{H^{s,0}_\Psi}^2).
    \end{aligned}
\end{equation}
The remaining term $\langle\partial_y\mathfrak{f}_\Phi,\partial_yu_\Phi\rangle_{H^{s,0}_\Psi}$ on the right-hand side of \eqref{eq:partialyuPhi} can be estimated in a way  similar that given in the proof of Proposition 3.4 in \cite{wang2025compressible}, and we obtain
\begin{equation}
    \begin{aligned}
        |\langle\partial_y\mathfrak{f}_\Phi,\partial_yu_\Phi\rangle_{H^{s,0}_\Psi}|\leq& C_\eta\dot{\mu}(t)\|\partial_yu_\Phi\|_{H^{s,0}_\Psi}^2+C\dot{\mu}(t)(\|u_\Phi\|_{H^{s,0}_\Psi}^2+\|\partial_y\theta_\Phi\|_{H^{s,0}_\Psi}^2)\\
        &+(\eta+C\zeta(1+t)^\frac{1}{4})\|\partial_y^2u_\Phi\|_{H^{s,0}_\Psi}^2+\eta\|\partial_y^2\theta_\Phi\|_{H^{s,0}_\Psi}^2.
    \end{aligned}
\end{equation}
From the equation of $u$ in \eqref{eq:main}, the boundary condition $u|_{y=0}=v|_{y=0}=0$ and \eqref{ineq:positivity}, we get
    \begin{equation}
        \partial_y^2u|_{y=0}=0.
    \end{equation}
Therefore, we can obtain the estimates of terms on the left-hand side of \eqref{eq:partialyuPhi} by using a similar argument of Lemma \ref{lem:leftofmathcalU}. By integrating the resulting inequality over $[0,t]$, we can easily deduce the prior estimate \eqref{ineq:partialyu}. 

(2) To derive the estimate \eqref{ineq:partialytheta}, parallel to the above calculation, by taking $\partial_y$ to the equation of $\theta$ in \eqref{eq:main}, it gives
\begin{equation}\label{3.91}
    \partial_t\partial_y\theta-\theta^E\partial_y^3\theta=-T_{\partial_yu}\partial_x\theta-T_u\partial_x\partial_y\theta+T_{\partial_y^2\theta}\int_0^y\partial_xu d\tilde{y}+T_{\partial_y\theta}\partial_xu+\partial_y\mathfrak{g}
\end{equation}
with
\begin{equation}
    \begin{aligned}
        \partial_y\mathfrak{g}=&-T_{\partial_x\theta}\partial_yu-R(\partial_yu,\partial_x\theta)-T_{\partial_x\partial_y\theta}u+R(u,\partial_x\partial_y\theta)+T_{\int_0^y\partial_xud\tilde{y}}\partial_y^2\theta+R\left(\partial_y^2\theta,\int_0^y\partial_xu d\tilde{y}\right)\\
        &+T_{\partial_xu}\partial_y\theta+R(\partial_xu,\partial_y\theta)-\partial_y\theta\partial_y^2\theta-\left(\int_0^y(\partial_yu)^2 d\tilde{y}\right)\partial_y^2\theta+2(\theta+\theta^E)\partial_yu\partial_y^2u+\theta\partial_y^3\theta.
    \end{aligned}
\end{equation}
Applying the operator $e^{\Phi(t,D_x)}$ to the equation \eqref{3.91} and taking $H^{s,0}_\Psi$ inner product with $\partial_y\theta_\Phi$, it follows
\begin{equation}\label{eq:partialytheta}
    \begin{aligned}
        &\langle\partial_t\partial_y\theta_\Phi,\partial_y\theta_\Phi\rangle_{H^{s,0}_\Psi}+\gamma\langle\dot{\mu}\langle D_x\rangle^\frac{1}{2}\partial_y\theta_\Phi,\partial_y\theta_\Phi\rangle_{H^{s,0}_\Psi}-\theta^E\langle\partial_y^3\theta_\Phi,\partial_y\theta_\Phi\rangle_{H^{s,0}_\Psi}\\
        =&-\langle[T_{\partial_yu}\partial_x\theta]_\Phi,\partial_y\theta_\Phi\rangle_{H^{s,0}_\Psi}-\langle[T_u\partial_x\partial_y\theta]_\Phi,\partial_y\theta_\Phi\rangle_{H^{s,0}_\Psi}+\langle\left[T_{\partial_y^2\theta}\int_0^y\partial_xud\tilde{y}\right]_\Phi,\partial_y\theta_\Phi\rangle_{H^{s,0}_\Psi}\\
        &+\langle[T_{\partial_y\theta}\partial_xu]_\Phi,\partial_y\theta_\Phi\rangle_{H^{s,0}_\Psi}+\langle\partial_y\mathfrak{g}_\Phi,\partial_y\theta_\Phi\rangle_{H^{s,0}_\Psi}.
    \end{aligned}
\end{equation}
In the same way as in the above estimate for $\partial_yu$, we can have:
\begin{equation}\label{3.94}
    \begin{aligned}
        &|\langle[T_{\partial_yu}\partial_x\theta]_\Phi,\partial_y\theta_\Phi\rangle_{H^{s,0}_\Psi}|+|\langle\left[T_{\partial_y^2\theta}\int_0^y\partial_xud\tilde{y}\right]_\Phi,\partial_y\theta_\Phi\rangle_{H^{s,0}_\Psi}|+|\langle[T_{\partial_y\theta}\partial_xu]_\Phi,\partial_y\theta_\Phi\rangle_{H^{s,0}_\Psi}|\\
        \leq& C(1+t)^\frac{1}{4}\|\partial_y^2u_\Phi\|_{H^{\frac{1}{2}+,0}_\Psi}\|\theta_\Phi\|_{H^{s+1,0}_\Psi}\|\partial_y\theta_\Phi\|_{H^{s,0}_\Psi}+C(1+t)^\frac{1}{4}\|\partial_y^2\theta_\Phi\|_{H^{\frac{1}{2}+,0}_\Psi}\|u_\Phi\|_{H^{s+1,0}_\Psi}\|\partial_y\theta_\Phi\|_{H^{s,0}_\Psi}\\
        \leq& C\dot{\mu}(t)(\|\theta_\Phi\|_{H^{s+1,0}_\Psi}^2+\|\partial_y\theta_\Phi\|_{H^{s,0}_\Psi}^2+\|u_\Phi\|_{H^{s+1,0}_\Psi}^2),
        \end{aligned}
\end{equation}

$$|\langle[T_u\partial_x\partial_y\theta]_\Phi,\partial_y\theta_\Phi\rangle_{H^{s,0}_\Psi}|\leq C(1+t)^\frac{1}{4}\|\partial_yu_\Phi\|_{H^{\frac{3}{2}+,0}_\Psi}\|\partial_y\theta_\Phi\|_{H^{s+\frac{1}{4},0}_\Psi}^2\leq C\dot{\mu}(t)\|\partial_y\theta_\Phi\|_{H^{s+\frac{1}{4},0}_\Psi}^2
$$
and
\begin{equation}\label{3.95}
    \begin{aligned}
        |\langle\partial_y\mathfrak{g}_\Phi,\partial_y\theta_\Phi\rangle_{H^{s,0}_\Psi}|\leq&C_\eta\dot{\mu}(t)\|\partial_y\theta_\Phi\|_{H^{s,0}_\Psi}^2+C\dot{\mu}(t)(\|u_\Phi\|_{H^{s,0}_\Psi}^2+\|\partial_yu_\Phi\|_{H^{s,0}_\Psi}^2+\|\theta_\Phi\|_{H^{s,0}_\Psi}^2)\\
        &+(\eta+C\zeta(1+t)^\frac{1}{4})\|\partial_y^2\theta_\Phi\|_{H^{s,0}_\Psi}^2+\eta\|\partial_y^2u_\Phi\|_{H^{s,0}_\Psi}^2.
    \end{aligned}
\end{equation}
By noting $\partial_y\theta|_{y=0}=0$, 
we can derive the estimate \eqref{ineq:partialytheta}
from the above identity \eqref{eq:partialytheta} by using the inequalities \eqref{3.94}-\eqref{3.95}. Thus we end the proof of Proposition \ref{prop:partialyu}.
\end{proof}

\begin{proposition}\label{prop:partialy2}
    Under the same assumptions as in Proposition \ref{prop:u}, there holds
    \begin{equation}\label{ineq:partialy2}
        \begin{aligned}
            &\|\partial_y^2u_\Phi(t)\|_{H^{s-1,0}_\Psi}^2+\|\partial_y^2\theta_\Phi(t)\|_{H^{s-1,0}_\Psi}^2+2(\gamma-C_\eta)(\|\partial_y^2u_\Phi\|_{L_{t,\dot{\mu}}^2(H^{s-\frac{3}{4},0}_\Psi)}^2+\|\partial_y^2\theta_\Phi\|_{L_{t,\dot{\mu}}^2(H^{s-\frac{3}{4},0}_\Psi)}^2)\\
            &+(\theta^E-C\zeta(1+t)^\frac{1}{4}-2\eta)(\|\partial_y^3u_\Phi\|_{L_t^2(H^{s-1,0}_\Psi)}^2+\|\partial_y^3\theta_\Phi\|_{L_t^2(H^{s-1,0}_\Psi)}^2)\\
            \leq& \|\partial_y^2u_\Phi(0)\|_{H^{s-1,0}_{\Psi_0}}^2+\|\partial_y^2\theta_\Phi(0)\|_{H^{s-1,0}_{\Psi_0}}^2+C(\|u_\Phi\|_{L_{t,\dot{\mu}}^2({H_\Psi^{s,1}})}^2+\|\theta_\Phi\|_{L_{t,\dot{\mu}}^2({H^{s-1,0}_\Psi})}^2+\|\partial_y\theta_\Phi\|_{L_{t,\dot{\mu}}^2({H^{s,0}_\Psi})}^2).
        \end{aligned}
    \end{equation}
\end{proposition}
The proof of the above estimate of $\partial_y^2u$ and $\partial_y^2\theta$ closely follows the above same  arguments used for $\partial_yu$ and $\partial_y\theta$, for brevity, we omit the details.

\section{Proof of the main result}
The proposal of this section is to finish the proof of the main result stated in Theorem \ref{thm:main}. First, we shall prove a priori estimate, from which one can close the a prior Assumption \ref{assumption} given at the beginning of \S3, by a bootstrap argument.

\begin{theorem}[A priori estimate]\label{thm:priori}
    Suppose that $(u,\theta)$ is a solution of \eqref{eq:main} and $\mathcal{U}$, $\lambda$ are auxiliary functions defined by \eqref{eq:intmathcalU}, \eqref{def:lambda} respectively, with norms appearing in \eqref{ineq:thmpriori}, \eqref{ineq:thmpriori2}, \eqref{ineq:thmpriori3} being finite. Define an energy functional as 
    $$\mathcal{E}_s[f](t)\triangleq\|f_\Phi\|_{L_t^\infty(H^{s,0}_\Psi)}^2+\gamma\|f_\Phi\|_{L^2_{t,\dot{\mu}}(H^{s+\frac{1}{4},0}_\Psi)}^2+\frac{\theta^E}{16}\|\partial_y f_\Phi\|_{L^2_t(H^{s,0}_\Psi)}^2.$$
    Then there exist positive time $T>0$, and  large constants $k>0$ and $\gamma>0$, so that for a fixed $s>\frac{5}{2}$ the following estimates 
    \begin{align}
        &\mathcal{E}_s[\mathcal{U}](t)+\mathcal{E}_{s+\frac{1}{2}}[\lambda](t)+\mathcal{E}_{s+1}[u](t)+k\mathcal{E}_{s+1}[\theta](t)\leq 2\|e^{\delta\langle D_x\rangle^{\frac{1}{2}}}u_0\|_{H^{s+\frac{3}{2},0}_{\Psi_0}}^2+k\|e^{\delta\langle D_x\rangle^{\frac{1}{2}}}\theta_0\|_{H^{s+1,0}_{\Psi_0}}^2,\label{ineq:thmpriori}\\
        &\mathcal{E}_s[\partial_yu](t)+\mathcal{E}_s[\partial_y\theta](t)\leq \frac{3}{2}\bigg(\|e^{\delta\langle D_x\rangle^{\frac{1}{2}}}\partial_yu_0\|_{H^{s,0}_{\Psi_0}}^2+\|e^{\delta\langle D_x\rangle^{\frac{1}{2}}}\partial_y\theta_0\|_{H^{s,0}_{\Psi_0}}^2\bigg),\label{ineq:thmpriori2}\\
        &\mathcal{E}_{s-1}[\partial_y^2u](t)+\mathcal{E}_{s-1}[\partial_y^2\theta](t)\leq \frac{3}{2}\bigg(\|e^{\delta\langle D_x\rangle^{\frac{1}{2}}}\partial_y^2u_0\|_{H^{s-1,0}_{\Psi_0}}^2+\|e^{\delta\langle D_x\rangle^{\frac{1}{2}}}\partial_y^2\theta_0\|_{H^{s-1,0}_{\Psi_0}}^2\bigg).\label{ineq:thmpriori3}
    \end{align}
holds for any $0<t<T$.
\end{theorem}

\begin{proof}
(1)    Under Assumption \ref{assumption}, by summing up \eqref{ineq:mathcalU}, \eqref{ineq:lambda}, \eqref{ineq:u} and $k\times$\eqref{ineq:theta}, where $k$ is a positive constant to be determined later, we obtain
    \begin{equation}\label{ineq:priori}
    \begin{aligned}
        &\|\mathcal{U}_\Phi(t)\|_{H^{s,0}_\Psi}^2
        + \|\lambda_\Phi(t)\|_{H^{s+\frac{1}{2},0}_\Psi}^2
        + \|u_\Phi(t)\|_{H^{s+1,0}_\Psi}^2
        + k\|\theta_\Phi(t)\|_{H^{s+1,0}_\Psi}^2 \\[1ex]
        &\quad
        + 2(\gamma - C_\eta) \Big(
            \|\mathcal{U}_\Phi\|_{L^2_{t,\dot{\mu}}(H^{s+\frac{1}{4},0}_\Psi)}^2
            + \|\lambda_\Phi\|_{L^2_{t,\dot{\mu}}(H^{s+\frac{3}{4},0}_\Psi)}^2
            + \|u_\Phi\|_{L^2_{t,\dot{\mu}}(H^{s+\frac{5}{4},0}_\Psi)}^2
            + k\|\theta_\Phi\|_{L^2_{t,\dot{\mu}}(H^{s+\frac{5}{4},0}_\Psi)}^2
        \Big) \\[1ex]
        &\quad
        + \big(\theta^E - C\zeta (1+t)^{\frac{1}{4}} - 2\eta\big)
            \Big(
                \|\partial_y \mathcal{U}_\Phi\|_{L^2_t(H^{s,0}_\Psi)}^2
                + \|\partial_y \lambda_\Phi\|_{L^2_t(H^{s+\frac{1}{2},0}_\Psi)}^2
                + \|\partial_y u_\Phi\|_{L^2_t(H^{s+1,0}_\Psi)}^2
            \Big) \\[1ex]
        &\quad
        + k\big(\theta^E - C\zeta (1+t)^{\frac{1}{4}} - 2\eta\big)
            \|\partial_y \theta_\Phi\|_{L^2_t(H^{s+1,0}_\Psi)}^2 \\[2ex]
        &\leq
        2\|u_\Phi(0)\|_{H^{s+\frac{3}{2},0}_\Psi}^2
        + k\|\theta_\Phi(0)\|_{H^{s+1,0}_\Psi}^2 \\[1ex]
        &\quad
        + C\Big(
            (1+k)\|\mathcal{U}_\Phi\|_{L^2_{t,\dot{\mu}}(H^{s+\frac{1}{4},0}_\Psi)}^2
            + \big[1 + (1+t)^{\frac{1}{2}}M^2 + (1+t)^{\frac{1}{2}}\zeta^2 k\big]
                \|\lambda_\Phi\|_{L^2_{t,\dot{\mu}}(H^{s+\frac{3}{4},0}_\Psi)}^2 \\[-1ex]
        &\hspace{4em}
            + (1+k)\|u_\Phi\|_{L^2_{t,\dot{\mu}}(H^{s+\frac{5}{4},0}_\Psi)}^2
            + (1+k)\|\theta_\Phi\|_{L^2_{t,\dot{\mu}}(H^{s+\frac{5}{4},0}_\Psi)}^2
        \Big) \\[1ex]
        &\quad
        + \eta(4 + 2 k)\|\partial_y\mathcal{U}_\Phi\|_{L^2_t(H^{s-\frac{1}{4},0}_\Psi)}^2 \\[1ex]
        &\quad
        + \eta\Big(
            6
            + C(1+t)^{\frac{1}{2}}M^2 2
            + 4 k
            + C(1+t)^{\frac{1}{2}}\zeta^2 2 k
        \Big)
        \|\partial_y u_\Phi\|_{L^2_t(H^{s+1,0}_\Psi)}^2 \\[1ex]
        &\quad
        + \Big(
            6\eta + \frac{5}{\theta^E}
            + C(1+t)^{\frac{1}{2}}M^2 \frac{5}{\theta^E}
            + 2\eta k
            + C(1+t)^{\frac{1}{2}}\zeta^2 \frac{5}{\theta^E}k
        \Big)
        \|\partial_y \theta_\Phi\|_{L^2_t(H^{s+1,0}_\Psi)}^2\\[1ex]
        &\triangleq 2\|u_\Phi(0)\|_{H^{s+\frac{3}{2},0}_\Psi}^2
        + k\|\theta_\Phi(0)\|_{H^{s+1,0}_\Psi}^2
        + D_1 +\cdots + D_4.
    \end{aligned}
    \end{equation}
    Obviously, we have
    \begin{equation}\label{ineq:tau}
        \sup\limits_{t\in[0,\tau^*)}C\zeta(1+t)^{\frac{1}{4}}< \frac{\theta^E}{4}\quad\text{when}\quad \tau^*< \left(\frac{\theta^E}{4C\zeta}\right)^4-1=\left(\frac{\theta^E}{8C\epsilon}\right)^4-1.
    \end{equation}
    Therefore, we can take a sufficiently small 
    \begin{equation}\label{ineq:eta}
        \eta\leq \frac{\theta^E}{16}
    \end{equation}
    such that 
    \begin{equation}\label{ineq:eta2}
        \theta^E-C\zeta(1+t)^\frac{1}{4}-2\eta\geq\frac{\theta^E}{2}, \quad \forall t\in [0, \tau^*).
    \end{equation}

   Next, by choose $k$ properly large satisfying
    \begin{equation}
        k\geq \frac{\frac{3\theta^E}{8}+\frac{5}{\theta^E}+C(1+t)^\frac{1}{2}M^2\frac{5}{\theta^E}}{\frac{5\theta^E}{16}-C(1+t)^{\frac{1}{2}}\zeta^2\frac{5}{\theta^E}},  \quad \forall t\in [0, \tau^*),     \end{equation}
  whose denominator is positive from \eqref{ineq:tau}, one has  by using \eqref{ineq:eta}, \eqref{ineq:eta2} that
    \begin{equation}
        k(\theta^E-C\zeta(1+t)^\frac{1}{4}-2\eta)-\bigg[6\eta+\frac{5}{\theta^E}+C(1+t)^{\frac{1}{2}}M^2\frac{5}{\theta^E}+2\eta k+C(1+t)^{\frac{1}{2}}\zeta^2\frac{5}{\theta^E}k\bigg]\geq \frac{\theta^E}{16}k.
    \end{equation}
Thus, the term $D_4$ given on the right hand side of \eqref{ineq:priori} can be absorbed by the fourth line on the left hand side of \eqref{ineq:priori}. 

After fixing $k$ as above, choosing $\eta>0$ small satisfying \eqref{ineq:eta}, \eqref{ineq:eta2}, the terms 
   $D_2$ and $D_3$ on the right hand side of \eqref{ineq:priori} can be absorbed by the third line on the left hand side of \eqref{ineq:priori}, by noting that $M$ and $\zeta$ are two positive constants given in \eqref{ineq:assumption} and \eqref{def:assumption} respectively.

Finally, we can choose a large $\gamma$ such that $\gamma>C_\eta$ to absorb $D_1$ by the second line of \eqref{ineq:priori}.
    Thus, we conclude the estimate \eqref{ineq:thmpriori} from the above inequality \eqref{ineq:priori}.

  (2)  By summing up \eqref{ineq:partialyu} and \eqref{ineq:partialytheta} given in Proposition \ref{prop:partialyu}, it follows that
    \begin{equation}\label{4.10}
        \begin{aligned}
            &\|\partial_yu_\Phi(t)\|_{H^{s,0}_\Psi}^2+\|\partial_y\theta_\Phi(t)\|_{H^{s,0}_\Psi}^2+2(\gamma-C_\eta)(\|\partial_yu_\Phi\|_{L^2_{t,\dot{\mu}}(H^{s+\frac{1}{4},0}_\Psi)}^2+\|\partial_y\theta_\Phi\|_{L^2_{t,\dot{\mu}}(H^{s+\frac{1}{4},0}_\Psi)}^2)\\
            &+(\theta^E-C\zeta(1+t)^\frac{1}{4}-2\eta)(\|\partial_y^2u_\Phi\|_{L^2_t(H^{s,0}_\Psi)}^2+\|\partial_y^2\theta_\Phi\|_{L^2_t(H^{s,0}_\Psi)}^2)\\
            \leq& \|\partial_yu_\Phi(0)\|_{H^{s,0}_{\Psi_0}}^2+\|\partial_y\theta_\Phi(0)\|_{H^{s,0}_{\Psi_0}}^2+C(\|u_\Phi\|_{L^2_{t,\dot{\mu}}(H^{s+1,0}_\Psi)}^2+\|\theta_\Phi\|_{L^2_{t,\dot{\mu}}(H^{s+1,0}_\Psi)}^2)\\
            &+C(\|\partial_yu_\Phi\|_{L^2_{t,\dot{\mu}}(H^{s,0}_\Psi)}^2+\|\partial_y\theta_\Phi\|_{L^2_{t,\dot{\mu}}(H^{s,0}_\Psi)}^2)+2\eta(\|\partial_y^2u_\Phi\|_{L^2_t(H^{s,0}_\Psi)}^2+\|\partial_y^2\theta_\Phi\|_{L^2_t(H^{s,0}_\Psi)}^2).
        \end{aligned}
    \end{equation}
    By taking $\gamma$ sufficiently large and $\eta$ sufficiently small as above, we can absorb the two terms in the last line of the above inequality \eqref{4.10} by the left-hand side. By using \eqref{ineq:thmpriori}, we have
    \begin{equation}
        \begin{aligned}
            \|u_\Phi\|_{L^2_{t,\dot{\mu}}(H^{s+\frac{5}{4},0}_\Psi)}^2+\|\theta_\Phi\|_{L^2_{t,\dot{\mu}}(H^{s+\frac{5}{4},0}_\Psi)}^2
            &\leq \frac{1}{\gamma}\bigg(1+\frac{1}{k}\bigg)\bigg(2\|e^{\delta\langle D_x\rangle^{\frac{1}{2}}}u_0\|_{H^{s+\frac{3}{2},0}_{\Psi_0}}^2+k\|e^{\delta\langle D_x\rangle^{\frac{1}{2}}}\theta_0\|_{H^{s+1,0}_{\Psi_0}}^2\bigg)\\
            &\leq\frac{1}{2C}\bigg(\|e^{\delta\langle D_x\rangle^{\frac{1}{2}}}\partial_yu_0\|_{H^{s,0}_{\Psi_0}}^2+\|e^{\delta\langle D_x\rangle^{\frac{1}{2}}}\partial_y\theta_0\|_{H^{s,0}_{\Psi_0}}^2\bigg).
        \end{aligned}
    \end{equation} 
    provided that $\gamma$ is sufficiently large. Plugging it into \eqref{4.10}, it follows the estimate \eqref{ineq:thmpriori2} immediately. Similarly, we can get the estimate \eqref{ineq:thmpriori3} by using Proposition \ref{prop:partialy2}.

(3) \underline{\it Verification of Assumption \ref{assumption}}: From \eqref{ineq:thmpriori2}, we directly obtain that
    \begin{equation}\label{bootofu}
        \|\partial_yu_\Phi\|_{L_t^\infty(H^{s,0}_\Psi)}\leq \frac{\sqrt{6}}{2}\bigg(\|e^{\delta\langle D_x\rangle^{\frac{1}{2}}}\partial_yu_0\|_{H^{s,0}_{\Psi_0}}+\|e^{\delta\langle D_x\rangle^{\frac{1}{2}}}\partial_y\theta_0\|_{H^{s,0}_{\Psi_0}}\bigg)=\frac{\sqrt{6}}{4}M.
    \end{equation} 
    Moreover, due to the inequality \eqref{ineq:partialytheta}, we have
    \begin{equation}\label{ineq:partialytheta2}
        \begin{aligned}
            \|\partial_y\theta_\Phi\|_{L_t^\infty(H^{s,0}_\Psi)}^2\leq& \|\partial_y\theta_\Phi(0)\|_{H^{s,0}_\Psi}^2+C(\|u_\Phi\|_{L^2_{t,\dot{\mu}}(H^{s+1,0}_\Psi)}^2+\|\partial_yu_\Phi\|_{L^2_{t,\dot{\mu}}(H^{s,0}_\Psi)}^2\\
            &+\|\theta_\Phi\|_{L^2_{t,\dot{\mu}}(H^{s+1,0}_\Psi)}^2)+2\eta\|\partial_y^2u_\Phi\|_{L^2_t(H^{s,0}_\Psi)}^2.
        \end{aligned}
    \end{equation}
    By using \eqref{ineq:thmpriori} and \eqref{ineq:thmpriori2}, we can also obtain that
    \begin{equation}
        \begin{aligned}
            &\|u_\Phi\|_{L^2_{t,\dot{\mu}}(H^{s+1,0}_\Psi)}^2+\|\partial_yu_\Phi\|_{L^2_{t,\dot{\mu}}(H^{s,0}_\Psi)}^2+\|\theta_\Phi\|_{L^2_{t,\dot{\mu}}(H^{s+1,0}_\Psi)}^2\\
            &\leq \frac{1}{\gamma}\bigg[\bigg(1+\frac{1}{k}\bigg)\bigg(2\|e^{\delta\langle D_x\rangle^{\frac{1}{2}}}u_0\|_{H^{s+\frac{3}{2},0}_{\Psi_0}}+k\|e^{\delta\langle D_x\rangle^{\frac{1}{2}}}\theta_0\|_{H^{s+1,0}_{\Psi_0}}\bigg)\\
            &\hspace{.2in}+\frac{3}{2}\bigg(\|e^{\delta\langle D_x\rangle^{\frac{1}{2}}}\partial_yu_0\|_{H^{s,0}_{\Psi_0}}^2+\|e^{\delta\langle D_x\rangle^{\frac{1}{2}}}\partial_y\theta_0\|_{H^{s,0}_{\Psi_0}}^2\bigg)\bigg]\leq \frac{\zeta^2}{4C}
        \end{aligned}
    \end{equation}
    by choosing large $\gamma$ properly.
    Similarly, we choose a small enough $\eta$ such that
    \begin{equation}
        2\eta\|\partial_y^2u_\Phi\|_{L^2_t(H^{s,0}_\Psi)}^2\leq 2\eta\frac{16}{\theta^E}\cdot\frac{3}{2}\bigg(\|e^{\delta\langle D_x\rangle^{\frac{1}{2}}}\partial_yu_0\|_{H^{s,0}_{\Psi_0}}^2+\|e^{\delta\langle D_x\rangle^{\frac{1}{2}}}\partial_y\theta_0\|_{H^{s,0}_{\Psi_0}}^2\bigg)\leq\frac{\zeta^2}{4}.
    \end{equation}
    Therefore, by plugging the above two inequalities into \eqref{ineq:partialytheta2} and using \eqref{assumption2}, it follows that
    \begin{equation}\label{bootoftheta}
        \|\partial_y\theta_\Phi\|_{L_t^\infty(H^{s,0}_\Psi)}\leq \sqrt{\frac{\zeta^2}{4}+\frac{\zeta^2}{4}+\frac{\zeta^2}{4}}=\frac{\sqrt{3}}{2}\zeta<\zeta.
    \end{equation}
    These two inequalities \eqref{bootofu} and \eqref{bootoftheta} improve the assumption \eqref{ineq:assumption}. Hence \eqref{ineq:assumption} holds for all $0<t<T^*$ by a continuity argument. 

(4)
Finally, it order to make sense of the multiplier 
$$u_\Phi(t,x,y)={\cal F}^{-1}_{\xi\to x}(e^{\Phi(t,\xi)}\hat{u}(t,\xi,y))$$
with $\Phi(t,\xi)=(\delta-\gamma\mu(t))\langle\xi\rangle^{\frac 12}$, and $\mu(t)$ being defined in \eqref{eq:mu}, we come to estimate the right hand side of \eqref{eq:mu}.

Obviously, the integration of the terms given on right hand side of \eqref{eq:mu} satisfy
    \begin{equation}
        \begin{aligned}
            &\int_0^t \left( 1 + (1+t')^{\frac{1}{4}} \big(\|(\partial_y u_\Phi, \partial_y \theta_\Phi)\|_{H^{\frac{5}{2}+,0}_\Psi} 
            + \|(\partial_y^2 u_\Phi, \partial_y^2 \theta_\Phi)\|_{H^{\frac{3}{2}+,1}_\Psi} \big) \right) dt' \\
            &\quad \lesssim t + \big((1+t)^{\frac{3}{2}} - 1\big)^{\frac{1}{2}} 
            \big(\|(\partial_y u_\Phi, \partial_y \theta_\Phi)\|_{L_t^2(H^{\frac{5}{2}+,0}_\Psi)} 
            + \|(\partial_y^2 u_\Phi, \partial_y^2 \theta_\Phi)\|_{L_t^2(H^{\frac{3}{2}+,1}_\Psi)} \big), 
                \end{aligned}
    \end{equation}

   \begin{equation}
        \begin{aligned}  
            &\int_0^t (1+t')^{\frac{1}{2}} \big(\|u_\Phi\|_{H^{\frac{5}{2}+,0}_\Psi}^2 
            + \|\theta_\Phi\|_{H^{\frac{3}{2}+,0}_\Psi}^2 
            + \|(\partial_y u_\Phi, \partial_y \theta_\Phi)\|_{H^{\frac{5}{2}+,0}_\Psi}^2 
            + \|(\partial_y^2 u_\Phi, \partial_y^2 \theta_\Phi)\|_{H^{\frac{3}{2}+,0}_\Psi}^2 \big) dt' \\
            &\quad \lesssim \big((1+t)^{\frac{3}{2}} - 1\big) 
            \big(\|u_\Phi\|_{L_t^\infty(H^{\frac{5}{2}+,0}_\Psi)}^2 
            + \|\theta_\Phi\|_{L_t^\infty(H^{\frac{3}{2}+,0}_\Psi)}^2 
            + \|(\partial_y u_\Phi, \partial_y \theta_\Phi)\|_{L_t^\infty(H^{\frac{5}{2}+,0}_\Psi)}^2 \\
            &\qquad + \|(\partial_y^2 u_\Phi, \partial_y^2 \theta_\Phi)\|_{L_t^\infty(H^{\frac{3}{2}+,0}_\Psi)}^2 \big), 
              \end{aligned}
    \end{equation}

   \begin{equation}
        \begin{aligned}        
            &\int_0^t \big(\|\theta_\Phi\|_{H^{\frac{1}{2}+,0}_\Psi}^4 
            + (1+t') \|(\partial_y u_\Phi, \partial_y \theta_\Phi)\|_{H^{\frac{3}{2}+,1}_\Psi}^4 \big) dt' \\
            &\quad \lesssim t \big(\|\theta_\Phi\|_{L_t^\infty(H^{\frac{1}{2}+,0}_\Psi)}^4 
            + (1+t) \|(\partial_y u_\Phi, \partial_y \theta_\Phi)\|_{L_t^\infty(H^{\frac{3}{2}+,1}_\Psi)}^4 \big), 
              \end{aligned}
    \end{equation}

   \begin{equation}
        \begin{aligned}   &\int_0^t (1+t')^{\frac{1}{2}} \|\partial_y \theta_\Phi\|_{H^{\frac{3}{2}+,0}_\Psi} 
            \big(\|\partial_y^3 u_\Phi\|_{H^{\frac{3}{2}+,0}_\Psi} 
            + \|\partial_y^3 \theta_\Phi\|_{H^{\frac{3}{2}+,0}_\Psi} \big) dt' \\
            &\quad \lesssim t^{\frac{1}{2}} (1+t)^{\frac{1}{2}} 
            \big(\|\partial_y^3 u_\Phi\|_{L_t^2(H^{\frac{3}{2}+,0}_\Psi)} 
            + \|\partial_y^3 \theta_\Phi\|_{L_t^2(H^{\frac{3}{2}+,0}_\Psi)} \big) 
            \|\partial_y \theta_\Phi\|_{L_t^\infty(H^{\frac{3}{2}+,0}_\Psi)}, 
              \end{aligned}
    \end{equation}
and
   \begin{equation}
        \begin{aligned}  
            &\int_0^t \left(\big(\|\partial_y^3 u_\Phi\|_{H^{\frac{1}{2}+,0}_\Psi} 
            \big(\|\partial_y u_\Phi\|_{H^{\frac{1}{2}+,0}_\Psi} 
            + \|\partial_y^2 u_\Phi\|_{H^{\frac{1}{2}+,0}_\Psi} 
            + \|\theta_\Phi\|_{H^{\frac{1}{2}+,0}_\Psi} \big)+ \|\partial_y^3 \theta_\Phi\|_{H^{\frac{1}{2}+,0}_\Psi} 
            \|\partial_y u_\Phi\|_{H^{\frac{1}{2}+,0}_\Psi} \big)\right) dt' \\
            &\quad \lesssim t^{\frac{1}{2}} 
            \big(\|\partial_y^3 u_\Phi\|_{L_t^2(H^{\frac{1}{2}+,0}_\Psi)} 
            + \|\partial_y^3 \theta_\Phi\|_{L_t^2(H^{\frac{1}{2}+,0}_\Psi)} \big) 
            \big(\|\partial_y u_\Phi\|_{L_t^\infty(H^{\frac{1}{2}+,0}_\Psi)} 
            + \|\partial_y^2 u_\Phi\|_{L_t^\infty(H^{\frac{1}{2}+,0}_\Psi)} 
            + \|\theta_\Phi\|_{L_t^\infty(H^{\frac{1}{2}+,0}_\Psi)} \big).
        \end{aligned}
    \end{equation}
    When $s>\frac{5}{2}$, by using the estimates given in \eqref{ineq:thmpriori}-\eqref{ineq:thmpriori3},  it deduces that all the norms on the right hand sides of the above five inequalities can be bounded by the initial data. Hence from the definition of $\mu(t)$ given in \eqref{eq:mu}, there exists a small $T^\star>0$ such that
    \begin{equation}
        \sup\limits_{t\in[0,T^\star)}\mu(t)\leq \frac{\delta}{2\gamma}.
    \end{equation}
    Therefore, from \eqref{def:T*}, it implies  that $T^\star\leq T^*$. 
    
    Thus we conclude that \eqref{ineq:thmpriori}-\eqref{ineq:thmpriori3} hold for any $0<t<T$, where $T=\min\{\tau^*,T^\star\}$. This completes the proof of Theorem \ref{thm:priori}.
\end{proof}

\begin{remark}
    We deduce from Lemma \ref{lem:Poincare} and the standard Sobolev embedding that
    \begin{equation}
        \|\theta_\Phi\|_{L_t^\infty(L_+^\infty)} \leq C_*\|\theta_\Phi\|_{L_t^\infty(L_y^\infty(H_x^{\frac{1}{2}+}))} \leq C_*(2\pi\theta^E)^\frac{1}{4}(1+t)^{\frac{1}{4}} \|\partial_y\theta_\Phi\|_{L_t^\infty(H^{s,0}_\Psi)},
    \end{equation}
    where $C_*$ is a constant only depending on the Sobolev embedding. Under Assumption \ref{assumption}, for any $t<T$, we have
    \begin{equation}
        \|\theta_\Phi\|_{L_t^\infty(L_+^\infty)} \leq C_*(2\pi\theta^E)^\frac{1}{4}(1+T)^{\frac{1}{4}}\zeta=2C_*(2\pi\theta^E)^\frac{1}{4}(1+T)^{\frac{1}{4}}\epsilon.
    \end{equation}
    Therefore, by taking $\epsilon\leq\dfrac{\theta^E}{4C_* (2\pi\theta^E)^\frac{1}{4}(1+T)^{\frac{1}{4}}}$, we have $\|\theta_\Phi\|_{L_t^\infty(L_+^\infty)}\leq\dfrac{\theta^E}{2}$, which implies
    \begin{equation}\label{ineq:positivity}
        \theta + \theta^E \geq \frac{\theta^E}{2} > 0.
    \end{equation}
    This lower bound guarantees the strict positivity of the coefficient in the dissipative term of \eqref{eq:main}.
\end{remark}

Since the proof of the main result on the well-posedness of \eqref{eq:main} is similar to that in analytic setting once we have the a priori estimate given in Theorem \ref{thm:priori}, we will only give a sketch of the proof, and refer to \cite[\S4]{wang2025compressible} for details. 
\begin{proof}[Proof of Theorem \ref{thm:main}]
    In order to obtain the existence of solutions to \eqref{eq:main}, there are two main steps, the first one is to construct a regularized boundary layer equation by adding artificial viscosity $\nu>0$ in order to obtain an appropriate approximate solution:
    \begin{equation}\left\{
        \begin{aligned}
            &\partial_t u^\nu+u^\nu\partial_x u^\nu+v^\nu\partial_y u^\nu=\nu\partial_x^2 u^\nu+(\theta^\nu+\theta^E)\partial_y^2 u^\nu,\\
            &\partial_t \theta^\nu+u^\nu\partial_x \theta^\nu+v^\nu\partial_y \theta^\nu=\nu\partial_x^2\theta^\nu+(\theta^\nu+\theta^E)\partial_y^2 \theta^\nu+(\theta^\nu+\theta^E)(\partial_y u^\nu)^2,\\
            &\partial_x u^\nu+\partial_y v^\nu=\partial_y^2\theta^\nu+(\partial_y u^\nu)^2,\\
            &u^\nu|_{y=0}=v^\nu|_{y=0}=\partial_y\theta^\nu|_{y=0}=0,\quad \lim\limits_{y\to+\infty}u^\nu=\lim\limits_{y\to+\infty}\theta^\nu=0,\\
            &(u^\nu,\theta^\nu)|_{t=0}=(u_0,\theta_0)(x,y),
        \end{aligned}
    \right.
    \end{equation}
    then, the second step is to pass to the limit $\nu\to 0$, of approximate problems, by deducing uniform estimates of approximate solution sequence and using a  compact argument.
    The uniqueness of solutions can be proved by using a similar argument as in \cite[\S4]{wang2025compressible}, so we omit the details for brevity. Thus we complete the proof of Theorem \ref{thm:main}.
\end{proof}

\begin{appendices}
\section{Proof of Lemma \ref{lem:rightoflambda}}\label{app:A}
In this section, we provide the proof of Lemma \ref{lem:rightoflambda}.
\begin{proof}[Proof of Lemma \ref{lem:rightoflambda}]
    Recall that in Section \ref{sec:priori}, we defined the right-hand side of \eqref{eq:inneroflambda} as $R$. Therefore, from the equation of $\lambda$ given in \eqref{eq:lambda}, we can decompose $R$ into the following terms:
    \begin{equation}\label{A-1}
        \begin{aligned}
            R=&-\langle [T_u\partial_x\lambda]_\Phi,\lambda_\Phi\rangle_{H^{s+\frac{1}{2},0}_\Psi}-\langle [T_v\partial_y\lambda]_\Phi,\lambda_\Phi\rangle_{H^{s+\frac{1}{2},0}_\Psi}+\langle [T_\theta\partial_y^2\lambda]_\Phi,\lambda_\Phi\rangle_{H^{s+\frac{1}{2},0}_\Psi}-\langle [(\partial_xu)^2]_\Phi,\lambda_\Phi\rangle_{H^{s+\frac{1}{2},0}_\Psi}\\
            &-\langle [T_{\partial_xv}\partial_yu]_\Phi,\lambda_\Phi\rangle_{H^{s+\frac{1}{2},0}_\Psi}-\langle [R(\partial_xv,\partial_yu)]_\Phi,\lambda_\Phi\rangle_{H^{s+\frac{1}{2},0}_\Psi}+\langle[\partial_x\theta\partial_y^2u]_\Phi,\lambda_\Phi\rangle_{H^{s+\frac{1}{2},0}_\Psi}\\
            &-\langle [T_{\partial_x^2u}u]_\Phi,\lambda_\Phi\rangle_{H^{s+\frac{1}{2},0}_\Psi}-\langle [R(\partial_x^2u,u)]_\Phi,\lambda_\Phi\rangle_{H^{s+\frac{1}{2},0}_\Psi}-\langle [T_{\partial_x\partial_yu}v]_\Phi,\lambda_\Phi\rangle_{H^{s+\frac{1}{2},0}_\Psi}\\
            &-\langle [R(\partial_x\partial_yu,v)]_\Phi,\lambda_\Phi\rangle_{H^{s+\frac{1}{2},0}_\Psi}-\langle [T_{\partial_x\partial_y^2u}\theta]_\Phi,\lambda_\Phi\rangle_{H^{s+\frac{1}{2},0}_\Psi}-\langle [R(\partial_x\partial_y^2u,\theta)]_\Phi,\lambda_\Phi\rangle_{H^{s+\frac{1}{2},0}_\Psi}\\
            &-\langle \left[T_{(-\partial_yu\partial_xu-\partial_yv\partial_yu+\partial_y\theta\partial_y^2u)}\int_0^y\mathcal{U}d\tilde{y}\right]_\Phi,\lambda_\Phi\rangle_{H^{s+\frac{1}{2},0}_\Psi}\\
            &-\langle \left[[(T_uT_{\partial_x\partial_yu}-T_{u\partial_x\partial_yu})+(T_vT_{\partial_y^2u}-T_{v\partial_y^2u})-(T_\theta T_{\partial^3_yu}-T_{\theta\partial^3_yu})]\int_0^y\mathcal{U}d\tilde{y}\right]_\Phi,\lambda_\Phi\rangle_{H^{s+\frac{1}{2},0}_\Psi}\\
            &-\langle \left[[T_u;T_{\partial_y u}]\partial_x \int_0^y \mathcal{U}d \tilde{y}\right]_\Phi,\lambda_\Phi\rangle_{H^{s+\frac{1}{2},0}_\Psi}-\langle [[T_v;T_{\partial_y u}]\mathcal{U}]_\Phi,\lambda_\Phi\rangle_{H^{s+\frac{1}{2},0}_\Psi}\\
            &+\langle [[T_\theta;T_{\partial_y u}]\partial_y\mathcal{U}]_\Phi,\lambda_\Phi\rangle_{H^{s+\frac{1}{2},0}_\Psi}+2\langle [T_{(\theta+\theta^E)}T_{\partial_y^2u}\mathcal{U}]_\Phi,\lambda_\Phi\rangle_{H^{s+\frac{1}{2},0}_\Psi}\\
            \triangleq&B_1+\cdots+B_{19}.
        \end{aligned}
    \end{equation}

    (1) \underline{Estimates of $B_1,B_2$ and $B_3$}.
    We first get, by a similar derivation of the estimates of $A_1,A_2$ and $A_3$ given in Lemma \ref{lem:rightofmathcalU}, that
     \begin{equation}
        \begin{aligned}
            |B_1|\leq& C\dot{\mu}(t)\|\lambda_\Phi\|_{H^{s+\frac{3}{4},0}_\Psi}^2,\\
            |B_2|\leq& C_\eta\dot{\mu}(t)\|\lambda_\Phi\|_{H^{s+\frac{1}{2},0}_\Psi}^2+\eta\|\partial_y\lambda_\Phi\|_{H^{s+\frac{1}{2},0}_\Psi}^2,\\
            |B_3|\leq& C_\eta\dot{\mu}(t)\|\lambda_\Phi\|_{H^{s+\frac{1}{2},0}_\Psi}^2+(\eta+C\zeta(1+t)^\frac{1}{4})\|\partial_y\lambda_\Phi\|_{H^{s+\frac{1}{2},0}_\Psi}^2.
        \end{aligned}
    \end{equation}

    (2) \underline{Estimate of $B_4$}.
    By applying Bony's decomposition to $(\partial_xu)^2$, we have
    \begin{equation}
        \begin{aligned}
            |B_4|\leq& 2|\langle[T_{\partial_xu}\partial_xu]_\Phi,\lambda_\Phi\rangle_{H^{s+\frac{1}{2},0}_\Psi}|+|\langle[R(\partial_xu,\partial_xu)]_\Phi,\lambda_\Phi\rangle_{H^{s+\frac{1}{2},0}_\Psi}|\\
            \leq& C(1+t)^\frac{1}{4}\|\partial_yu_\Phi\|_{H^{\frac{3}{2}+,0}_\Psi}\|u_\Phi\|_{H^{s+\frac{5}{4},0}_\Psi}\|\lambda_\Phi\|_{H^{s+\frac{3}{4},0}_\Psi}\\
            \leq& C\dot{\mu}(t)(\|u_\Phi\|_{H^{s+\frac{5}{4},0}_\Psi}^2+\|\lambda_\Phi\|_{H^{s+\frac{3}{4},0}_\Psi}^2).
        \end{aligned}
    \end{equation}

    (3) \underline{Estimates of $B_5$ and $B_6$}.
    It follows from a similar derivation of the estimates of $A_2$ that
    \begin{equation}
        \begin{aligned}
            |B_5|+|B_6|\leq& C\left((1+t)^\frac{1}{4}\|u_\Phi\|_{H^{\frac{5}{2}+,0}_\Psi}+(1+t)^\frac{1}{4}\|\partial_y^2\theta_\Phi\|_{H^{\frac{3}{2}+,0}_\Psi}+\|\partial_yu_\Phi\|_{H^{\frac{3}{2}+,0}_\Psi}^2\right)\|\partial_yu_\Phi\|_{H^{s+\frac{1}{2},0}_\Psi}\|\lambda_\Phi\|_{H^{s+\frac{1}{2},0}_\Psi}\\
            \leq& C_\eta\left((1+t)^\frac{1}{2}\|u_\Phi\|_{H^{\frac{5}{2}+,0}_\Psi}^2+(1+t)^\frac{1}{2}\|\partial_y^2\theta_\Phi\|_{H^{\frac{3}{2}+,0}_\Psi}^2+\|\partial_yu_\Phi\|_{H^{\frac{3}{2}+,0}_\Psi}^4\right)\|\lambda_\Phi\|_{H^{s+\frac{1}{2},0}_\Psi}^2\\
            &+\eta\|\partial_yu_\Phi\|_{H^{s+\frac{1}{2},0}_\Psi}^2\\
            \leq& C_\eta\dot{\mu}(t)\|\lambda_\Phi\|_{H^{s+\frac{1}{2},0}_\Psi}^2+\eta\|\partial_yu_\Phi\|_{H^{s+\frac{1}{2},0}_\Psi}^2.
        \end{aligned}
    \end{equation}

    (4) \underline{Estimate of $B_7$}.
    By applying Bony's decomposition to $\partial_x\theta\partial_y^2u$, we have
    \begin{equation}
        \begin{aligned}
            B_7=&\langle[T_{\partial_x\theta}\partial_y^2u]_\Phi,\lambda_\Phi\rangle_{H^{s+\frac{1}{2},0}_\Psi}+\langle[T_{\partial_y^2u}\partial_x\theta]_\Phi,\lambda_\Phi\rangle_{H^{s+\frac{1}{2},0}_\Psi}+\langle[R(\partial_x\theta,\partial_y^2u)]_\Phi,\lambda_\Phi\rangle_{H^{s+\frac{1}{2},0}_\Psi}\\
            \triangleq&B_{7,1}+B_{7,2}+B_{7,3}.
        \end{aligned}
    \end{equation}
    Similarly to the estimate of $A_3$ given in Lemma \ref{lem:rightofmathcalU}, by using integration by parts, we have
    \begin{equation}
        \begin{aligned}
            |B_{7,1}|\leq& |\langle[T_{\partial_x\theta}\partial_yu]_\Phi,\partial_y\lambda_\Phi\rangle_{H^{s+\frac{1}{2},0}_\Psi}|+|\langle[T_{\partial_x\partial_y\theta}\partial_yu]_\Phi,\lambda_\Phi\rangle_{H^{s+\frac{1}{2},0}_\Psi}|+2|\langle\partial_y\Psi[T_{\partial_x\theta}\partial_yu]_\Phi,\lambda_\Phi\rangle_{H^{s+\frac{1}{2},0}_\Psi}|\\
            \leq& C\zeta(1+t)^\frac{1}{4}\|\partial_yu_\Phi\|_{H^{s+\frac{1}{2},0}_\Psi}\|\partial_y\lambda_\Phi\|_{H^{s+\frac{1}{2},0}_\Psi}+C(1+t)^\frac{1}{4}\|\partial_y^2\theta_\Phi\|_{H^{\frac{3}{2}+,0}_\Psi}\|\partial_yu_\Phi\|_{H^{s+\frac{1}{2},0}_\Psi}\|\lambda_\Phi\|_{H^{s+\frac{1}{2},0}_\Psi}\\
            &+C(1+t)^\frac{1}{4}\left(\|\partial_y^2\theta_\Phi\|_{H^{\frac{3}{2}+,0}_\Psi}+(1+t)^{-1}\|\theta_\Phi\|_{H^{\frac{3}{2}+,0}_\Psi}\right)\|\partial_yu_\Phi\|_{H^{s+\frac{1}{2},0}_\Psi}\|\lambda_\Phi\|_{H^{s+\frac{1}{2},0}_\Psi}\\
            \leq& C\zeta(1+t)^\frac{1}{4}\|\partial_yu_\Phi\|_{H^{s+\frac{1}{2},0}_\Psi}\|\partial_y\lambda_\Phi\|_{H^{s+\frac{1}{2},0}_\Psi}+\eta\|\partial_yu_\Phi\|_{H^{s+\frac{1}{2},0}_\Psi}^2\\
            &+C_\eta\left((1+t)^\frac{1}{2}\|\partial_y^2\theta_\Phi\|_{H^{\frac{3}{2}+,0}_\Psi}^2+(1+t)^{-\frac{3}{2}}\|\theta_\Phi\|_{H^{\frac{3}{2}+,0}_\Psi}^2\right)\|\lambda_\Phi\|_{H^{s+\frac{1}{2},0}_\Psi}^2\\
            \leq& C\zeta(1+t)^\frac{1}{4}(\|\partial_yu_\Phi\|_{H^{s+\frac{1}{2},0}_\Psi}^2+\|\partial_y\lambda_\Phi\|_{H^{s+\frac{1}{2},0}_\Psi}^2)+C_\eta\dot{\mu}(t)\|\lambda_\Phi\|_{H^{s+\frac{1}{2},0}_\Psi}^2+\eta\|\partial_yu_\Phi\|_{H^{s+\frac{1}{2},0}_\Psi}^2.
        \end{aligned}
    \end{equation}
    Whereas for $B_{7,2}$ and $B_{7,3}$, by using Holder's inequality, we have
    \begin{equation}
        \begin{aligned}
            |B_{7,2}|+|B_{7,3}|\leq& C(1+t)^\frac{1}{4}\|\partial_y^3u_\Phi\|_{H^{\frac{1}{2}+,0}_\Psi}\|\theta_\Phi\|_{H^{s+\frac{5}{4},0}_\Psi}\|\lambda_\Phi\|_{H^{s+\frac{3}{4},0}_\Psi}\\
            \leq& C\dot{\mu}(t)(\|\theta_\Phi\|_{H^{s+\frac{5}{4},0}_\Psi}\|\lambda_\Phi\|_{H^{s+\frac{3}{4},0}_\Psi})\\
            \leq& C\dot{\mu}(t)(\|\theta_\Phi\|_{H^{s+\frac{5}{4},0}_\Psi}^2+\|\lambda_\Phi\|_{H^{s+\frac{3}{4},0}_\Psi}^2).
        \end{aligned}
    \end{equation}
   Summarizing the above estimates of $B_{7,1}, B_{7,2}$ and $B_{7,3}$, it follows
    \begin{equation}
        \begin{aligned}
        |B_7|\leq& C\zeta(1+t)^\frac{1}{4}(\|\partial_yu_\Phi\|_{H^{s+\frac{1}{2},0}_\Psi}^2+\|\partial_y\lambda_\Phi\|_{H^{s+\frac{1}{2},0}_\Psi}^2)\\
        &+C_\eta\dot{\mu}(t)\|\lambda_\Phi\|_{H^{s+\frac{3}{4},0}_\Psi}^2+\eta\|\partial_yu_\Phi\|_{H^{s+\frac{1}{2},0}_\Psi}^2+C\dot{\mu}(t)\|\theta_\Phi\|_{H^{s+\frac{5}{4},0}_\Psi}^2.
        \end{aligned}
    \end{equation}

    (5) \underline{Estimates of $B_8$ and $B_9$}.
    We deduce from Lemmas \ref{lem:paraestimate} and \ref{lem:Poincare} that
    \begin{equation}
        \begin{aligned}
            |B_8|+|B_9|\leq& C(1+t)^\frac{1}{4}\|u_\Phi\|_{H^{\frac{5}{2}+,0}_\Psi}\|u_\Phi\|_{H^{s+\frac{1}{2},0}_\Psi}\|\lambda_\Phi\|_{H^{s+\frac{1}{2},0}_\Psi}\\
            \leq& C\dot{\mu}(t)\|u_\Phi\|_{H^{s+\frac{1}{2},0}_\Psi}\|\lambda_\Phi\|_{H^{s+\frac{1}{2},0}_\Psi}\\
            \leq& C\dot{\mu}(t)(\|u_\Phi\|_{H^{s+\frac{1}{2},0}_\Psi}^2+\|\lambda_\Phi\|_{H^{s+\frac{1}{2},0}_\Psi}^2).
        \end{aligned}
    \end{equation}

    (6) \underline{Estimates of $B_{10}$ and $B_{11}$}.
    A direct computation gives
    \begin{equation}
        \begin{aligned}
            B_{10}=&-\langle \left[T_{\partial_x\partial_yu}\int_0^y\partial_xud\tilde{y}\right]_\Phi,\lambda_\Phi\rangle_{H^{s+\frac{1}{2},0}_\Psi}+\langle [T_{\partial_x\partial_yu}\partial_y\theta]_\Phi,\lambda_\Phi\rangle_{H^{s+\frac{1}{2},0}_\Psi}\\
            &+\langle \left[T_{\partial_x\partial_yu}\int_0^y(\partial_yu)^2d\tilde{y}\right]_\Phi,\lambda_\Phi\rangle_{H^{s+\frac{1}{2},0}_\Psi}\\
            \triangleq&B_{10,1}+B_{10,2}+B_{10,3}.
        \end{aligned}
    \end{equation}
    By applying Lemmas \ref{lem:paraestimate} and \ref{lem:Poincare}, we obtain
    \begin{equation}
        \begin{aligned}
            |B_{10,1}|\leq& C(1+t)^\frac{1}{4}\|\partial_yu_\Phi\|_{H^{\frac{3}{2}+,0}_\Psi}\|u_\Phi\|_{H^{s+\frac{5}{4},0}_\Psi}\|\lambda_\Phi\|_{H^{s+\frac{3}{4},0}_\Psi}\\
            \leq& C\dot{\mu}(t)\|u_\Phi\|_{H^{s+\frac{5}{4},0}_\Psi}\|\lambda_\Phi\|_{H^{s+\frac{3}{4},0}_\Psi}\\
            \leq& C\dot{\mu}(t)(\|u_\Phi\|_{H^{s+\frac{5}{4},0}_\Psi}^2+\|\lambda_\Phi\|_{H^{s+\frac{3}{4},0}_\Psi}^2)
        \end{aligned}
    \end{equation}
    and
    \begin{equation}
        \begin{aligned}
            |B_{10,2}|\leq& C(1+t)^\frac{1}{4}\|\partial_y^2u_\Phi\|_{H^{\frac{3}{2}+,0}_\Psi}\|\partial_y\theta_\Phi\|_{H^{s+\frac{1}{2},0}_\Psi}\|\lambda_\Phi\|_{H^{s+\frac{1}{2},0}_\Psi}\\
            \leq& C_\eta(1+t)^\frac{1}{2}\|\partial_y^2u_\Phi\|_{H^{\frac{3}{2}+,0}_\Psi}^2\|\lambda_\Phi\|_{H^{s+\frac{1}{2},0}_\Psi}^2+\eta\|\partial_y\theta_\Phi\|_{H^{s+\frac{1}{2},0}_\Psi}^2\\
            \leq& C_\eta\dot{\mu}(t)\|\lambda_\Phi\|_{H^{s+\frac{1}{2},0}_\Psi}^2+\eta\|\partial_y\theta_\Phi\|_{H^{s+\frac{1}{2},0}_\Psi}^2.
        \end{aligned}
    \end{equation}
    By using Lemma \ref{lem:paraestimate} and Minkowski's inequality, it follows that
    \begin{equation}
        \begin{aligned}
            |B_{10,3}|\leq& C\|\partial_yu_\Phi\|_{H^{\frac{3}{2}+,0}_\Psi}\|\int_0^y[(\partial_yu)^2]_\Phi d\tilde{y}\|_{L_y^\infty(H_x^{s+\frac{1}{2}})}\|\lambda_\Phi\|_{H^{s+\frac{1}{2},0}_\Psi}\\
            \leq& C\|\partial_yu_\Phi\|_{H^{\frac{3}{2}+,0}_\Psi}\int_0^\infty\|[(\partial_yu)^2]_\Phi\|_{H_x^{s+\frac{1}{2}}}d\tilde{y}\|\lambda_\Phi\|_{H^{s+\frac{1}{2},0}_\Psi}.
        \end{aligned}
    \end{equation}
    Applying Bony's decomposition to $(\partial_yu)^2$ and Holder's inequality, it gives
    \begin{equation}
        \begin{aligned}
            |B_{10,3}|\leq& C\|\partial_yu_\Phi\|_{H^{\frac{3}{2}+,0}_\Psi}\int_0^\infty\|[T_{\partial_yu}\partial_yu]_\Phi\|_{H_x^{s+\frac{1}{2}}}d\tilde{y}\|\lambda_\Phi\|_{H^{s+\frac{1}{2},0}_\Psi}\\
            &+C\|\partial_yu_\Phi\|_{H^{\frac{3}{2}+,0}_\Psi}\int_0^\infty\|[R(\partial_yu,\partial_yu)]_\Phi\|_{H_x^{s+\frac{1}{2}}}d\tilde{y}\|\lambda_\Phi\|_{H^{s+\frac{1}{2},0}_\Psi}\\
            \leq& C\|\partial_yu_\Phi\|_{H^{\frac{3}{2}+,0}_\Psi}^2\|\partial_yu_\Phi\|_{H^{s+\frac{1}{2},0}_\Psi}\|\lambda_\Phi\|_{H^{s+\frac{1}{2},0}_\Psi}\\
            \leq& C_\eta\|\partial_yu_\Phi\|_{H^{\frac{3}{2}+,0}_\Psi}^4\|\lambda_\Phi\|_{H^{s+\frac{1}{2},0}_\Psi}^2+\eta\|\partial_yu_\Phi\|_{H^{s+\frac{1}{2},0}_\Psi}^2\\
            \leq& C_\eta\dot{\mu}(t)\|\lambda_\Phi\|_{H^{s+\frac{1}{2},0}_\Psi}^2+\eta\|\partial_yu_\Phi\|_{H^{s+\frac{1}{2},0}_\Psi}^2.
        \end{aligned}
    \end{equation}
    One can study $B_{11}$ in a way similar to the above for $B_{10}$, and concludes that 
    \begin{equation}
        |B_{10}|+|B_{11}|\leq C_\eta\dot{\mu}(t)\|\lambda_\Phi\|_{H^{s+\frac{3}{4},0}_\Psi}^2+C\dot{\mu}(t)\|u_\Phi\|_{H^{s+\frac{5}{4},0}_\Psi}^2+\eta(\|\partial_yu_\Phi\|_{H^{s+\frac{1}{2},0}_\Psi}^2+\|\partial_y\theta_\Phi\|_{H^{s+\frac{1}{2},0}_\Psi}^2).
    \end{equation}

    (7) \underline{Estimates of $B_{12}$ and $B_{13}$}.
    In the same way as the estimates of $B_8,B_9$, we have
    \begin{equation}
        \begin{aligned}
            |B_{12}|+|B_{13}|\leq& C(1+t)^\frac{1}{4}\|\partial_y^3u_\Phi\|_{H^{\frac{3}{2}+,0}_\Psi}\|\theta_\Phi\|_{H^{s+\frac{1}{2},0}_\Psi}\|\lambda_\Phi\|_{H^{s+\frac{1}{2},0}_\Psi}\\
            \leq& C\dot{\mu}(t)\|\theta_\Phi\|_{H^{s+\frac{1}{2},0}_\Psi}\|\lambda_\Phi\|_{H^{s+\frac{1}{2},0}_\Psi}\\
            \leq& C\dot{\mu}(t)(\|\theta_\Phi\|_{H^{s+\frac{1}{2},0}_\Psi}^2+\|\lambda_\Phi\|_{H^{s+\frac{1}{2},0}_\Psi}^2).
        \end{aligned}
    \end{equation}

    (8) \underline{Estimate of $B_{14}$}.
    From the third equation of \eqref{eq:main}, we get
    \begin{equation}
        \begin{aligned}
            B_{14}=&-\langle \left[T_{(-\partial_yu\partial_xu-\partial_yv\partial_yu+\partial_y\theta\partial_y^2u)}\int_0^y\mathcal{U} d\tilde{y}\right]_\Phi,\lambda_\Phi\rangle_{H^{s+\frac{1}{2},0}_\Psi}\\
            =&\langle \left[T_{(\partial_yu\partial_y^2\theta+(\partial_yu)^3-\partial_y\theta\partial_y^2u)}\int_0^y\mathcal{U} d\tilde{y}\right]_\Phi,\lambda_\Phi\rangle_{H^{s+\frac{1}{2},0}_\Psi}.
        \end{aligned}
    \end{equation}
    From Lemmas \ref{lem:paraestimate}, \ref{lem:Poincare} and the inequality \eqref{ineq:nonlinear}, we obtain that
    \begin{equation}
        \begin{aligned}
            |B_{14}|\leq& C(1+t)^\frac{1}{4}\left(\|[\partial_yu\partial_y^2\theta]_\Phi\|_{H^{\frac{1}{2}+,0}_\Psi}+\|[(\partial_yu)^3]_\Phi\|_{H^{\frac{1}{2}+,0}_\Psi}+\|[\partial_y\theta\partial_y^2u]_\Phi\|_{H^{\frac{1}{2}+,0}_\Psi}\right)\|\mathcal{U}_\Phi\|_{H^{s+\frac{1}{4},0}_\Psi}\|\lambda_\Phi\|_{H^{s+\frac{3}{4},0}_\Psi}\\
            \leq& C\left((1+t)^\frac{1}{2}\|\partial_y^2u_\Phi\|_{H^{\frac{1}{2}+,0}_\Psi}\|\partial_y^2\theta_\Phi\|_{H^{\frac{1}{2}+,0}_\Psi}+(1+t)^\frac{3}{4}\|\partial_y^2u_\Phi\|_{H^{\frac{1}{2}+,0}_\Psi}^2\|\partial_yu_\Phi\|_{H^{\frac{1}{2}+,0}_\Psi}\right)\\
            &\times\|\mathcal{U}_\Phi\|_{H^{s+\frac{1}{4},0}_\Psi}\|\lambda_\Phi\|_{H^{s+\frac{3}{4},0}_\Psi}\\
            \leq& C\dot{\mu}(t)\|\mathcal{U}_\Phi\|_{H^{s+\frac{1}{4},0}_\Psi}\|\lambda_\Phi\|_{H^{s+\frac{3}{4},0}_\Psi}\\
            \leq& C\dot{\mu}(t)(\|\mathcal{U}_\Phi\|_{H^{s+\frac{1}{4},0}_\Psi}^2+\|\lambda_\Phi\|_{H^{s+\frac{3}{4},0}_\Psi}^2).
        \end{aligned}
    \end{equation}

    (9) \underline{Estimate of $B_{15}$}.
    Decompose $B_{15}$ into
    \begin{equation}
        \begin{aligned}
            B_{15}=&-\langle \left[(T_uT_{\partial_x\partial_yu}-T_{u\partial_x\partial_yu})\int_0^y\mathcal{U}d\tilde{y}\right]_\Phi,\lambda_\Phi\rangle_{H^{s+\frac{1}{2},0}_\Psi}-\langle \left[(T_vT_{\partial_y^2u}-T_{v\partial_y^2u})\int_0^y\mathcal{U}d\tilde{y}\right]_\Phi,\lambda_\Phi\rangle_{H^{s+\frac{1}{2},0}_\Psi}\\
            &+\langle \left[(T_\theta T_{\partial^3_yu}-T_{\theta\partial^3_yu})\int_0^y\mathcal{U}d\tilde{y}\right]_\Phi,\lambda_\Phi\rangle_{H^{s+\frac{1}{2},0}_\Psi}\\
            \triangleq&B_{15,1}+B_{15,2}+B_{15,3}.
        \end{aligned}
    \end{equation}
    It follows from Lemmas \ref{lem:commutator} and \ref{lem:Poincare} that
    \begin{equation}
        \begin{aligned}
            |B_{15,1}|+|B_{15,3}|\leq& C(1+t)^\frac{1}{2}\left(\|\partial_yu_\Phi\|_{H^{\frac{5}{2}+,0}_\Psi}^2+\|\partial_y\theta_\Phi\|_{H^{\frac{3}{2}+,0}_\Psi}\|\partial_y^3u_\Phi\|_{H^{\frac{3}{2}+,0}_\Psi}\right)\|\mathcal{U}_\Phi\|_{H^{s-\frac{1}{2},0}_\Psi}\|\lambda_\Phi\|_{H^{s+\frac{1}{2},0}_\Psi}\\
            \leq& C\dot{\mu}(t)\|\mathcal{U}_\Phi\|_{H^{s-\frac{1}{2},0}_\Psi}\|\lambda_\Phi\|_{H^{s+\frac{1}{2},0}_\Psi}\\
            \leq& C\dot{\mu}(t)(\|\mathcal{U}_\Phi\|_{H^{s-\frac{1}{2},0}_\Psi}^2+\|\lambda_\Phi\|_{H^{s+\frac{1}{2},0}_\Psi}^2).
        \end{aligned}
    \end{equation}
In a way similar to the study of $A_2$ given in Lemma \ref{lem:rightofmathcalU},    one can estimate $B_{15,2}$ as:
    \begin{equation}\label{a21}
        \begin{aligned}
            |B_{15,2}|\leq& C(1+t)^\frac{1}{4}\|v_\Phi\|_{L_y^\infty(H_x^{\frac{3}{2}+})}\|\partial_y^2u_\Phi\|_{H^{\frac{3}{2}+,0}_\Psi}\|\mathcal{U}_\Phi\|_{H^{s-\frac{1}{2},0}_\Psi}\|\lambda_\Phi\|_{H^{s+\frac{1}{2},0}_\Psi}\\
            \leq& C\bigg((1+t)^\frac{1}{2}\|u_\Phi\|_{H^{\frac{5}{2}+,0}_\Psi}\|\partial_y^2u_\Phi\|_{H^{\frac{3}{2}+,0}_\Psi}+(1+t)^\frac{1}{2}\|\partial_y^2\theta_\Phi\|_{H^{\frac{3}{2}+,0}_\Psi}\|\partial_y^2u_\Phi\|_{H^{\frac{3}{2}+,0}_\Psi}\\
            &+(1+t)^\frac{1}{4}\|\partial_yu_\Phi\|_{H^{\frac{3}{2}+,0}_\Psi}^2\|\partial_y^2u_\Phi\|_{H^{\frac{3}{2}+,0}_\Psi}\bigg)\|\mathcal{U}_\Phi\|_{H^{s-\frac{1}{2},0}_\Psi}\|\lambda_\Phi\|_{H^{s+\frac{1}{2},0}_\Psi}\\
            \leq& C\dot{\mu}(t)\|\mathcal{U}_\Phi\|_{H^{s-\frac{1}{2},0}_\Psi}\|\lambda_\Phi\|_{H^{s+\frac{1}{2},0}_\Psi}\\
            \leq& C\dot{\mu}(t)(\|\mathcal{U}_\Phi\|_{H^{s-\frac{1}{2},0}_\Psi}^2+\|\lambda_\Phi\|_{H^{s+\frac{1}{2},0}_\Psi}^2).
        \end{aligned}
    \end{equation}
    Summarizing the above estimates, we get
    \begin{equation}
        |B_{15}|\leq C\dot{\mu}(t)(\|\mathcal{U}_\Phi\|_{H^{s-\frac{1}{2},0}_\Psi}^2+\|\lambda_\Phi\|_{H^{s+\frac{1}{2},0}_\Psi}^2).
    \end{equation}

    (10) \underline{Estimate of $B_{16}$}.
    We deduce from Lemmas \ref{lem:commutator} and \ref{lem:Poincare} that
    \begin{equation}
        \begin{aligned}
            |B_{16}|\leq& C\|u_\Phi\|_{L_y^\infty(H_x^{\frac{3}{2}+})}\|\partial_yu_\Phi\|_{H^{\frac{3}{2}+,0}_\Psi}\|\int_0^y\mathcal{U}_\Phi d\tilde{y}\|_{L_y^\infty(H_x^{s+\frac{1}{4}})}\|\lambda_\Phi\|_{H^{s+\frac{3}{4},0}_\Psi}\\
            \leq& C(1+t)^\frac{1}{2}\|\partial_yu_\Phi\|_{H^{\frac{3}{2}+,0}_\Psi}^2\|\mathcal{U}_\Phi\|_{H^{s+\frac{1}{4},0}_\Psi}\|\lambda_\Phi\|_{H^{s+\frac{3}{4},0}_\Psi}\\
            \leq& C\dot{\mu}(t)\|\mathcal{U}_\Phi\|_{H^{s+\frac{1}{4},0}_\Psi}\|\lambda_\Phi\|_{H^{s+\frac{3}{4},0}_\Psi}\\
            \leq& C\dot{\mu}(t)(\|\mathcal{U}_\Phi\|_{H^{s+\frac{1}{4},0}_\Psi}^2+\|\lambda_\Phi\|_{H^{s+\frac{3}{4},0}_\Psi}^2).
        \end{aligned}
    \end{equation}

    (11) \underline{Estimates of $B_{17}$ and $B_{18}$}.
    Similarly to the estimate of $B_{15,2}$ given in \eqref{a21}, we have
    \begin{equation}
        \begin{aligned}
            |B_{17}|\leq& C(1+t)^\frac{1}{4}\|v_\Phi\|_{L_y^\infty(H_x^{\frac{3}{2}+})}\|\partial_y^2u_\Phi\|_{H^{\frac{3}{2}+,0}_\Psi}\|\mathcal{U}_\Phi\|_{H^{s-\frac{1}{2},0}_\Psi}\|\lambda_\Phi\|_{H^{s+\frac{1}{2},0}_\Psi}\\
            \leq& C\dot{\mu}(t)(\|\mathcal{U}_\Phi\|_{H^{s-\frac{1}{2},0}_\Psi}^2+\|\lambda_\Phi\|_{H^{s+\frac{1}{2},0}_\Psi}^2),
        \end{aligned}
    \end{equation}
   and similar to the above estimate of $B_{16}$, one has
    \begin{equation}
        \begin{aligned}
            |B_{18}|\leq& C(1+t)^\frac{1}{2}\|\partial_y\theta_\Phi\|_{H^{\frac{3}{2}+,0}_\Psi}\|\partial_y^2u_\Phi\|_{H^{\frac{3}{2}+,0}_\Psi}\|\partial_y\mathcal{U}_\Phi\|_{H^{s-\frac{1}{2},0}_\Psi}\|\lambda_\Phi\|_{H^{s+\frac{1}{2},0}_\Psi}\\
            \leq& C_\eta(1+t)\|\partial_y\theta_\Phi\|_{H^{\frac{3}{2}+,0}_\Psi}^2\|\partial_y^2u_\Phi\|_{H^{\frac{3}{2}+,0}_\Psi}^2\|\lambda_\Phi\|_{H^{s+\frac{1}{2},0}_\Psi}^2+\eta\|\partial_y\mathcal{U}_\Phi\|_{H^{s-\frac{1}{2},0}_\Psi}^2\\
            \leq& C_\eta\dot{\mu}(t)\|\lambda_\Phi\|_{H^{s+\frac{1}{2},0}_\Psi}^2+\eta\|\partial_y\mathcal{U}_\Phi\|_{H^{s-\frac{1}{2},0}_\Psi}^2.
        \end{aligned}
    \end{equation}

    (12) \underline{Estimate of $B_{19}$}.
    By using Lemmas \ref{lem:paraestimate} and \ref{lem:Poincare} twice, we obtain
    \begin{equation}
        \begin{aligned}
            |B_{19}|\leq& C\left((1+t)^\frac{1}{4}+(1+t)^\frac{1}{2}\|\partial_y\theta_\Phi\|_{H^{\frac{1}{2}+,0}_\Psi}\right)\|\partial_y^3u_\Phi\|_{H^{\frac{1}{2}+,0}_\Psi}\|\mathcal{U}_\Phi\|_{H^{s+\frac{1}{4},0}_\Psi}\|\lambda_\Phi\|_{H^{s+\frac{3}{4},0}_\Psi}\\
            \leq& C\dot{\mu}(t)\|\mathcal{U}_\Phi\|_{H^{s+\frac{1}{4},0}_\Psi}\|\lambda_\Phi\|_{H^{s+\frac{3}{4},0}_\Psi}\\
            \leq& C\dot{\mu}(t)(\|\mathcal{U}_\Phi\|_{H^{s+\frac{1}{4},0}_\Psi}^2+\|\lambda_\Phi\|_{H^{s+\frac{3}{4},0}_\Psi}^2).
        \end{aligned}
    \end{equation}
    
    Combining the estimates of $B_1,\cdots,B_{19}$ given above, we complete the proof of Lemma \ref{lem:rightoflambda} immediately.
\end{proof}

\end{appendices}

	\vspace{.1in}
	\par{\bf Acknowledgements.} This research was supported by National Key R\&D Program of China under grant 2024YFA1013302, and National Natural Science Foundation of China under grant 12331008.

\bibliographystyle{plain}
\bibliography{ref}

\end{document}